\documentclass[10pt]{article}
\usepackage[utf8]{inputenc}
\usepackage[T1]{fontenc}
\usepackage{cancel}
\usepackage{thmtools}
\usepackage{enumitem}
\usepackage{relsize} 
\usepackage{fancyhdr}
\usepackage{stmaryrd}
\SetSymbolFont{stmry}{bold}{U}{stmry}{m}{n}
\usepackage{graphicx}
\usepackage[export]{adjustbox}
\usepackage{afterpage}

\usepackage[toc,page]{appendix}
\usepackage{lipsum}
\usepackage{etoolbox}
\appto\appendix{\addtocontents{toc}{\protect\setcounter{tocdepth}{2}}}
\usepackage{biblatex} 
\DeclareFieldFormat{eprint:arXiv}{%
  \mkbibacro{arXiv}\addcolon%
  \ifhyperref{\href{https://arxiv.org/abs/#1}{#1}}{#1}%
  \iffieldundef{eprintclass}{}{\addspace\mkbibbrackets{\thefield{eprintclass}}}%
}
\DeclareFieldFormat{eprint:arxiv}{%
  \mkbibacro{arXiv}\addcolon%
  \ifhyperref{\href{https://arxiv.org/abs/#1}{#1}}{#1}%
  \iffieldundef{eprintclass}{}{\addspace\mkbibbrackets{\thefield{eprintclass}}}%
}

\appto\listoffigures{\addtocontents{lof}{\protect\setcounter{tocdepth}{1}}}
\appto\listoftables{\addtocontents{lot}{\protect\setcounter{tocdepth}{1}}}

\graphicspath{ {./images/} }
\usepackage{amssymb,amsmath,amsfonts,mathtools,mathrsfs} 
\usepackage[version=4]{mhchem}

\usepackage[hidelinks]{hyperref}
\hypersetup{
  pdftitle={The Geometry of Rough Path Space},
  pdfauthor={Martin Geller and Terry Lyons},
  pdfsubject={Rough path space},
  pdfkeywords={rough paths, rough path space, p-variation, perturbations}
}
\usepackage{tcolorbox}\newlength{\saveparindent}
\usepackage{bbold}
\usepackage{tikz}
\usepackage{tikz-qtree}
\usepackage{float}
\usepackage{soul}
\usepackage{comment}

\allowdisplaybreaks

\numberwithin{equation}{section}
\newcommand{\vertspace}{\vspace{4pt}}

\newtheorem{theorem}{Theorem}[section]
\newtheorem{definition}[theorem]{Definition}
\newtheorem{example}[theorem]{Example}
\newtheorem{lemma}[theorem]{Lemma}
\newtheorem{proposition}[theorem]{Proposition}
\newtheorem{remark}[theorem]{Remark}

\newtheorem{paragraph extract}{Paragraph Extract}
\newtheorem{corollary}[theorem]{Corollary}

\newenvironment{lemma*}
{\expandafter\def\expandafter\thelemma\expandafter{\thelemma*}\lemma}
{\endlemma}
\newenvironment{proposition*}
{\expandafter\def\expandafter\theproposition\expandafter{\theproposition*}\proposition}
{\endproposition}
\newenvironment{theorem*}
{\expandafter\def\expandafter\thetheorem\expandafter{\thetheorem*}\theorem}
{\endtheorem}

\makeatletter
\DeclareRobustCommand{\qed}{%
	\ifmmode 
	\else \leavevmode\unskip\penalty9999 \hbox{}\nobreak\hfill
	\fi
	\quad\hbox{\qedsymbol}}
\newcommand{\openbox}{\leavevmode
	\hbox to.77778em{%
		\hfil\vrule
		\vbox to.675em{\hrule width.6em\vfil\hrule}%
		\vrule\hfil}}
\newcommand{\qedsymbol}{\openbox}
\newenvironment{proof}[1][\proofname]{\par
	\normalfont
	\topsep6\p@\@plus6\p@ \trivlist
	\item[\hskip\labelsep\itshape
	#1.]\ignorespaces
}{%
	\qed\endtrivlist
}
\newcommand{\proofname}{Proof}
\makeatother

\makeatletter
\newcommand{\eqnum}{\refstepcounter{equation}\textup{\tagform@{\theequation}}}
\makeatother

\newcommand{\mathcolorbox}[2]{\colorbox{#1}{$\displaystyle #2$}}

\newlength{\EqIndent}           
\newcommand{\EqLine}[2]{%
  \par\noindent
  \hspace*{\EqIndent}%
  $\displaystyle (#1)\quad #2$\par
}

\title{The Geometry of Rough Path Space}
\date{January 2026}
\author{   Martin Geller\thanks{Mathematical Institute, University of Oxford, Oxford, United Kingdom. \texttt{martin.geller@maths.ox.ac.uk}}   \and   Terry Lyons\thanks{Mathematical Institute, University of Oxford, Oxford, United Kingdom. \texttt{terry.lyons@maths.ox.ac.uk}} }

\begin{document}
	\setlength{\saveparindent}{\parindent}
	\maketitle
	
	\begin{abstract} 
    \noindent 
We describe $\mathscr H^p(V)$, a subset of $p$-rough path space $\Omega^p(V)$ which is a vector space under an addition operation $\boxplus$ and a scalar multiplication $\odot$. We show that the domain of $\boxplus$ can be extended to $\Omega^p(V) \times \mathscr{H}^p(V)$, allowing any $p$-rough path $X$ to be additively perturbed by an $H$ in $\mathscr H^p(V)$. We show that the extended $\boxplus$ operation is well-behaved, proving associativity $(X\boxplus H) \boxplus \widetilde H = X\boxplus (H \boxplus \widetilde H)$ and $X\boxplus H = X\iff H=\mathbb{1}$, where $\mathbb{1}$ is the additive zero in $(\mathscr H^p(V),\boxplus, \odot)$. Additionally, we show that expanding the space of perturbations $\mathscr H^p(V)$ to include almost rough paths  $\mathscr H^{\text{am},p}(V)$ does not enlarge the space of displacements of a given $X$, i.e. $\{X\boxplus H | H \in \mathscr H^p(V)\}$ $=\{X\boxplus H | H \in\mathscr H^{\text{am},p}(V)\}$.
    
    \vspace{\baselineskip}
	\end{abstract}
	
\tableofcontents

    \section*{Notation}
\addcontentsline{toc}{section}{Notation}

    \begin{itemize}
        \item $\triangle_J$: given an interval $J=[S,T]$, the set of pairs $(s,t) \in J$ s.t. $s\le t$
        \item $T((V))$: space of formal tensor series over the Banach space $V$
        \item $T(V)$: free tensor algebra over $V$
        \item $T^{(n)}(V)$: for $k\in\mathbb N$, the free tensor algebra over $V$, truncated at level $k$
        \item $\Omega^{(k,p)}(V)$: \scalebox{0.95}[1.0]{for $p\ge1$, space of multiplicative $X:\triangle_J\to T^{(k)}(V) $ of finite $p$-var.}
        \item $\Omega^p(V)$: space of $p$-rough paths $X:\triangle_J\to T^{(\lfloor p \rfloor )}(V) $.
        \item $WG\Omega^p(V)$: space of weakly geometric $p$-rough paths.
        \item $\Omega_{\omega,\theta}^{\text{am},p}(V)$: space of $\theta$-almost $p$-rough paths controlled by $\omega$
        \item $\stackrel{\text{def.}}{=}$ : "by definition"
    \end{itemize}

\newpage
\section{Introduction}
\label{sec:introduction}

\subsection{Background}
This introduction assumes familiarity with the material used later in the paper, and in particular with the first three sections of \cite{stflour}. Some points will be clearer when the introduction is revisited after reading the rest of the paper. Its purpose is to provide a conceptual roadmap and a brief literature review of the problem we address.

Nevertheless, to keep the discussion as self-contained as possible, we now fix notation and recall the main objects used below; precise definitions and further details are collected in Section~1. We fix an interval $J=[S,T]$ with $S\ge0$ and consider the simplex $\triangle_J:=\{(s,t)\in J^2|s\le t\}$. We define a control function $\omega:\triangle_J\to[0,\infty)$ to be a jointly continuous function such that it is super-additive (i.e. $\omega(s,u)+\omega(u,t)\le\omega(s,t)$ for all $s\le u\le t$ with $s,u,t\in J$) and $\omega(s,s) =0$ for all $s\in J$ (see Definition \ref{def:controls}. For a Banach space \(V\) we write \(T((V))\) for the space of formal tensor series,
\[
T((V)) := \{(a_0,a_1,\ldots) : a_n\in V^{\otimes n}\ \text{for all }n\ge 0\},
\]
endowed with the product $\otimes$ given degree-wise by $(a\otimes \tilde a)^k = \sum^k_{j=0} a^{j}\otimes \tilde a^{k-j} $; this is a unital (non-commutative) algebra with unit \(\mathbb 1=(1,0,0,\ldots)\). We assume norms on the tensor powers of $V$ are \textit{admissible} in the sense that $\|v_1\otimes\ldots\otimes v_k\|\le\| v_1\|\ldots\| v_k\|$ for any pure tensor  $v_1\otimes\ldots\otimes v_k\in V^{\otimes k}$, as well as a condition on the permutations of $v_1\otimes\ldots\otimes v_k$ (see Definition \ref{def: admissible}). For each integer \(n\ge 1\) we denote by \(T^{(n)}(V)\) the truncated tensor algebra of order \(n\), defined as the quotient \(T((V))/B_n\) by the two-sided ideal \(B_n:=\{a=(a_k)_k : a_0=\dots=a_n=0\}\); \(T^{(n)}(V)\) is again a unital algebra, canonically isomorphic to \(\bigoplus_{k=0}^n V^{\otimes k}\) with the induced tensor product (Definitions~\ref{def:tensor-series} and~\ref{def:trunc-tensor}). 

Let \(X=(X^0,\ldots,X^{\lfloor p \rfloor})\colon\triangle_J\to T^{(\lfloor p\rfloor)}(V)\) be a functional and define the multiplicative defect of $X$, which we call $\Delta(X)$, through
\[\Delta(X)_{s,u,t}:=X_{s,u}\otimes X_{u,t}- X_{s,t}, \qquad \forall s\le u \le t\]
We would then say $X$ is a multiplicative functional of degree \(\lfloor p\rfloor\) if it is continuous, if \(X^0_{s,t}=1\) and  if \(\Delta(X)_{s,u,t}=0, \qquad \forall s\le u \le t,\)(see Definition~\ref{def: multiplicative functional}). Given \(p\ge 1\), we say that \(X\) has \emph{finite \(p\)-variation controlled by \(\omega\)} if
\[
\|X^j_{s,t}\|\;\le\; \frac{\omega(s,t)^{j/p}}{\beta\cdot(j/p)!}\qquad j=1,\dots,\lfloor p \rfloor;\ (s,t)\in\triangle_J,
\]
for the constant \(\beta=\beta(p)\) of Definition~\ref{def:beta} (see Definition~\ref{def:finite p-var}). Given  $n\in\mathbb N$,  $X:\triangle_J\to T^{(n)}(V)$ is a functional, $\omega$ is a control, and $\theta>1$ is a constant, then we say that $X$ is $\theta$-almost multiplicative controlled by $\omega$ if $$\|\Delta(X)_{s,u,t}\|\leq \omega(s,t)^\theta \qquad \forall s\leq u \leq t;$$
in other words, $X$ is almost multiplicative if the multiplicative defect of $X$ is bounded supra-linearly in $\omega$. 

For \(p\ge 1\) and an integer \(k\ge 1\) we then write \(\Omega^{(k,p)}(V)\) for the space of \(k\)-step multiplicative functionals of finite \(p\)-variation with values in \(T^{(k)}(V)\), and set
\[
\Omega^p(V):=\Omega^{(\lfloor p\rfloor,p)}(V),
\]
whose elements are called \(p\)-rough paths (see Definition\ref{def:rough-path-space}). Its weakly geometric subspace \(WG\Omega^p(V)\subset \Omega^{p}(V)\) consists of those rough paths whose values lie in the truncated group of group-like elements \(G^{(\lfloor p\rfloor)}(V)\subset T^{(\lfloor p\rfloor)}(V)\) (Definition~\ref{def: weakly geometric RPs}). More generally, we call \(X:\triangle_J\to T^{(\lfloor p\rfloor)}(V)\) a \(\theta\)-almost \(p\)-rough path if it is \(\theta\)-almost multiplicative and has finite \(p\)-variation (Definitions~\ref{def: almost multiplicativity}–\ref{def: Omega^(am, (k,p))_omega,theta}). 

Crucially, we rely on Lyons' \textit{rough sewing lemma} (Theorem~\ref{thm:Rough-Sewing-Lemma}), which was first introduced in 1998 \cite{lyons1998differential}, and claims the following: for each almost multiplicative functional $X$, there exists a unique multiplicative functional $\bar X$ which is close to $X$ in the sense that
\begin{equation}
		\sup_{\substack{S\leqslant s < t \leqslant T \\
				i=0, \ldots, \lfloor p \rfloor}} 
		\frac{\left\| \bar{X}_{s,t}^i - X_{s,t}^i \right\|}{\omega(s,t)^\theta} < +\infty.
	\end{equation}
We define the \emph{sewing map} \(\mathscr S\) to be the map that assigns \(X\) to its unique rough path \(\mathscr S(X)\in\Omega^p(V)\) as per the rough sewing lemma (Definition~\ref{def: sewing map mathscr S}). 

The central new object of this paper is a distinguished subset \(\mathscr{H}^p(V)\subset\Omega^p(V)\), which we call \emph{\(\mathscr H\)-space} (of degree $p$): we say \(H\in \mathscr{H}^p(V)\) if \(H\in\Omega^p(V)\) and there exist both a control \(\omega\) and an exponent \(\phi\in(1-1/p,1]\) such that, for every level \(j\in\{1,\dots,\lfloor p\rfloor\}\) and all \((s,t)\in\triangle_J\),
\begin{equation}\label{ineq: H-space, intro}
    \|H^j_{s,t}\|\;\le\; K \omega(s,t)^{\phi},
\end{equation}
with $K\ge0$ being a constant. 


\subsection{Objectives and motivation of this paper}

The basic objective of this paper is to introduce a vector space of perturbations inside rough path space by identifying a class of genuine rough paths $\mathscr H^p(V)\subset\Omega^p(V)$ that acts as a space of additive perturbations. We define the corresponding perturbation map $(X,H)\mapsto X\boxplus H$ canonically via the rough sewing lemma and study its algebraic properties. We briefly recall the role of $\Omega^p(V)$ in Lyons’ theory and place our construction in context. Given a Banach space \(V\) and \(p\ge 1\), we consider \(\Omega^p(V)\), i.e. the space of \(p\)-rough paths over \(V\); this was introduced in \cite{lyons1998differential}, where it serves to develop a solution theory of controlled differential equations that extends classical It\^o calculus to a much broader class of stochastic integrals (see \cite{lyons1998differential,stflour}). The structure of the subset \(WG\Omega^p(\mathbb{R}^d)\subset \Omega^p(V)\) of weakly geometric $p$-rough paths has also been extensively studied: when \(V\) is finite dimensional, weakly geometric \(p\)-rough paths correspond to continuous paths of finite \(p\)-variation in the step-\(\lfloor p\rfloor\) free nilpotent group \(G^{\lfloor p\rfloor}(V)\), which is a Carnot group endowed with a Carnot--Carath\'eodory metric \cite{Friz_Victoir_2010}. From an analytic viewpoint, for a rough path $X:\triangle_J \to T^{(\lfloor p \rfloor)}(V)$, Lyons' extension theorem shows how there exists a unique extension as a multiplicative functional $X:\triangle_J\to T((V))$ of finite $p$-variation (see \cite{Friz_Victoir_2010,lyons1998differential}). From a geometric viewpoint, however, \(\Omega^p(V)\) and $WG\Omega^p(V)$ are not well-understood; in particular, there is no canonical notion of additive perturbation of rough paths beyond bounded variation perturbations, or of a tangent space at a given point \cite{Friz_Victoir_2010}.

A basic manifestation of this is that, for any $p\ge1$, the space of $p$-rough paths $\Omega^p(V)$ is not a vector space under the pointwise addition operation $+$ inherited from the ambient space $C_0(\triangle_J, T^{(\lfloor p \rfloor)}(V))\supset\Omega^p(V)$. Indeed, for any two $p$-rough paths $X=(1,X^1,...,X^{\lfloor p \rfloor})$ and $H=(1,H^1,...,H^{\lfloor p \rfloor})$, both controlled by a common control $\omega$, the pointwise sum $X+H:\triangle_J \to T^{(\lfloor p \rfloor)}(V)$ given by $(s,t)\mapsto X_{st}+H_{st}$ has $(X+H)^0_{st}=2(\neq 1)$, so it is not a rough path. We will now describe the strategy this paper employs to circumvent the fact that  pointwise addition fails, in general, to produce a rough path. 

To address this problem with $X+H$, we can then define the \textit{unit-preserving pointwise sum}, which is the map $X\oplus H:\triangle_J\to T^{(\lfloor p \rfloor)}(V)$ given by $(s,t)\mapsto (1,X^1+H^1,...,X^{\lfloor p \rfloor}+H^{\lfloor p \rfloor})$. However, $X\oplus H$ suffers from a similar problem: the multiplicativity condition is not satisfied. Indeed, for each $k=1,...,\lfloor p \rfloor$ and $\forall s,u,t\in J$ with $s\le u\le t$, we have that
\begin{small}
$$
\begin{aligned}
\left((X\oplus H)_{s, u} \otimes(X\oplus H)_{u, t}\right)^k
&=\sum_{j=0}^{k}  (X\oplus H)_{su}^j\otimes(X\oplus H)_{ut}^{k-j}\\
&=X_{s, t}^k+ H_{s, t}^k+\sum_{j=1}^{k-1} H_{s, u}^j \otimes X_{u, t}^{k-j}
 +\sum_{j=1}^{k-1} X_{s, u}^j \otimes H_{u, t}^{k-j}\\
&=(X \oplus H)_{s, t}^k+\sum_{j=1}^{k-1} H_{s, u}^j \otimes X_{u, t}^{k-j}
 +\sum_{j=1}^{k-1} X_{s, u}^j \otimes H_{u, t}^{k-j}.
\end{aligned}
$$
\end{small}
Hence, for each $k=1,\dots,\lfloor p\rfloor$,
\[
\Delta(X\oplus H)_{s,u,t}^k
=  \sum_{j=1}^{k-1} H_{s, u}^j \otimes X_{u, t}^{k-j}
 +\sum_{j=1}^{k-1} X_{s, u}^j \otimes H_{u, t}^{k-j},
\]which is in general not equal to zero. Thus, $X\oplus H$ is in general not multiplicative, so it is not a $p$-rough path. However, we can assume a stronger regularity condition on $H$ so that $\Delta(X\oplus H)$, the multiplicative defect of $X\oplus H$, is bounded supra-linearly in the control $\omega$, making $X\oplus H$ an almost $p$-rough path. This stronger regularity condition is precisely the one we introduced in Inequality \ref{ineq: H-space, intro}.

With this in hand, we can
\[
\begin{aligned}
&\|(X_{su}\otimes H_{ut}+H_{su}\otimes X_{ut})^k\| \\
& \leq \sum_{1\leq j \leq k-1} \|X_{su}^j\otimes H_{ut}^{k-j}+H_{su}^j\otimes X_{ut}^{k-j}\| \\
& \leq \sum_{1\leq j \leq k-1} \|X_{su}^j\|\| H_{ut}^{k-j}\|+\|H_{su}^j\| \|X_{ut}^{k-j}\|
\qquad (\text{admissible norms}) \\
& \leq \sum_{1\leq j \leq k-1}\Bigl[(1/\beta\cdot(j/p)!) \; \omega(s,t)^{\frac{j}{p}} \omega(s,t)^{\phi}
+\omega(s,t)^{\phi} \omega(s,t)^{\frac{k-j}{p}} (1/\beta((k-j)/p)!)\Bigr]\\
&\phantom{\leq}\qquad\qquad (\text{regularity of } H \text{ \& finite } p\text{-variation of } X) \\
&\leq K \omega(s,t)^{\frac{1}{p}+\phi}\sum_{1\leq j \leq k-1} \omega(s,t)^{\frac{j-1}{p}} + \omega(s,t)^{\frac{k-j-1}{p}} \\
& \leq K\left(\sum_{1\leq j \leq k-1} \omega(S,T )^{\frac{j-1}{p}} +\omega(S,T)^{\frac{k-j-1}{p}}\right) \omega(s,t)^{\frac{1}{p}+\phi}
\hspace{0.15cm} (\text{super additivity of } \omega)
\end{aligned}
\]

with $K=(\max_{j=1,..,\lfloor p \rfloor}\{(1/\beta\cdot(j/p)!)\})^2$. Hence, thanks to the fact that $H$ is bounded by $\omega^\phi$ with $\phi>1-1/p$, the multiplicative defect of $X\oplus H$ is bounded supra-linearly in $\omega$, and thus $X\oplus H$ is an almost rough path. Using the rough sewing lemma, we can then obtain a $p$-rough path, which we introduce notation for:
\[
X\boxplus H = \mathscr{S}(X\oplus H).
\]
The above observations motivate the viewpoint of this paper: although $\Omega^p(V)$ is not a vector space under pointwise addition, there is a natural class of perturbations $H\in\mathscr H^p(V)$ for which one can define a canonical operation via sewing, namely $(X,H)\mapsto X\boxplus H$ as defined above. The focus of this paper will be to analyse the properties of the operation $\boxplus$ and of the space $\mathscr H^p(V)$, which we will show to be a vector space.


\subsection{Review of the literature}

In the literature, several frameworks have been proposed to circumvent the absence of a linear structure on \(\Omega^p(V)\). For instance, if \(x\) is a smooth path in \(V\) and \(h\) is a path of bounded variation, one may consider the perturbed path \(x+h\) and its truncated signature\[S^{(\lfloor p \rfloor)}(x+h),\]
which can be expressed in terms of \(S^{(\lfloor p \rfloor)}(x)\), \(S^{(\lfloor p \rfloor)}(h)\) and explicit cross-iterated integrals (see, for instance, \cite{lyons1998differential,Friz_Victoir_2010}).\textbf{ In Lyons' 1998 paper on differential equations driven by rough signals, this principle is extended to genuinely rough drivers}: in Remark~B in \cite{lyons1998differential}, if \(X=(1,X^1,...,X^{\lfloor p \rfloor})\) and \(H=(1,H^1,...,H^{\lfloor p\rfloor})\) are $p$-rough paths mapping $\triangle_J$ to $T^{(\lfloor p \rfloor)}(V)$ which are controlled by $\omega$, and if $H$ satisfies
\[\|H^j_{st}\|\le K\omega(s,t)^{\phi} \qquad \forall(s,t)\in\triangle_J \text{ and } \forall j\in\{1,\ldots,\lfloor p \rfloor\},\]
for some constant $K\ge0$, then the map \(X\otimes H\) given by $(s,t)\mapsto X_{st}\otimes H_{st}$ defines an almost \(p\)-rough path. Then, Lyons' rough sewing lemma (originally in \cite{lyons1998differential}; also Theorem 4.3 in \cite{stflour}) canonically associates to it a unique genuine \(p\)-rough path \cite{lyons1998differential}, which in this paper we call $X\boxplus H$. As we later show (Theorem \ref{thm:X_otimes_H_is_ARP}), the map $X\oplus H$ given by $$(s,t) \mapsto (1, X^1_{st}+H^1_{st}, ..., X^{\lfloor p \rfloor}_{st}+H^{\lfloor p \rfloor}_{st})$$is also an almost \(p\)-rough path and its canonically associated rough path is $X\boxplus H$ as well. This framework defines a rigorous notion of additive perturbation of a rough path \(X\) by any \(H\) as defined above, and it does so by passing through the rough sewing lemma.

Several strands of work have developed constructions in a similar direction to \cite{lyons1998differential}, Remark B. \textbf{In Friz--Victoir’s monograph}, one first fixes a collection of vector fields
\[
f = (f_1,\dots,f_d), \qquad f_i : \mathbb{R}^e \to \mathbb{R}^e,
\]
of class \(\mathrm{Lip}^\gamma\) for some \(\gamma>p\).  The rough differential equation
\[
dY_t = f(Y_t)\,dX_t,
\qquad\text{i.e.}\qquad
dY_t = \sum_{i=1}^d f_i(Y_t)\,dX_t^i,
\quad Y_0 = y_0
\]
has a unique \emph{solution rough path} on \([S,T]\)
\[
Y = \mathbf{\Pi}_f(y_0;X) \in WG\Omega^p(\mathbb{R}^e),
\]
which depends continuously on \((y_0,X) \in \mathbb{R}^e \times WG\Omega^p(\mathbb{R}^d)\). The map
\[
\mathbf{\Pi}_f : \mathbb{R}^e \times WG\Omega^p(\mathbb{R}^d) \to WG\Omega^p(\mathbb{R}^e)
\]
is the \emph{Itô--Lyons map}. We write \(
\pi_f : \mathbb{R}^e \times WG\Omega^p(\mathbb{R}^d) \longrightarrow C^{p\text{\rm-var}}([S,T],\mathbb{R}^e)
\) for the associated \emph{solution trajectory}
\[
\pi_f(y_0;X) := \pi_1\big(\mathbf{\Pi}_f(y_0;X)\big)
   \in C^{p\text{\rm-var}}([S,T],\mathbb{R}^e),
\]
which is obtained by projecting the solution rough path onto its first level (in the notation of \cite{Friz_Victoir_2010} one writes \(\pi_f(0,y_0;X)\) to emphasize the starting time \(0\), which we suppress here). In this setting the map $\pi_f$ is shown to be locally Lipschitz, and, under the \(\mathrm{Lip}^\gamma\) regularity assumption on \(f\) with \(\gamma>p\), differentiable with respect to perturbations of the driving rough path along Young (in particular Cameron--Martin) directions, when \(\mathbb{R}^e\) is given its Euclidean topology and \(WG\Omega^p(\mathbb{R}^d)\) is equipped with the \(p\)-variation topology; see \cite[Ch.~11]{Friz_Victoir_2010}. 

More precisely, their approach leverages a translation operator $T_{\varepsilon h}$ on geometric rough paths, defined by first lifting the classical translation $(x,h)\mapsto x+h$ for paths $x,h$ of finite $q$-variation with $q\in[1,2)$ to a rough path on $G^{(\lfloor p\rfloor)}(\mathbb{R}^d\oplus\mathbb{R}^d)$ and then projecting back onto the first component; see \cite[Sec.~9.4.6]{Friz_Victoir_2010}. In this setting, for a fixed initial value $y_0$, given a geometric $p$-rough path $X$, and a perturbation $h$ of finite $q$-variation with $1/p+1/q>1$ (so that the Young pairing $(X,h)$ is well-defined), the curve
\[
\varepsilon \longmapsto \pi_f\bigl(y_0;T_{\varepsilon h}(X)\bigr)
\]
is shown to be differentiable as a map with values in the Banach space $C^{p\text{\rm-var}}([S,$ $T],\mathbb{R}^e)$, and its derivative at $\varepsilon=0$ is characterized as the unique solution of a linear rough differential equation driven by $(X,h)$; see \cite[Thm.~11.3]{Friz_Victoir_2010}. Thus the resulting “differential calculus along \(h\)” lives in the ambient Banach space
\(C^{p\text{\rm-var}}([S,T],\mathbb{R}^e)\) of solution paths; it does not equip the rough-path space
\(WG\Omega^p(\mathbb{R}^d)\) itself with a linear structure or an internal notion of displacement.

\textbf{Qian and Tudor construct a differential structure on the space \(W G\Omega^{p}\) \((V)\) of weakly geometric \(p\)-rough paths, for \(2<p<3\)}, by viewing a rough path \(X\) as an element of the Banach space \(C_0(\triangle_J;T^{(2)}(V))\) of continuous maps on the simplex with values in the truncated tensor algebra \(T^{(2)}(V))\), and by defining \(T_X W G\Omega^{p}(V)\) as the set of equivalence classes \([Z,\phi]\) of variational curves \(V_{(Z,\phi)}(\varepsilon)\) through \(X\), where two pairs \((Z,\phi)\) and \((\tilde Z,\tilde\phi)\) are equivalent if the associated variational curves have the same derivative at \(\varepsilon=0\) in the ambient space  \(C_0(\triangle_J;T^{(2)}(V))\) \cite{QianTudor2011}. Here \(Z\in W G\Omega^{p}(V\oplus V)\) is a Lyons-Victoir joint lift of the base rough path $X$ together with a candidate direction, and \(\phi\in\Omega^{p/2}(V\oplus V)\) is a second-level correction term that encodes the choice of cross-iterated integrals. Their construction yields a linear bundle of tangent spaces over \(W G\Omega^{p}(V)\), and it is defined via the ambient Banach space and non-canonical Lyons--Victoir extensions, but they explicitly note that it does not come with local trivializations, so it does not give rise to a genuine tangent vector bundle on rough path space.

\textbf{More recently, Bellingeri, Friz, Paycha and Prei{\ss} introduced the notions of smooth geometric rough paths} and, in a quasi‑geometric and Hopf‑algebraic framework, smooth rough paths and models, and endowed these smooth rough models with a canonical linear structure via a canonical sum and scalar multiplication \cite{bellingeri2021smooth}. Their construction is algebraic in nature using the Maurer–Cartan form on the truncated free nilpotent group $G^{(N)}(\mathbb{R}^d)$, they define a minimal extension of level-$N$ smooth geometric rough models by solving a linear differential equation in $G^{(N)}(\mathbb{R}^d)$, and show that this minimal extension coincides with the Lyons lift of a level-$N$ weakly geometric rough path \cite[Thm.~2.8, Prop.~2.14]{bellingeri2021smooth}. However, their framework is genuinely smooth: a smooth geometric rough model is, by definition, tensor-wise smooth (each coordinate $(s,t)\mapsto\langle X_{s,t},w\rangle$ is smooth for all words $w$), and its minimal extension is obtained as the Cartan development of a smooth path in the free nilpotent Lie algebra \cite[Defs.~2.1, 2.3, Rem.~2.9]{bellingeri2021smooth}. In Remark~2.17 they also observe, in the spirit of Lyons’ Remark~B \cite[Sec.~3.3.1~B]{lyons1998differential}, that an $N$-smooth geometric rough model $X$ can be canonically “added’’ to any $\gamma$-Hölder weakly geometric rough path $W$, for $\gamma\in(0,1)$, by requiring 
\[
(X\boxplus W)_{s,t} \;=\; X_{s,t}\otimes W_{s,t} + o(|t-s|)\quad\text{as }t\downarrow s,
\]
and appealing to the sewing lemma \cite[Rem.~2.17]{bellingeri2021smooth}. However, the main body of \cite{bellingeri2021smooth} constructs a vector space structure only on spaces of smooth geometric, quasi-geometric and Hopf–algebraic rough models; it does not provide a general perturbation theory that acts on an arbitrary weakly geometric, let alone a general (possibly non–weakly geometric), $p$-rough path by a smooth or Young-type path in the Lie algebra.


\subsection{Setup and results of this paper}

By contrast to the rest of the literature, the present work takes Lyons' observation in Remark~B as a starting point and builds a perturbation theory for arbitrary \(p\)-rough paths in \(\Omega^p(V)\) along lifts of degree-wise Young paths. \textbf{We now outline the framework; details appear in Sections 1--5.}

For each \(p\ge 1\) we define \emph{\(\mathscr H\)-space} (of degree $p$), i.e. the subset \(\mathscr H^p(V)\subset \Omega^p(V)\) of $p$-rough paths satisfying\textbf{ \eqref{ineq: H-space, intro}} that we discussed earlier. Elements of \(\mathscr H^p(V)\) are interpreted as \emph{rough displacements}: for any base rough path \(X\in\Omega^p(V)\) there is an extended operation
\[
\boxplus : \Omega^p(V)\times \mathscr H^p(V)\to \Omega^p(V),\qquad (X,H)\mapsto X\boxplus H,
\]
defined by considering the almost rough path \(X\oplus H\) and then canonically associating to it a unique genuine \(p\)-rough path via Lyons' Almost Rough Paths theorem (Theorem 4.3 in \cite{stflour}, which first appeared in \cite{lyons1998differential}). Thus \(\mathscr H^p(V)\) plays the role of a model linear space of increments, and \(\boxplus\) provides a canonical perturbation map  \(\Omega^p(V)\times \mathscr H^p(V)\to \Omega^p(V)\). Two of our main structural results are that this action has trivial kernel---if \(X\boxplus H=X\) for some \(X\in\Omega^p(V)\) and \(H\in\mathscr H^p(V)\), then \(H\) is the zero element of \(\mathscr H^p(V)\)---and that it is associative in the sense that
\[
(X\boxplus H)\boxplus \widetilde H \;=\; X\boxplus (H\boxplus \widetilde H)
\]
for all \(X\in\Omega^p(V)\) and \(H,\widetilde H\in\mathscr H^p(V)\) (see Theorem~\ref{thm:assoc-trivial-kernel}). With these results in mind, we may then say that the elements of \(\mathscr H^p(V)\) represent well-behaved perturbations of rough paths $X\in \Omega^p(V)$ by ``genuine'' rough paths $H\in \mathscr H^p(V)$, and by genuine we mean that the rough paths $H$ are not just extensions of paths of finite $q$-variation for $q\in[1,2)$; $H$ may be a non-canonical extension of a path of finite $q$-variation for $q\in[1,2)$.

\textbf{We now turn our discussion to a space, which we will call \(\mathfrak{I}^p(V)\)}, which is naturally a vector space and is bijective to $\mathscr H^p(V).$ Fix a control function \(\omega\) and an exponent \(\phi\in(1-1/p,1]\). Consider paths whose zero-th entry equals zero
\[
I_\cdot:J\to T^{(\lfloor p\rfloor)}(V),\qquad t\mapsto I_t=(0,I_t^1,\dots,I_t^{\lfloor p\rfloor}),
\]
and write \(I_{s,t}:=I_t-I_s\). We define \(\mathfrak{I}^p(V)\) to be the collection of the increments $I$ of such paths \(I_\cdot\) for which, for every \(j\in\{1,\dots,\lfloor p\rfloor\}\) and all \((s,t)\in\triangle_J\),
\begin{equation}\label{eq:Ip-bound-intro}
    \|I^j_{s,t}\|\;\le\;\omega(s,t)^{\phi}.
\end{equation}
In Section~2 we construct a canonical lift map
\[
\mathbb{1}^{(\cdot)}: \mathfrak{I}^p(V)\to \mathscr H^p(V) \subset \Omega^p(V),
\]
which associates to the increments \(I\) of an \(I_\cdot\) a rough displacement \(H=\mathbb{1}^{I}\), iteratively computed by applying Lyons' extension theorem to the levels of \(I=(0,I^1,\ldots,I^{\lfloor p \rfloor})\). Conversely, for every \(H\in\mathscr H^p(V)\) we define a map
\[
\operatorname{dev}:\mathscr H^p(V)\to \mathfrak{I}^p(V),
\]
which recovers the corresponding path \(I_\cdot\). We prove that \(\mathbb{1}^{(\cdot)}\) and \(\operatorname{dev}\) are inverse to each other, so that paths \(I\in\mathfrak{I}^p(V)\) are in bijection with the rough displacements in \(\mathscr H^p(V)\). Since the space \(\mathfrak{I}^p(V)\) carries a natural vector space structure by pointwise addition and scalar multiplication, the bijection \(\mathbb{1}^{(\cdot)}\) allows us to transport linear operations from \(\mathfrak{I}^p(V)\) to \(\mathscr H^p(V)\). We prove that the addition operation thus transported coincides with the previously-defined \(\boxplus\), and we call the induced scalar multiplication \(\odot\) (see Definition~\ref{def:scalar_multiplication_on_H}) on \(\mathscr H^p(V)\)).

\textbf{This is a crucial observation}: put in different words, the addition operation on \(\mathscr H^p(V)\) induced by \(\mathbb{1}^{(\cdot)}\)  given by $H\boxplus \widetilde H$ for every \(H\in\mathscr H^p(V)\) and \(\widetilde H\in\mathscr H^p(V)\) coincides with the rough path obtained by sewing the almost rough path $H\oplus\widetilde H=(1,H^1+\widetilde H^1,...,H^{\lfloor p\rfloor}+\widetilde H^{\lfloor p\rfloor})$. 

We now explain this crucial observation in greater detail. Given \(X\in\Omega^p(V)\) and \(H\in\mathscr H^p(V)\), we consider the map
\[
(X\oplus H)_{s,t} := (1,X^1_{s,t}+H^1_{s,t},\ldots,X^{\lfloor p \rfloor}_{s,t}+H^{\lfloor p \rfloor}_{s,t})\quad (s,t)\in\triangle_J,
\]
viewed as a map on the two-simplex with values in the truncated tensor algebra \(T^{(\lfloor p\rfloor)}(V)\). In Theorem \ref{thm:X_otimes_H_is_ARP}, we show that \(X\oplus H\) is a \(\theta\)-almost \(p\)-rough path with \(\theta=\phi+1/p\). Then, if we define
\[
X\boxplus H := \mathscr{S}(X\oplus H),
\]
we show that, when considering two $H$ and $\widetilde H$ in $\mathscr{H}^p(V)$, the following identity holds:
$$\mathscr{S}(H\oplus\widetilde H) \stackrel{\text{def}}{=} H\boxplus \widetilde H =\mathbb{1}^{\operatorname{dev}(H)}\boxplus\mathbb{1}^{\operatorname{dev}(\widetilde H)}= \mathbb{1}^{\operatorname{dev}(H)+\operatorname{dev}(\widetilde H)};$$
in other words, $\boxplus$ viewed as the operation induced onto $\mathscr H^p(V)$ by the bijection $\mathbb{1}^{(\cdot)}$ coincides with $\boxplus$ viewed as the sewing of the unit-preserving pointwise addition of two elements $H$ and $\widetilde H$ of $\mathscr{H}^p(V)$.

\textbf{The final part of the paper enlarges the class of admissible perturbations without changing the resulting space of displacements}. We introduce \emph{almost \(\mathscr H\)-space} \(\mathscr{H}^{\mathrm{am},p}(V)\), consisting of \(\theta\)-almost \(p\)-rough paths \(H=(1,H^1,...,H^{\lfloor p \rfloor})\) whose degrees \(H^j\) satisfy Inequality \eqref{ineq: H-space, intro}. For any \(H\in\mathscr{H}^{\mathrm{am},p}(V)\) and \(X\in\Omega^p(V)\), \(X\oplus H\) is again a \(\theta\)-almost \(p\)-rough path (with \(\theta=\phi+1/p\)), and the sewing theorem canonically associates to it a rough path \(X\boxplus H\). A priori this construction could produce strictly more displacements than those arising from genuine \(\mathscr{H}^p(V)\)-perturbations. Our main result in Section~5 shows that this is not the case: for every \(X\in\Omega^p(V)\),
\[
\{X\boxplus H : H\in\mathscr{H}^p(V)\}
=
\{X\boxplus H : H\in\mathscr{H}^{\mathrm{am},p}(V)\}.
\]
Thus the true space of displacements around a given rough path \(X\) is already exhausted by \(\mathscr{H}^p(V)\); enlarging the class of perturbations to almost rough paths does not enlarge the orbit.


\subsection{Comparison to other results in the literature}

Conceptually, this framework is complementary to the earlier approaches mentioned above. It develops Lyons' observation that, when \(H\) is of bounded variation or classical Young type, the object often denoted \(X\otimes H\) is almost multiplicative \cite{lyons1998differential}, into a more general perturbation theory in which the admissible increments form a vector space \(\mathscr{H}^p(V)\), are shown to be in bijection with degree-wise Young paths $\mathfrak I^p(V)$, and act on the whole rough path space \(\Omega^p(V)\) via the operation $\boxplus$ given by \(X\boxplus H = \mathscr{S}(X\oplus H)\). When \(H\) is weakly geometric, the bijection with Lie paths identifies elements of \(\mathscr{H}^p(V)\) with degree-wise Young paths in the truncated Lie algebra, so that our displacements admit both algebraic and pathwise descriptions. The construction is formulated directly at the level of \(\Omega^p(V)\) and does not require any weakly geometric or group-valued hypotheses, hence applies equally to non--weakly geometric rough paths.

Our results also clarify the relation with the translation operators of Friz--Victoir: classical translations in the direction of a rough path $H$ of finite $q$-variation for $q\in[1,2)$ (and in particular Cameron--Martin) paths are recovered as special cases of the \(\boxplus\)-operation when \(H\) is a canonical lift of a $q$-rough path\cite{Friz_Victoir_2010}. In contrast to their framework, we obtain a genuine vector space of displacements on (both weakly and non-weakly geometric) rough path space and show that perturbation by a degree-wise Young path in the Lie algebra induces a well-defined operation \(\boxplus\) on \(\Omega^p(V)\) itself, which can be interpreted as sewing the unit-preserving pointwise sum of two rough paths. Compared to the differential structure proposed by Qian and Tudor \cite{QianTudor2011}, our ``tangent model'' is internal to \(\Omega^p(V)\), globally trivialized via the bijection with Lie paths, and defined for all \(p\ge1\), not only for \(2<p<3\). Finally, while Bellingeri--Friz--Paycha--Prei{\ss} construct a canonical linear structure on spaces of smooth rough paths \cite{bellingeri2021smooth}, the \(\mathscr{H}\)-space approach allows us to perturb genuine (weakly and non-weakly geometric) rough paths by elements of \(\mathscr{H}\)-space; i.e. it provides an action of our displacement space (\(\mathscr{H}\)-space) on $p$-rough path space ($\Omega^p(V)$), and it dispenses of smoothness requirements.


\subsection{Structure of this paper}

The paper is organized as follows. \textbf{Section~1} collects preliminaries on tensor series, rough paths and almost rough paths, and fixes the notation used throughout. We introduce and define \(\mathscr H^p(V)\) and the operation $\boxplus$. \textbf{In Section~2 }we explore the properties of \(\mathscr H^p(V)\), define \(\mathfrak{I}^p(V)\) (and show it can be characterised as the space of degree-wise Young paths), and construct the lift map \(\mathbb{1}^{(\cdot)}\).  We then prove that the lift map \(\mathbb{1}^{(\cdot)}\) is a bijection between \(\mathfrak{I}^p(V)\) and \(\mathscr H^p(V)\). \textbf{Section~3} gathers the analytic lemmas needed in Section 4. \textbf{Section~4} contains the main structural results: we show firstly that for $I,\widetilde I\in \mathfrak I ^p(V),$ 
\begin{equation}\label{eqn: structure of paper, intro}
    \mathbb{1}^I  \boxplus \mathbb{1}^{\widetilde{I}} = \mathbb{1}^{I+\widetilde{I}} ,
\end{equation}and secondly that our perturbations are associative in the following sense
$$
(X \boxplus H) \boxplus \widetilde{H} = X \boxplus (H \boxplus \widetilde{H}).$$
\textbf{In Section~5}, we first show that displacements have trivial kernel, in the sense that, if $H\in \mathscr H^p(V)$ and $X\in\Omega^p(V)$, then $X\boxplus H = X\iff H=\mathbb{1}$. We then discuss how the bijection $\mathbb 1^{(\cdot)}$ allows us to endow $\mathscr H^p(V)$ with a vector space structure with addition $\boxplus$ and scalar multiplication $\odot$, where $\boxplus$ is justified by Equation \eqref{eqn: structure of paper, intro}, and $\odot$ is simply defined to be the induced operation carried over by the bijection $\mathbb 1^{(\cdot)}$. Finally, we introduce the almost \(\mathscr H\)-space \(\mathscr H^{\mathrm{am},p}(V)\), extend the perturbation map to almost rough paths, and prove that this enlargement does not change the set of attainable displacements from any given base rough path.

	\newpage
\section{Preliminaries}

Throughout this piece, $V$, $W$ and $E$ will stand for generic Banach spaces and $J=[S,T]$ will be a fixed interval.

The following definition, taken from \cite{stflour}, will serve as a key tool that we shall repeatedly use to obtain bounds on elements of the tensor powers of a Banach space.

\subsection{Tensor algebra toolkit}
We first fix the tensor-algebra conventions (admissible norms, truncations, and unit-preserving linear operations) that will be used throughout.

\begin{definition}[Admissible Norm]\label{def: admissible}
We say that the tensor powers of the Banach space $V$ are endowed with admissible norms if the following conditions hold:
\begin{itemize}
    \item For each $n \geq 1$, the symmetric group $S_n$ acts by isometries on $V^{\otimes n}$
$$
\|\sigma v\|=\|v\| \quad \forall v \in V^{\otimes n}, \forall \sigma \in S_n .
$$
\item  The tensor product has norm 1 , i.e. for all $n, m \geq 1$,
$$
\|v \otimes w\| \leq\|v\|\|w\| \quad \forall v \in V^{\otimes n}, w \in V^{\otimes m} .
$$
\end{itemize}
\end{definition}

The next three statements (Definitions \ref{def:tensor-series}–\ref{def:trunc-tensor}) are adapted from Chapter 2 of \cite{stflour} and set up the key vector spaces in which the functionals studied in this paper take values.

\begin{definition}[Tensor algebra, Ch.~2 in \cite{stflour}]\label{def:tensor-series}
The space of formal tensor series over \(E\) is
\[
T((E)):=\{a=(a^0,a^1,\ldots)\ :\ a^n\in E^{\otimes n}\ \text{for all }n\ge0\},
\]
with addition \(a+b=(a^0+b^0,a^1+b^1,\ldots)\). We may also write $T^{(\infty)}(E)=T((E)).$ We write $\operatorname{pr}_j((a^0,a^1,\dots)):=a^j.$
\end{definition}
\begin{proposition} $T((E))$ is a unital (non-commutative) algebra with unit \(\mathbb 1=(1,0,0,\ldots)\) and product \(a\otimes b=(c^0,c^1,\ldots)\) given by \(c^n=\sum_{k=0}^n a^k\otimes b^{n-k}\). The product $a \otimes  b$ is also denoted by $ab$.
\end{proposition}

\begin{proposition} \label{prop: Bn is an ideal} For $n\ge1$, define
$$ B_n(E):=\{a=(a^0,a^1,\ldots)\ :\ a^0=\cdots=a^n=0\}\subset T((E)).$$
 Then $B_n$ is a two-sided ideal in $T((E))$.
\end{proposition}

\begin{definition}[Truncated tensor algebra]\label{def:trunc-tensor}
For \(n\ge1\), we define the \textit{truncated tensor algebra of order} \(n\) to be the quotient
\[
T^{(n)}(E):=T((E))/B_n.
\]
We write \(\pi_n:T((E))\to T^{(n)}(E)\) to denote the canonical homomorphism and $\mathbb{1}(n)=\pi_n(\mathbb{1})$ for its unit. When a single truncation level is fixed (e.g. \(n=\lfloor p\rfloor\)), we may simply write \(\mathbb{1}\) for \(\mathbb{1}(n)\).
\end{definition}

\begin{proposition}[Canonical structure of \(T^{(n)}(E)\)]\label{prop:trunc-iso}
Equip \( \bigoplus_{k=0}^n E^{\otimes k}\) with the product $\widetilde{\otimes}$ given by, for $(a^0,\ldots,a^n)$ and $(b^0,...,b^n)$ in $\bigoplus_{k=0}^n E^{\otimes k}$
\[
(a^0,\ldots,a^n)\,\widetilde{\otimes}\,(b^0,\ldots,b^n)=(c^0,\ldots,c^n),\qquad
c^k=\sum_{i+j=k} a^i\otimes b^j.
\]
Then there exists an algebra isomorphism between \( \bigoplus_{k=0}^n E^{\otimes k}\)  and $T^{(n)}(E)$. For simplicity, we will denote the product $\widetilde{\otimes}$ by $\otimes$.
\end{proposition}

We also consider a distinguished subset, which can be turned into a vector space in a natural way:

\begin{definition}[Unit tensor series]
We define the \textit{unit tensor series} to be the set given by
\[
\widetilde T((E)) := \{a=(a^0,a^1,a^2,\dots)\in T((E)):\ a^0=1\}.
\]
For $n\in\mathbb N$, define its truncation
\[
\widetilde T^{(n)}(E) := \{a\in T^{(n)}(E):\ a^0=1\},
\]
where $\mathbb{1}(n):=(1,0,...,0)\in T^{(n)}(E)$.
\end{definition}

\begin{definition}[Unit-preserving addition and scalar multiplication]\hfill\break
 Let $n\in\mathbb N\cup\{\infty\}$, where $\widetilde T^{(\infty)}(E):=\widetilde T((E))$
and $\mathbb 1{(\infty)}:=\mathbb 1$. For $a,b\in \widetilde T^{(n)}(E)$ and $\lambda\in\mathbb R$, define
\[
a \oplus b \;:=\; a + b - \mathbb 1{(n)}, 
\qquad
\lambda\; \tilde{\cdot}\; a \;:=\; \mathbb 1{(n)} + \lambda\,(a-\mathbb 1{(n)}).
\]
\end{definition}

\begin{proposition}[Vector space structure on $\widetilde T^{(n)}(E)$]
For each $n\in\mathbb N\cup\{\infty\}$, the triple $\big(\widetilde T^{(n)}(E),\oplus,\;\tilde{\cdot}\;\big)$ is a real vector space. Its zero element is $\mathbb 1{(n)}$.
\end{proposition}

\begin{proposition}[Group structure on $\widetilde T^{(n)}(E)$]
For each $n\in\mathbb N\cup\{\infty\}$, the pair $\big(\widetilde T^{(n)}(E),\otimes)$ is a group. Its identity element is $\mathbb 1{(n)}$ and the inverse of an element $\mathbf a=(a^i)_{i=0}^n\in \widetilde T^{(n)}(E)$ is given by
$$\mathbf a^{-1}=\sum_{k=0}^n (\mathbb 1-\mathbf a)^{\otimes k}
.
$$
\end{proposition}

Below, we will introduce control functions and prove a lemma about them.

\subsection{Controls and rough paths}
We now introduce controls, $p$-variation bounds, and the resulting rough-path spaces, together with the extension machinery that will be used later.

\begin{definition} We set $\triangle_J:=\{(s,t)\in J^2 \:| \: s\le t \}$. \end{definition}	

\begin{definition}[Controls, Def. 1.9 in \cite{stflour}] \label{def:controls}
	A control function, or control, on $J$ is a continuous non-negative function $\omega$ on $\triangle_J$ which is super-additive in the sense that
$$
\omega(s, u)+\omega(u, t) \leqslant \omega(s, t) \quad \forall s \leqslant u \leqslant t \in J,
$$
and for which $\omega(t, t)=0$ for all $t \in J$.
\end{definition}

\begin{lemma}[Controls are convex cone] \label{lem: controls are a convex cone}
The space given by $$\{\omega\, | \,\omega \text{ is a control over } J\}$$ is a convex cone.
\end{lemma}
\begin{proof}[Proof sketch]
If $\omega,\tilde\omega$ are controls and $a,b\ge0$, then $a\omega+b\tilde\omega$ inherits continuity, non-negativity, and the null diagonal, and remains super-additive by a triangle-inequality argument (using that $\omega$ and $\tilde\omega$ are super-additive); hence,  $a\omega+b\tilde\omega$ is a control and the set of controls is a convex cone.
\end{proof}

\begin{definition}[Multiplicative functional, drawn from \cite{stflour}]
	\label{def: multiplicative functional}
	Let $n \geqslant 1$ be an integer. Let $X:\triangle_J \longrightarrow T^{(n)}(V)$ be a continuous map. For each $(s, t) \in \triangle_J$, denote by $X_{s, t}$ the image by $X$ of $(s, t)$ and write
	\vspace{-5pt}        
	$$
	X_{s, t}=\left(X_{s, t}^0, X_{s, t}^1, \ldots, X_{s, t}^n\right) \in \mathbb{R} \oplus V \oplus V^{\otimes 2} \oplus \ldots \oplus V^{\otimes n}.
	$$
	The function $X$ is called a multiplicative functional of degree $n$ in $V$ if $X_{s, t}^0=1$ for all $(s, t) \in \triangle_J$ and
	\vspace{-8pt}
	\begin{equation}
		X_{s, u} \otimes X_{u, t}=X_{s, t} \quad \forall s, u, t \in J, \quad s \leqslant u \leqslant t .
	\end{equation}
\end{definition}		

\begin{remark}We will refer to the space of continuous functionals from $\triangle_J$ into $T((V))$ by $ C_0\left(\triangle_J, T((V))\right) $, whereas we will call $ C_0^{\operatorname{mult}} \left(\triangle_J, T((V))\right) $ the subset of $ C_0\left(\triangle_J, T((V))\right) $ that contains all of the multiplicative functionals.
    
\end{remark}

The following theorem is due to \cite{Hara_2010} but an earlier form of it was first used in \cite{lyons1998differential} (with a less sharp $1/p^2$ factor instead of $1/p$) to prove the extension theorem (which we reproduce as Theorem \ref{thm:extension-theorem} in this work).
\begin{theorem}[Neo-classical Inequality]
For any $p \in[1, \infty), n \in \mathbb{N}$ and $s, t \geq 0$,

$$
\frac{1}{p} \sum_{i=0}^n \frac{s^{\frac{i}{p}} t^{\frac{n-i}{p}}}{\left(\frac{i}{p}\right)!\left(\frac{n-i}{p}\right)!} \leq \frac{(s+t)^{\frac{n}{p}}}{\left(\frac{n}{p}\right)!} .
$$
\end{theorem}
\begin{definition}[$\beta(p)$] \label{def:beta} Given $p\ge1$, we set

$$
\begin{aligned}
&\beta(p)=p\left(1+\sum_{r=3}^{\infty}\left(\frac{2}{r-2}\right)^{\frac{\lfloor p\rfloor+1}{p}}\right).
\end{aligned}
$$
    
\end{definition}
Henceforth, we use the notation $x!:=\Gamma(x+1)$.

\begin{definition}[Finite $p$-variation, drawn from \cite{stflour}]\label{def:finite p-var}
	Let $p \geqslant 1$ be a real number and $n \geqslant 1$ be an integer. Let $\omega: \triangle_J \to [0,\infty)$ be a control and $K\ge0$ be a constant. Let $X: \triangle_J \to T^{(n)}(V)$ be a functional. We say that $X$ has finite $p$-variation on $\triangle_J$ controlled by $\omega$ with constant $K$ if
	\begin{equation}\label{eqn:finite-p-variation}
	    \left\|X_{s,t}^i\right\| \leqslant K\frac{\omega(s,t)^{\frac{i}{p}}}{\beta\cdot\left(\frac{i}{p}\right)!} \quad \forall i = 1, \ldots, n,\ \forall(s,t) \in \triangle_J.
	\end{equation}
where $\beta:=\beta(p)$.

In general, we say that $X$ has finite $p$-variation, without specifying the control $\omega$, if there exists some control $\omega$ such that Equation~\eqref{eqn:finite-p-variation} is satisfied. If we say that $X$ has finite $p$-variation on $\triangle_J$ controlled by $\omega$ without specifying any constant, then we mean that $X$ has finite $p$-variation on $\triangle_J$ controlled by $\omega$ with constant $1$.
\end{definition}


\begin{definition} \label{def:rough-path-space}
	Let $k \geqslant 1$ be an integer. Let $p \geqslant 1$  be a real number. Denote by $\Omega^{(k,p)}(V)$ the space of $k$-step multiplicative functionals of finite $p$-variation with values in $T^{(k)}(V)$. Set $\Omega^p(V) := \Omega^{(\lfloor p \rfloor, p)}$. We say $X$ is a $p$-rough path if $X \in \Omega^p(V)$. 
\end{definition}

We introduce below the set of group-like elements (further discussed in  \cite{stflour}  and \cite{reutenauer}).

\begin{definition}[Group-like elements, drawn from \cite{stflour}]
	An element $\mathbf{a} \in \widetilde  T((V))$ is said to be group-like if the evaluation mapping $\mathrm{ev}_{\mathbf{a}}: T\left(V^*\right) \longrightarrow \mathbb{R}$ defined by $\mathrm{ev}_{\mathbf{a}}\left(\mathbf{e}^*\right)=\mathbf{e}^*(\mathbf{a})$ is a morphism of algebras when $T\left(V^*\right)$ is endowed with the shuffle product, where $V^*$ denotes the dual space of $V$. The set of group-like elements is denoted by $G^{(*)}$.
\end{definition}

Below, we define the subspace $WG\Omega^p(V)\subset\Omega^p(V)$ of so-called weakly geometric $p$-rough paths.

\begin{definition}\label{def: weakly geometric RPs}
Let $p \ge1 $ be a real number and $k$ be a natural number. Let $G^{(k)}(V):=\pi_k(G^{(*)})\subset T^{(k)}(V)$. We say $X\in WG\Omega^{(k,p)}(V)$ if $X\in\Omega^{(k,p)}(V)$ with $\mathrm{Im}(X)\subset G^{(k)}(V)$.

\end{definition}

The following theorem is one of the fundamental results in \cite{stflour}.
\begin{theorem}[Extension theorem, drawn from \cite{stflour}]
	\label{thm:extension-theorem}
	Let $p \geqslant 1$ be a real number and $n \geqslant 1$ an integer. Let $X: \triangle_J \longrightarrow T^{(n)}(V \, )$ be a multiplicative functional with finite $p$-variation controlled by a control $\omega$. Assume that $n \geqslant \lfloor p \rfloor$. Then there exists a unique extension of $X$ to a multiplicative functional $\triangle_J \longrightarrow T((V))$ which possesses finite $p$-variation.\\
	More precisely, for every $m \geqslant \lfloor p \rfloor + 1$, there exists a unique continuous function $X^m: \triangle_J \longrightarrow V^{\, \otimes m}$ such that
	$$
	(s, t) \mapsto X_{s, t}=\left(\, 1, X_{s, t}^1, \ldots, X_{s, t}^{\lfloor p\rfloor}, \ldots, X_{s, t}^{ m}, \ldots\right) \in T((V))
	$$
	is a multiplicative functional with finite $p$-variation controlled by $\omega$. By this we mean that
	$$
	\left\|X_{s, t}^{ i}\right\| \leqslant \frac{\omega(s, t)^{\frac{i}{p}}}{\beta\cdot\left(\frac{i}{p}\right)!} \quad \forall i \geqslant 1, \quad \forall(s, t) \in \triangle_J \text{,}
	$$
where $\beta=\beta(p)$.
\end{theorem}

We would like to apply Theorem \ref{thm:extension-theorem}, thus we define here below the function $\operatorname{Ext}(\cdot)$, which achieves exactly so.

\begin{definition}
	Let $p \geqslant 1$ be a real number and $k \geqslant \lfloor p \rfloor $ be an integer. Define $\operatorname{Ext}_{(k,p)}: \Omega^{(k, p)}(V) \to \Omega^{(k+1, p)}(V)$ to be the map taking $X \in \Omega^{(k, p)}(V)$ to its unique extension $\operatorname{Ext}_{(k,p)}(X) \in \Omega^{(k+1, p)}(V)$ given by Theorem \ref{thm:extension-theorem}. 
\end{definition}

\subsection{Sewing and almost rough paths}
Next we recall the ‘almost’ framework and the rough sewing lemma, which upgrades almost multiplicative functionals to genuine multiplicative ones. We also record basic invariance properties of sewing that will be reused later.

\begin{definition}[Almost Multiplicative, drawn from \cite{stflour}]\label{def: almost multiplicativity}
Let $n\ge1$ be an integer, $\theta > 1 $ be a real number, and $X=(1,X^1,...,X^n) : \triangle_J \longrightarrow T^{(n)}(V)$ be a functional. Then we say $X$ is $\theta$-almost multiplicative of degree $n$ if there exist a constant  $ K\ge0 $ and a control $\omega$ such that, for all $s  \leqslant u \leqslant t $ with $s,u,t\in J$,$$
	\left\| \left( X_{su} \otimes X_{ut} - X_{st}  \right)^i \right\| \leqslant K\omega(s,t)^\theta \quad  \quad \forall i\in\{1, \ldots, n\} 
	\text{.}
	$$
If we wish to specify a control $\omega$ and a constant $K\ge0$ such that the bound is satisfied, we say that $X$ is $\theta$-almost multiplicative of degree $n$ controlled by $\omega$ and with constant $K$. If we say that $X$ is controlled by $\omega$ without specifying a constant, we mean that the constant is $1$.
\end{definition}    

\begin{definition}
Let $p\in\mathbb R_{\ge1}$, $\theta\in\mathbb R_{>1}$ $n\in\mathbb N$, $\omega$ be a control, and $K\in\mathbb R_{\ge0}$. Given a functional $X : \triangle_J \longrightarrow T^{(n)}(V)$, we say that $X$ is $\theta$-almost multiplicative of degree $n$ and of finite $p$-variation controlled by $\omega$ with constant $K$, if both (i) $X$ is $\theta$-almost multiplicative of degree $n$ controlled by $\omega$ with constant $K$, and (ii) $X$ has finite $p$-variation controlled by $\omega$ and with constant $K$. 
\end{definition}

\begin{definition}\label{def: almost rough path}
	Given $p \geqslant 1 \ $ and $   \theta > 1$, we say that $X : \triangle_J \longrightarrow T^{(\lfloor p \rfloor)}(V)$ is a $\theta$-almost $p$-rough path controlled by $\omega$ with constant $K$ if $X$ is $\theta$-almost multiplicative and has finite $p$-variation controlled by $\omega$ with constant $K$. 

\end{definition} 


\begin{definition}\label{def: Omega^(am, (k,p))_omega,theta}
	Let $p \in \mathbb{R}_{\geqslant 1}$ and $\theta \in \mathbb{R}_{> 1}$, as well as $k\in\mathbb N$. Let $\omega: \triangle_J \to [0,\infty)$ be a control. We then define:

    \begin{small}
    $$
	\begin{aligned}
		\Omega_{\theta, \omega}^{\text{am},(k,p)}(V) & := \left\{X: \triangle_J \longrightarrow T^{(k)}(V) \left|   \begin{array}{l}
			\operatorname{\:\exists \,K \geqslant 0 \ s.t.\:X \ is \ a \ k-step\ \theta - almost    multi-} \\
			\operatorname{\: plicative functional \,of \,
             finite }\,p\operatorname{-variation}   \\
			\operatorname{   \, controlled \, by} \omega\text{ with constant $K$}
		\end{array} 
		\right. \right\} \\[6pt]
		\Omega_\omega^{(k,p)}(V) & :=
		\left\{X: \triangle_J \longrightarrow T^{(k)}(V) \left|   \begin{array}{l}
			\operatorname{\:\exists \,K \geqslant 0 \ s.t. \ X \: is \,a \ k-step multiplicative }  \\
		  \operatorname{ \;functional \,of \,
             finite }p \text{-variation } \\
			\:\operatorname{controlled  by}  \omega \text{ with constant }K
		\end{array} 
		\right. \right\}\\            
	\end{aligned}
	$$
    \end{small}

Moreover, we may write $\Omega_\omega^{\text{am},(k,p)}$ if we wish to fix the control $\omega$ but not commit to a choice of exponent, or $\Omega_\theta^{\text{am},(k,p)}$ if we wish to fix the exponent $\theta$ without fixing the control.
\end{definition}
\begin{remark}
	Note that in the notation $\Omega^{\text{am}, (k, p)}_{ \theta, \omega}$, while the dependence on $(k, p)$ and $\theta$ is ``exact'', an element of $\Omega^{\text{am}, (k, p)}_{\theta, \omega}$ is only controlled by $\omega$ ``up to a multiplicative constant $K$''.  
\end{remark}

\begin{remark}[Partitions notation]\label{rmk:partitions}
Throughout, when we write $D\subset[s,t]$ we mean a finite partition
$D=\{s=t_0<t_1<\cdots<t_n=t\}$ of the interval $[s,t]$. For such $D$, if
$Y$ is either a $T((V))-$valued or a $T^{(k)}(V)-$valued functional for some $k\in \mathbb N$, we use notation $Y^{D}:=Y_{t_0,t_1}\otimes\cdots\otimes Y_{t_{n-1},t_n}$. For $Y^D\in T^{(k)}(V)$ we write $Y^{D,j }:= \operatorname{pr}_j(Y^D)\in V^{\otimes j}$.

\end{remark}

The rough sewing lemma, stated below, first appeared in \cite{lyons1998differential}. Part (i) of the following theorem is directly quoted from \cite{stflour}, where it is stated as ``Theorem 4.3''. On the other hand, Part (ii) has been adapted from the ``upgrading statement", which is a separate claim found in the proof of Theorem 4.3 in \cite{stflour}.

\begin{theorem}[Rough Sewing Lemma]\label{thm:Rough-Sewing-Lemma}
	Let $p \geqslant 1$ and $\theta>1$ be real numbers and let $k\in\mathbb N$. Let $\omega: \triangle_J \longrightarrow[0,+\infty)$ be a control. Let $X=(1,X^1,...,X^{k}): \triangle_J \longrightarrow T^{(k)}(V)$ be a $\theta$-almost multiplicative functional of degree $k$ with finite $p$-variation controlled by $\omega$; i.e., $X\in\Omega^{\text{am},(k,p)}_{\theta, \omega}(V)$. Then, 
    
    (i) there exists a unique multiplicative functional of degree $k$ with finite $p$-variation $\bar{X}: \triangle_J \longrightarrow T^{(k)}(V)$ that satisfies\begin{equation}\label{bound almost mult. functionals}
		\sup_{\substack{S\leqslant s < t \leqslant T \\
				\text{s.t. }\omega(s,t)>0;\\i=1, \ldots, k 
                }} 
		\frac{\left\| \bar{X}_{s,t}^i - X_{s,t}^i \right\|}{\omega(s,t)^\theta} < +\infty;
	\end{equation}
i.e., $\bar{X}$ is the unique functional in $\Omega^{(k,p)}(V)$ satisfying this bound. Moreover, there exists a constant $K$ which depends only on $p, \theta$ and $\omega(S, T)$ such that the supremum above is smaller than $K$ and the $p$-variation of $\bar{X}$ is controlled by $\omega$ with constant $K$ (i.e., $\bar{X} \in \Omega^{(k,p)}_\omega(V)$).    

    (ii) The construction of $\bar{X}$ is as follows: set $X^{(0)} = X$, and for $m\in\{0,..., k-1\}$, given $X^{(m)}$, set
        \begin{equation} \label{eqn: construction of associated rough path}
(X^{(m+1)})^j = 
\begin{cases}
(X^{(m)})^j, 
   & \hspace{-0.17cm}\text{if } j \neq m+1,\\[6pt]
\displaystyle \lim_{\substack{D \subset [s,t]\\|D|\to 0}}
   \Bigl(1,\; (X^{(m)})^1,\; \ldots,\; (X^{(m)})^m,\; X^{m+1}\Bigr)^{D,\,m+1},
   & \hspace{-0.17cm}\text{if } j = m+1,
\end{cases}
        \end{equation}
         
        \vspace{-8pt}
        where $X^{(m)}$ is multiplicative up to step $m$ and has finite $p$-variation. Then, we set $\bar{X}:= X^{(k)}\in \Omega^{(k,p)}_\omega(V)$, which is multiplicative up to degree $k$ and has finite $p$-variation. 
\end{theorem}

\begin{definition} \label{def: sewing map mathscr S}
Take real numbers $p\ge1$ and $\theta >1$, a natural number $k\in \mathbb N$ and a control $\omega$. We define the function $\mathscr{S}^{(k,p)}_{ \theta, \omega}: \Omega_{ \theta, \omega}^{\text{am}, (k,p)} (V)\to \Omega_{\omega}^{(k,p)}(V)$ taking an $X \in \Omega_{ \theta, \omega}^{\text{am}, (k,p)}$ and mapping it to its unique $\bar{X} \in \Omega^{(k,p)}_{ \omega}(V)$, i.e. $\mathscr{S}^{(k,p)}_{ \theta, \omega}(X) := \bar{X}$ as per Theorem \ref{thm:Rough-Sewing-Lemma}. 
\end{definition}

\begin{remark}We call \(\mathscr{S}\) the “sewing” map. A formal parallel to the sewing map of \cite{gubinelli04} can be established, but we will not discuss this.
\end{remark}

\begin{lemma}[Control- and $p$-independence of the sewing map]\label{lem:control-independence-sewing}
Fix $k\in \mathbb N$. Let $X:\triangle_ J\to T^{(k)}(V)$ be a functional such that
\[
X\in \Omega^{\text{am},(k,p)}_{ \theta, \omega}(V)
\quad\text{and}\quad
X\in \Omega^{\text{am},(k,\tilde p)}_{\tilde\theta, \tilde\omega}(V)
\]
for some controls $\omega,\tilde\omega$, some real numbers $p,\tilde p\ge1$ and some exponents $\theta,\tilde\theta>1$ (in the sense of Definitions \ref{def: almost multiplicativity}--\ref{def: Omega^(am, (k,p))_omega,theta}).
Set
\[
\omega^\ast := \omega+\tilde\omega,
\qquad
\theta^\ast := \min\{\theta,\tilde\theta\},
\qquad
p^\ast := \max\{p,\tilde p\}.
\]
Then
\[
\mathscr{S}^{(k,p)}_{ \theta, \omega}(X) \;=\; \mathscr{S}^{(k,\tilde p)}_{\tilde \theta,\tilde \omega}(X)
\quad\text{in } T^{(k)}(V)\text{-valued multiplicative functionals.}
\]
In particular, both sewings lie in $\Omega^{(k,p^\ast)}(V)$ (and are equal to $\mathscr{S}^{(k,p^\ast)}_{\theta^\ast,\omega^\ast}(X))$. The sewn multiplicative functional depends only on $X$ and $k$, not on the choice of $p$ nor on the control $\omega$ nor the exponent $\theta$ used to witness that $X$ is an almost rough path.
\end{lemma}
\begin{proof}
By Lemma \ref{lem: controls are a convex cone}, $\omega^\ast=\omega+\tilde\omega$ is a control.

\noindent\textbf{\boldmath\underline{Step 1: $X\in \Omega^{\text{am},(k,p^\ast)}_{\theta^\ast,\omega^\ast}(V)$}}

Without loss of generality assume $p^\ast=p$ (otherwise swap the roles of $(p, \theta, \omega)$ and $(\tilde p,\tilde\theta, \tilde\omega)$; the argument is identical).
Since $X\in \Omega^{\text{am},(k,p)}_{ \theta, \omega}(V)$, there exists a constant $K\ge0$ such that for all $i\in\{1,\dots,k\}$ and all $(s,t)\in\triangle_J$,
\[
\|X^i_{st}\|
\le K\,\frac{\omega(s,t)^{\frac{i}{p}}}{\beta( p)\cdot(\frac{i}{p})!}
= K\,\frac{\omega(s,t)^{\frac{i}{p^\ast}}}{\beta( p)\cdot(\frac{i}{p^\ast})!}
\le K\,\frac{\omega^\ast(s,t)^{\frac{i}{p^\ast}}}{\beta( p)\cdot(\frac{i}{p^\ast})!}.
\]
Thus $X$ has finite $p^\ast$-variation controlled by $\omega^\ast$.

Next, let $\Delta(X)$ denote the multiplicative defect of $X$. Since $X\in \Omega^{\text{am},(k,p)}_{ \theta, \omega}(V)$, we have
\[
\|\Delta(X)^i_{s,u,t}\|\le K\,\omega(s,t)^{\theta},
\qquad s\le u\le t,\ \ i=1,\dots,k.
\]
Since $\theta\ge \theta^\ast$ and $\omega(s,t)\le \omega^\ast(S,T)$, we obtain
\[
\omega(s,t)^{\theta}
\le \omega^\ast(S,T)^{\theta-\theta^\ast}\,\omega(s,t)^{\theta^\ast}
\le \omega^\ast(S,T)^{\theta-\theta^\ast}\,\omega^\ast(s,t)^{\theta^\ast}.
\]
Hence $\|\Delta(X)^i_{s,u,t}\|\le K^\ast\,\omega^\ast(s,t)^{\theta^\ast}$ for $K^\ast:=K\,\omega^\ast(S,T)^{\theta-\theta^\ast}$, and therefore
\[
X\in \Omega^{\text{am},(k,p^\ast)}_{\theta^\ast,\omega^\ast}(V).
\]

\medskip
\noindent\textbf{\boldmath\underline{Step 2: Compare sewings via $(p^\ast,\omega^\ast,\theta^\ast)$}}

Define
\[
\bar X := \mathscr{S}^{(k,p)}_{ \theta, \omega}(X),
\qquad
\tilde X := \mathscr{S}^{(k,\tilde p)}_{\tilde\theta,\tilde\omega}(X),
\qquad
\hat X := \mathscr{S}^{(k,p^\ast)}_{\theta^\ast,\omega^\ast}(X),
\]
where $\hat X$ is well-defined by Step~1.

By Theorem \ref{thm:Rough-Sewing-Lemma}, $\bar X,\tilde X,\hat X$ are multiplicative and there exist
constants $C,\tilde C,\hat C\ge 0$ such that for all $(s,t)\in\triangle_J$ and $i=1,\dots,k$,
\[
\|\bar X^i_{s,t}-X^i_{s,t}\|\le C\,\omega(s,t)^{\theta},
\hspace{0.1cm}
\|\tilde X^i_{s,t}-X^i_{s,t}\|\le \tilde C\,\tilde\omega(s,t)^{\tilde\theta},
\hspace{0.1cm}
\|\hat X^i_{s,t}-X^i_{s,t}\|\le \hat C\,\omega^\ast(s,t)^{\theta^\ast}
\]
Using $\omega\le \omega^\ast$, $\tilde\omega\le \omega^\ast$, and $\theta,\tilde\theta\ge\theta^\ast$ (as in Step~1), we obtain bounds of the form
\[
\|\bar X^i_{s,t}-X^i_{s,t}\|\le C_1\,\omega^\ast(s,t)^{\theta^\ast},
\qquad
\|\tilde X^i_{s,t}-X^i_{s,t}\|\le C_2\,\omega^\ast(s,t)^{\theta^\ast},
\]
for suitable constants $C_1,C_2\ge0$.

We know that since we assumed $p^\ast=p$ without loss of generality, $\bar X$ is an element of $\Omega^{(k,p^\ast)}_{\omega^\ast}(V)$. Now, we would like to show that $\tilde  X$ has finite $p^\ast$-variation. Given that we assumed $p^\ast=p$ without loss of generality, either (a) $\tilde p=p$ (in which case the reasoning is immediate) or (b) $p>\tilde p$. In case (b), by the rough sewing lemma (\ref{thm:Rough-Sewing-Lemma}), since $X\in \Omega^{\text{am},(k,\tilde p)}_{\tilde\theta, \tilde\omega}(V)$, there exists a constant $\widetilde K\ge0$ such that $\tilde X= \mathscr{S}^{(k,\tilde p)}_{\tilde\theta,\tilde\omega}(X)$ has finite $\tilde p $-variation controlled by $\widetilde K\tilde\omega$. With this, we have that, for each $i=1,\dots,k$ and $(s,t)\in\triangle_J$,
\[
\|\tilde X^i_{s,t}\|
\le \widetilde K\,\frac{\omega^\ast(s,t)^{\frac{i}{\tilde p}}}{\beta(\tilde p)\cdot(\frac{i}{\tilde p})!} \le\widetilde K \omega^\ast (S,T)^{\frac{i}{\tilde p}- \frac{i}{ p}}\,\frac{\beta( p)\cdot(\frac{i}{ p})!}{\beta(\tilde p)\cdot(\frac{i}{\tilde p})!} \frac{\omega^\ast(s,t)^{\frac{i}{ p}}}{\beta(p)\cdot(\frac{i}{ p})!} \le C_0\frac{\omega^\ast(s,t)^{\frac{i}{ p}}}{\beta(p)\cdot(\frac{i}{ p})!}
    .
\]

where $C_0= \widetilde K\max_{1\le i\le k}\{  \omega^\ast (S,T)^{\frac{i}{\tilde p}- \frac{i}{ p}}\,\frac{\beta( p)\cdot(\frac{i}{ p})!}{\beta(\tilde p)\cdot(\frac{i}{\tilde p})!} \}$. Hence, in both cases (a) and (b), setting $C^*=\max\{\widetilde K,C_0\}$, we have that
$$\|\tilde  X^i_{s,t}\|\le C^*\frac{\omega^\ast(s,t)^{\frac{i}{ p^\ast}}}{\beta( p^\ast)\cdot(\frac{i}{p^\ast})!}.$$
Therefore, $\bar X$ and $\tilde X$ are both elements of $\Omega^{(k,p^\ast)}_{\omega^\ast}(V)$ satisfying the sewing-closeness condition to $X$ with respect to $(\omega^\ast,\theta^\ast)$. By the \emph{uniqueness} part of Theorem \ref{thm:Rough-Sewing-Lemma} applied with $(p^\ast,\omega^\ast,\theta^\ast)$, we conclude
\[
\bar X=\hat X=\tilde X,
\]
which is exactly $\mathscr{S}^{(k,p)}_{ \theta, \omega}(X)= \mathscr{S}^{(k,p^\ast)}_{\theta^\ast,\omega^\ast}(X)=\mathscr{S}^{(k,\tilde p)}_{\tilde \theta,\tilde \omega}(X)$.
\end{proof}

\begin{definition}[Global sewing map $\mathscr{S}^{k}$]\label{def:global-sewing}
Fix an integer $k\ge1$. Define
\[
\Omega^{\text{am},(k)}(V):=\bigcup_{\substack{p\ge 1\\\omega\ \text{control}\\ \theta>1}}\Omega^{\text{am},(k,p)}_{ \theta, \omega}(V) \quad \text{ and }\quad \Omega^{ (k)}(V):=\bigcup_{\substack{p\ge 1\\\omega\ \text{control}}}\Omega^{ (k,p)}_{\omega}(V) .
\]For $X\in\Omega^{\text{am},(k)}(V)$, choose any triple $(p, \theta, \omega)$ such that $X\in\Omega^{\text{am},(k,p)}_{ \theta, \omega}(V)$ and define $\mathscr{S}^{k}:\Omega^{\text{am},(k)}(V)\to \Omega^{ (k)}(V)$ given by
\[
\mathscr{S}^{k}(X):=\mathscr{S}^{(k,p)}_{ \theta, \omega}(X),
\]
where $\mathscr{S}^{(k,p)}_{ \theta, \omega}$ denotes the sewing map of Definition~\ref{def: sewing map mathscr S} (with the chosen parameters $\theta$ and $\omega$) applied to $X$.
\end{definition}

\begin{proposition}[Well-definedness of $\mathscr{S}^{k}$]\label{prop:global-sewing-well-defined}
For any integer $k\ge1$, map $\mathscr{S}^{k}:\Omega^{\text{am},(k)}(V)\to \Omega^{ (k)}(V)$ in Definition~\ref{def:global-sewing} is well-defined, i.e. it does not depend on the choice of admissible triple $(p,  \theta, \omega)$ used to represent $X$. 
\end{proposition}

\begin{proof}
Let $X\in\Omega^{\text{am},(k)}(V)$ and suppose that $X\in\Omega^{\text{am},(k,p)}_{ \theta, \omega}(V)$ and also $X\in\Omega^{\text{am},(k,\tilde p)}_{\tilde\theta, \tilde\omega}(V)$ for some parameters $p,\tilde p\ge 1$, controls $\omega,\tilde\omega$ and exponents $\theta,\tilde\theta>1$. By Lemma~\ref{lem:control-independence-sewing} we have $\mathscr{S}^{(k,p)}_{\theta,\omega}(X)=\mathscr{S}^{(k,\tilde p)}_{\tilde\theta,\tilde\omega}(X)$. Therefore the definition of $\mathscr{S}^{k}$ is independent of the choice of admissible parameters.
\end{proof}

\begin{remark}\label{rmk: if X^j =Y^j for all j, then S(X)^j = S(Y)^j for all j}
Upon examining the construction detailed in $\bar X$ as presented in Theorem \ref{thm:Rough-Sewing-Lemma}, it becomes evident that, for any given $j\in\{1,..., k-1\}$, $\bar X^j$ depends upon $X^1$ to $X^j$ but is independent of $X^{j+1}$ to $X^{k}$. This observation is crucial in demonstrating that if two almost $p$-rough paths, denoted as $X$ and $Y$, coincide up to a certain level $j$ for some $j\in\{1,..., k-1\}$, then their corresponding multiplicative functionals $\mathscr{S}^k(Y)$ and $\mathscr{S}^k(X)$ will similarly coincide up to level $j$.
\end{remark}

\begin{definition}[Bracket truncations]\label{def:truncations}
Fix \(p\ge1\) and a functional \(Y=(Y^0,Y^1,\dots):\triangle_J\!\to T^{(\lfloor p\rfloor)}(V)\).
For each \(k\in\{0,\dots,\lfloor p\rfloor\}\) set
\[
  Y(k):=(Y^0,\dots,Y^k)\in T^{(k)}(V),
  \qquad
  Y[k]:=(Y^0,\dots,Y^k,0,\dots,0)\in T^{(\lfloor p\rfloor)}(V),
\]
where \(Y[k]\) has \(\lfloor p\rfloor-k\) trailing zeros. \(Y(k)\) is the  truncation to level \(k\); \(Y[k]\) is the zero-padded version that stays in the level-\(\lfloor p\rfloor\) algebra.
\end{definition}

\begin{corollary}\label{cor:X_Y_close_implies_SX_eq_SY}
Fix real numbers $p\ge 1$ and $\theta>1$ and an integer $k\in\mathbb N$, and let $\omega:\triangle_J\to[0,\infty)$ be a control. Suppose $X=(1,X^1,\dots,X^{k}):\triangle_J \to T^{(k)}(V)$ is a $\theta$-almost multiplicative functional of degree $k$ with finite $p$-variation controlled by $\omega$ with constant $C_X\ge0$. Let $Y=(1,Y^1,\dots,Y^{k}):\triangle_J \to T^{(k)}(V)$ be a functional such that
\begin{equation}\label{ineq: ineq in X and Y close => S(X)=S(Y) corollary}
\|Y_{st}^{j}-X_{st}^{j}\|
\;\le\;K\,\omega(s,t)^{\theta},
\qquad
\forall\,(s,t)\in\triangle_J,\;
j=1,\dots,k.
\end{equation}
Assume moreover that there exists a constant $K_{\mathrm{hi}}\ge0$ such that
\begin{equation}\label{eq:hi-level-pvar-assumption}
\|Y_{st}^{j}\|
\;\le\;
K_{\mathrm{hi}}\,
\frac{\omega(s,t)^{\frac{j}{p}}}{\beta(p)\left(\frac{j}{p}\right)!},
\qquad
\forall\,(s,t)\in\triangle_J,\;
\forall\,j\in\{1,\dots,k\}\ \text{with}\ \frac{j}{p}>\theta.
\end{equation}
(\emph{This condition is vacuous if there is no $j\in\{1,\dots,k\}$ with $\frac{j}{p}>\theta$}.)

Then $Y$ is itself a $\theta$-almost multiplicative functional of degree $k$ with finite $p$-variation controlled by $\omega$ with constant $C_Y$ for some constant $C_Y\ge0$, and its sewing coincides with that of $X$:
\[
\mathscr{S}^k(Y)=\mathscr{S}^k(X).
\]
\end{corollary}

\vspace{2cm}
\begin{proof}
\textbf{\underline{Step 1:  $Y$ is $\theta$-almost multiplicative}}

Fix $s\le u\le t$ in $J$ and $j\in\{1,\dots,k\}$, and set
\[
\varepsilon_{ab}:=Y_{ab}-X_{ab},\qquad a<b .
\]
Since $X$ is $\theta$-almost multiplicative, there exists $C_0\ge0$ such that
\[
\|(X_{su}\otimes X_{ut}-X_{st})^j\|\le C_0\,\omega(s,t)^{\theta}.
\]
Moreover,
\[
( Y_{su}\otimes Y_{ut}-Y_{st})^j
=\bigl(X_{su}\otimes X_{ut}-X_{st}\bigr)^j
+\bigl(X_{su}\otimes\varepsilon_{ut}
+\varepsilon_{su}\otimes X_{ut}
+\varepsilon_{su}\otimes\varepsilon_{ut}
-\varepsilon_{st}\bigr)^j.
\]
The norm of the first term is bounded by \(C_0\,\omega(s,t)^{\theta}\), as required. For a mixed tensor \(X_{su}^{(a)}\otimes\varepsilon_{ut}^{(b)}\)   with \(a+b\le k\), we have that (using admissibility of the tensor norm, finite $p$-variation of $X$, the bound $\|\varepsilon_{ab}^m\|\le K\omega(a,b)^\theta$, and $\omega(a,b)\le \omega(S,T)$)\[
      \|X_{su}^{(a)}\otimes\varepsilon_{ut}^{(b)}\|
      \le C_XK\,\omega(s,u)^{\frac{a}p}\omega(u,t)^{\theta}
      \le C_XK\,\omega(s,t)^{\frac{a}p+\theta}
      \le C_X\tilde{K}\,\omega(s,t)^{\theta}
  \]
since \(\frac{a}p\ge0\), and where $\tilde{K}:= K\max_{a=1,...,k}\omega(S,T)^{\frac{a}p}$. Each of the remaining terms \textendash$\varepsilon_{su}\otimes X_{ut}$, $\varepsilon_{su}\otimes\varepsilon_{ut}$, and $\varepsilon_{st}$ \textendash also satisfies a bound of the form $C_{\text{term}}\;\omega(s,t)^{\theta}$, with a (term-dependent) constant $C_{\text{term}}$. Hence there exists $C_{\mathrm{am}}\ge0$ such that
\[
\|(Y_{su}\otimes Y_{ut}-Y_{st})^j\|\le C_{\mathrm{am}}\,\omega(s,t)^{\theta},
\qquad j=1,\dots,k.
\]
\medskip        
\textbf{\boldmath\underline{Step 2:  Finite $p$-variation of $Y$}}

Fix $(s,t)\in\triangle_J$ and $j\in\{1,\dots,k\}$. If $\frac{j}{p}>\theta$, then \eqref{eq:hi-level-pvar-assumption} gives directly\[
\|Y_{st}^{j}\|
\le
K_{\mathrm{hi}}\,
\frac{\omega(s,t)^{\frac{j}{p}}}{\beta(p)\left(\frac{j}{p}\right)!}.
\]
If instead $\frac{j}{p}\le\theta$, then by the triangle inequality,
\[
\|Y_{st}^j\|\le \|X_{st}^j\|+\|Y_{st}^j-X_{st}^j\|
\le \frac{\omega(s,t)^{\frac{j}{p}}}{\beta(p)\left(\frac{j}{p}\right)!}
+K\,\omega(s,t)^\theta.
\]
Since $\theta-\frac{j}{p}\ge0$ in this case and $\omega(s,t)\le \omega(S,T)$,
\[
\omega(s,t)^\theta
=\omega(s,t)^{\frac{j}{p}}\omega(s,t)^{\theta-\frac{j}{p}}
\le \omega(S,T)^{\theta-\frac{j}{p}}\,\omega(s,t)^{\frac{j}{p}}.
\]
Therefore
\[
\|Y_{st}^j\|
\le
\Bigl(1+K\,\beta(p)\Bigl(\tfrac{j}{p}\Bigr)!\,\omega(S,T)^{\theta-\frac{j}{p}}\Bigr)
\frac{\omega(s,t)^{\frac{j}{p}}}{\beta(p)\left(\frac{j}{p}\right)!}.
\]
Taking
\[
C_{\mathrm{var}}
:=
\max\!\left\{
K_{\mathrm{hi}},\;
\max_{\substack{1\le j\le k\\ \frac{j}{p}\le \theta}}
\Bigl(1+K\,\beta(p)\Bigl(\tfrac{j}{p}\Bigr)!\,\omega(S,T)^{\theta-\frac{j}{p}}\Bigr)
\right\},
\]
we get for all $j=1,\dots,k$ and all $(s,t)\in\triangle_J$,
\[
\|Y_{st}^{j}\|
\le
C_{\mathrm{var}}\,
\frac{\omega(s,t)^{\frac{j}{p}}}{\beta(p)\left(\frac{j}{p}\right)!},
\]
so $Y$ has finite $p$-variation controlled by $C_{\mathrm{var}}\omega$.

\medskip
\textbf{\underline{Step 3:  Coincidence of sewings}}

Let $\bar X:=\mathscr{S}^k(X)$. By the rough sewing lemma applied to $X$, there exists $C_{\mathscr S}\ge0$ such that for all $(s,t)\in\triangle_J$ and $i=1,\dots,k$,
\[
\|\bar X_{st}^i-X_{st}^i\|\le C_{\mathscr S}\,\omega(s,t)^\theta.
\]
Using \eqref{ineq: ineq in X and Y close => S(X)=S(Y) corollary},
\[
\|\bar X_{st}^i-Y_{st}^i\|
\le \|\bar X_{st}^i-X_{st}^i\|+\|X_{st}^i-Y_{st}^i\|
\le (C_{\mathscr S}+K)\,\omega(s,t)^\theta.
\]
Now Step~1 and Step~2 ensure $Y$ lies in the domain of the sewing lemma, so $\bar Y:=\mathscr{S}^k(Y)$ is well-defined and is the unique multiplicative functional satisfying an $\omega^\theta$-closeness bound to $Y$. Since $\bar X$ is
multiplicative and satisfies such a bound, uniqueness gives $\bar Y=\bar X$, i.e.
\[
\mathscr{S}^k(Y)=\mathscr{S}^k(X).
\]
Finally set $C_Y:=\max\{C_{\mathrm{am}},C_{\mathrm{var}}\}$.
\end{proof}

\subsection{\textit{H}-space perturbations}
We now define unit-preserving pointwise combinations and isolate the class of perturbations ($H$-space) for which sewing yields a canonical perturbed rough path.

\begin{definition}{\normalfont\textbf{(Point-wise tensor product and unit-preserving pointwise sum)}}\label{def: X otimes H}
Let $n\in\mathbb N\cup\{\infty\}$. Let $X=(X^0,...,X^n):\triangle_J\to T^{(n)}(V)$ and $H=(H^0,...,H^n):\triangle_J\to T^{(n)}(V)$ be functionals. We define $X \otimes H: \triangle_J \to T^{(n)}(V)$ by $$
	(X \otimes H)_{s t}=X_{s t} \otimes H_{s t}, \qquad (s,t)\in\triangle_J.
	$$
Assume further that $X^0=1=H^0$. We then define $X\oplus H:\triangle_J\to \widetilde T^{(n)}(V)$ by
\[
(X\oplus H)_{s,t} := X_{s,t}\oplus H_{s,t}, \qquad (s,t)\in\triangle_J.
\]
\end{definition}
The ensuing space will be integral to our study. We present its definition in what follows:

  
\begin{definition}[$\mathscr{H}^{(k,p)}(V)$ and $\mathscr{H}_{\phi,\omega}^{(k,p)}(V)$] \label{def:h-space}
	Let $p \geqslant 1$ be a real number, $k \geqslant 0$ be an integer, $\omega$ be a control, and $\phi \in \bigl(1-\tfrac{1}{p}, 1\bigr]$. 
	If $k\in\mathbb N$, we set    
	\begin{multline*}
		\mathscr{H}^{(k,p)}(V) := \left\{ 
		\vphantom{\mathcolorbox{yellow}{\widetilde{\omega}(s,t)^{\widetilde{\phi}}}}
		H \in  \Omega^{(k,p)}(V) \,\left|\, \exists\, K_H\ge0,\exists \hspace{-1pt} \,\operatorname{ control} \, \omega_H \, \operatorname{and} \, \exists \,  \phi_H \in \bigl( 1-\tfrac{1}{p}, 1 \bigr]: \right. \right. \\
		\left. \left( \forall j \in \left\{ 1, \ldots, k \right\}, \forall(s,t) \in \triangle_J, \,\left\| H_{st}^j \right\| \leqslant K_H \omega_H(s,t)^{\phi_H} \right) \right\},
	\end{multline*}
\begin{multline*}
	\mathscr{H}_{\phi,\omega}^{(k,p)}(V) := \Bigl\{ 
	\vphantom{\mathcolorbox{yellow}{\omega_H(s,t)^{\widetilde{\phi}}}}
	H \in  \Omega^{(k,p)}_\omega(V)\, |\,  \exists\, K_H\ge0 : \\
    \left( \forall j \in \left\{ 1, \ldots, k \right\}, \forall(s,t) \in \triangle_J,  
   \left\| H_{st}^j \right\| \leqslant K_H\omega(s,t)^{\phi} 
    \vphantom{K_H\omega_H(s,t)^{\widetilde{\phi}}}
    \right) \Bigr \}.
\end{multline*}

If $k=0$, then $\mathscr{H}^{(k,p)}(V)= \mathscr{H}_{\phi,\omega}^{(k,p)}(V) := \{ (1) \}$. We set $\mathscr{H}^p(V):=\mathscr{H}^{(\lfloor p \rfloor, \,p)}$.

\end{definition} 

We introduce more notation.

\begin{definition} We set
\[
  \mathscr{H}^{(k,p)}_{\omega}(V)
    := \bigcup_{\varphi\in(1-1/p,\,1]} \mathscr{H}^{(k,p)}_{\varphi,\omega}(V),
  \qquad
  \mathscr{H}^{(k,p)}_{\varphi}(V)
    := \bigcup_{\omega\ \text{control}} \mathscr{H}^{(k,p)}_{\varphi,\omega}(V).
\]
\end{definition}

\begin{remark} Define $\mathscr{H}(V):=\cup_{p\ge1} \mathscr{H}^p(V)$. We refer to $\mathscr{H}(V)$ as ``$H$-space''. \end{remark}

\begin{example}[Pure area rough path] \label{ex:pure-area-rp}Let $H=(1,0,H^2):\triangle_J\to T^{(2)}(V)$ be a functional and assume that $H^2:\triangle_J\to V^{\otimes 2}$ satisfies 
$$\|H^2_{st}\|\le K \frac{\omega(s,t)^{\frac{2}{p}}}{\beta\cdot(\frac{2}{p})!}\qquad \text{for all $(s,t)\in\triangle_J$},$$
where $p\in[2,3)$, $\omega$ is a control, and $K\ge0$. Then, $H$ has finite $p$-variation. If we assume that $H^2$ is additive, i.e. $H_{s,u}+H_{u,t}=H_{s,t}$ for all  $s,u,t\in \triangle_J$ with $s\le u \le t$, then $H$ is also multiplicative.  In this case, $H$ is a $p$-rough path. Since $H$ is a functional of degree $2$  (i.e., it takes values in $T^{(2)}(V)$) and its first level $H^1$ is equal to zero, we say that $H$ is a pure area rough path. 

Notice that if we set $2/p =:\phi$, then we have that $1-1/p < 2/p \iff p<3$ and $2/p \le 1 \iff  2\le p$. Since we assumed that $p\in[2,3)$, we can conclude that $\phi\in( 1-1/p, 1] $. If we set $\widetilde K = K/(\beta\cdot (2/p)!)$, then it follows that $H\in  \mathscr{H}_{\phi,\omega}^{2}(V)$.  

This example shows that $\mathscr{H}^2(V)$ is strictly larger than the space of degree-$2$ multiplicative functionals of finite $q$-variation for $q\in[1,2)$, which is the regime that has been most extensively studied in the literature of rough path perturbations (see Section \ref{sec:introduction} for greater detail).
\end{example}

An earlier version of this theorem features originally in \cite{lyons1998differential} as Remark ``B''.
\begin{theorem} \label{thm:X_otimes_H_is_ARP}
Let $p\in \mathbb R_{\ge1}$, $k\in\mathbb N$. Take $\phi\in(1-\frac{1}{p},1]$  and let $\omega$ be a control. Consider an  $H\in \mathscr{H}^{{(k,p)}}_{\phi,\omega}(V)$. Consider any $X \in \Omega^{(k,p)}_\omega(V)$. Then the maps  $X\oplus H$ and $X \otimes H $ belong to $ \Omega_{\phi+\frac1p,\omega}^{\text{am},(k,p)}(V)$, i.e. they are $\phi + \frac{1}{p}$-almost multiplicative functionals of finite $p$-variation respectively controlled by $\omega$ with constant $K$ and $\omega$ with constant $\tilde K$ for some constants $K, \tilde K \geqslant 0$. We then define $$
	X \boxplus H :=  \mathscr{S}^k(X \oplus H)\in \Omega_\omega^{(k,p)}(V).
	$$
Furthermore, $\mathscr{S}^k(X \otimes H) =X\boxplus H $.
\end{theorem}

\begin{proof}
Set
\[
\theta := \phi+\frac1p>1,
\qquad
\beta:=\beta(p)
\ \text{ as in Definition~\ref{def:beta}.}
\]

\medskip

\noindent\textbf{\underline{Step 0: Fix constants}}

Since $X\in\Omega^{(k,p)}_{\omega}(V)$, by Definition~\ref{def:finite p-var} there exists $K_X\ge0$ such that
\begin{equation}\label{eq:KX-bound-new}
\|X^j_{s,t}\|
\le
K_X\,\frac{\omega(s,t)^{\frac{j}{p}}}{\beta\left(\frac{j}{p}\right)!},
\qquad
j=1,\dots,k,\ (s,t)\in\triangle_J.
\end{equation}
By the hypothesis on $H$, there exists $K_H\ge0$ such that
\begin{equation}\label{eq:KH-phi-bound-new}
\|H^j_{s,t}\|
\le
K_H\,\omega(s,t)^{\phi},
\qquad
j=1,\dots,k,\ (s,t)\in\triangle_J.
\end{equation}

\medskip

\noindent\textbf{\boldmath\underline{Step 1: $X\oplus H$ is $\theta$-almost multiplicative of degree $k$}}

Fix $m\in\{1,\dots,k\}$ and $s\le u\le t$. Write the level-$m$ multiplicative defect:
\[
\Delta^m_{s,u,t}(X\oplus H)
:=
\sum_{i=0}^m (X\oplus H)^i_{s,u}\otimes (X\oplus H)^{m-i}_{u,t}
-
(X\oplus H)^m_{s,t}
\]
Since $(X\oplus H)^0\equiv 1$ and $(X\oplus H)^j=X^j+H^j$ for $j\ge1$, we expand
\begin{align*}
&\sum_{i=0}^m (X\oplus H)^i_{s,u}\otimes (X\oplus H)^{m-i}_{u,t}\\
=\;&
(X^m_{s,u}+H^m_{s,u})+(X^m_{u,t}+H^m_{u,t})
+\sum_{i=1}^{m-1}(X^i_{s,u}+H^i_{s,u})\otimes(X^{m-i}_{u,t}+H^{m-i}_{u,t})
\\
=\;&
\sum_{i=0}^m X^i_{s,u}\otimes X^{m-i}_{u,t}
+\sum_{i=0}^m H^i_{s,u}\otimes H^{m-i}_{u,t} \\[4pt]
&\quad\qquad
+\sum_{i=1}^{m-1} X^i_{s,u}\otimes H^{m-i}_{u,t}
+\sum_{i=1}^{m-1} H^i_{s,u}\otimes X^{m-i}_{u,t}.
\end{align*}
Because $X$ and $H$ are multiplicative (Definition~\ref{def: multiplicative functional}),
\[
\sum_{i=0}^m X^i_{s,u}\otimes X^{m-i}_{u,t}=X^m_{s,t},
\qquad
\sum_{i=0}^m H^i_{s,u}\otimes H^{m-i}_{u,t}=H^m_{s,t}.
\]
Since $(X\oplus H)^m_{s,t}=X^m_{s,t}+H^m_{s,t}$, we obtain the cancellation identity
\begin{equation}\label{eq:defect-XoplusH-new}
\Delta^m_{s,u,t}(X\oplus H)
=
\sum_{i=1}^{m-1}\Bigl(
X^i_{s,u}\otimes H^{m-i}_{u,t}
+
H^i_{s,u}\otimes X^{m-i}_{u,t}
\Bigr).
\end{equation}
In particular, every term contains an $H$-factor of positive level.

We now bound \eqref{eq:defect-XoplusH-new}. By admissibility (Definition~\ref{def: admissible}),
\begin{align*}
\|X^i_{s,u}\otimes H^{m-i}_{u,t}\|
&\le \|X^i_{s,u}\|\,\|H^{m-i}_{u,t}\|,
&
\|H^i_{s,u}\otimes X^{m-i}_{u,t}\|
&\le \|H^i_{s,u}\|\,\|X^{m-i}_{u,t}\|.
\end{align*}
Using \eqref{eq:KX-bound-new}--\eqref{eq:KH-phi-bound-new} and $\omega(s,u)\le\omega(s,t)$, $\omega(u,t)\le\omega(s,t)$ (Definition~\ref{def:controls}), we get for $i\in\{1,\dots,m-1\}$:
\begin{align*}
\|X^i_{s,u}\otimes H^{m-i}_{u,t}\|
&\le
K_X K_H\,
\frac{\omega(s,u)^{\frac{i}{p}}}{\beta\left(\frac{i}{p}\right)!}\,
\omega(u,t)^{\phi}
\le
K_X K_H\,
\frac{\omega(s,t)^{\frac{i}{p}+\phi}}{\beta\left(\frac{i}{p}\right)!},
\\
\|H^i_{s,u}\otimes X^{m-i}_{u,t}\|
&\le
K_H K_X\,
\omega(s,u)^{\phi}\,
\frac{\omega(u,t)^{\frac{m-i}{p}}}{\beta\left(\frac{m-i}{p}\right)!}
\le
K_X K_H\,
\frac{\omega(s,t)^{\phi+\frac{m-i}{p}}}{\beta\left(\frac{m-i}{p}\right)!}.
\end{align*}
Now, we factor out $\omega(s,t)^\theta$ using the fact that the leftover exponents are nonnegative: since $i\ge 1$ we have $(i-1)/p\ge 0$, and since $i\le m-1$ we have $(m-i-1)/p\ge 0$, so
\begin{align*}
\omega(s,t)^{\frac{i}{p}+\phi}
&=
\omega(s,t)^\theta\,\omega(s,t)^{\frac{i-1}{p}}
\le
\omega(s,t)^\theta\,\omega(S,T)^{\frac{i-1}{p}},
\end{align*}
and
\begin{align*}
\omega(s,t)^{\phi+\frac{m-i}{p}}
&=
\omega(s,t)^\theta\,\omega(s,t)^{\frac{m-i-1}{p}}
\le
\omega(s,t)^\theta\,\omega(S,T)^{\frac{m-i-1}{p}}.
\end{align*}
Combining with \eqref{eq:defect-XoplusH-new} yields
\[
\|\Delta^m_{s,u,t}(X\oplus H)\|
\le
C_m\,\omega(s,t)^\theta,
\qquad
s\le u\le t,
\]
where 
\begin{equation}\label{eq:Cm-explicit-new}
C_m
:=
\frac{K_XK_H}{\beta}
\left(
\sum_{i=1}^{m-1}\frac{\omega(S,T)^{\frac{i-1}{p}}}{\left(\frac{i}{p}\right)!}
+\sum_{i=1}^{m-1}
\frac{\omega(S,T)^{\frac{m-i-1}{p}}}{\left(\frac{m-i}{p}\right)!}
\right)
\end{equation}
Setting $C:=\max_{1\le m\le k}C_m$, Definition~\ref{def: almost multiplicativity} gives that $X\oplus H$ is $\theta$-almost multiplicative of degree $k$, controlled by $\omega$ with constant $C$.

\medskip

\noindent\textbf{\boldmath\underline{Step 2: Degree-wise $\omega^\theta$-closeness of $X\oplus H$ and $X\otimes H$}}

Fix $m\in\{1,\dots,k\}$ and $(s,t)\in\triangle_J$. By the product in $T^{(k)}(V)$,
\[
(X\otimes H)^m_{s,t}=\sum_{i=0}^m X^i_{s,t}\otimes H^{m-i}_{s,t},
\qquad
(X\oplus H)^m_{s,t}=X^m_{s,t}+H^m_{s,t}.
\]
Hence
\[
\bigl((X\oplus H)-(X\otimes H)\bigr)^m_{s,t}
=
-\sum_{i=1}^{m-1}X^i_{s,t}\otimes H^{m-i}_{s,t}.
\]
Using admissibility and \eqref{eq:KX-bound-new}--\eqref{eq:KH-phi-bound-new},
\[
\bigl\|\bigl((X\oplus H)-(X\otimes H)\bigr)^m_{s,t}\bigr\|
\le
\sum_{i=1}^{m-1}
K_XK_H\,\frac{\omega(s,t)^{\frac{i}{p}+\phi}}{\beta\left(\frac{i}{p}\right)!}.
\]
As in Step~1, since $i\ge 1$ we may write
\[
\omega(s,t)^{\frac{i}{p}+\phi}
=\omega(s,t)^\theta\,\omega(s,t)^{\frac{i-1}{p}}
\le
\omega(s,t)^\theta\,\omega(S,T)^{\frac{i-1}{p}}.
\]
Therefore there exists a constant $K^{\mathrm{cl}}_m\ge0$ such that
\[
\bigl\|\bigl((X\oplus H)-(X\otimes H)\bigr)^m_{s,t}\bigr\|
\le
K^{\mathrm{cl}}_m\,\omega(s,t)^\theta,
\]
and explicitly one can take
\[
K^{\mathrm{cl}}_m
:=
\frac{K_XK_H}{\beta}
\sum_{i=1}^{m-1}\frac{\omega(S,T)^{\frac{i-1}{p}}}{\left(\frac{i}{p}\right)!}.
\]
With $K^{\mathrm{cl}}:=\max_{1\le m\le k}K^{\mathrm{cl}}_m$, we have the uniform closeness
\begin{equation}\label{eq:close-new}
\bigl\|(X\otimes H)^m_{s,t}-(X\oplus H)^m_{s,t}\bigr\|
\le
K^{\mathrm{cl}}\omega(s,t)^\theta,
\qquad
m=1,\dots,k.
\end{equation}

\medskip

\noindent\textbf{\underline{Step 3: Compare sewings}}

Recall that $H\in\mathscr{H}^{(k,p)}_\omega(V)$ implies $H\in\Omega^{(k,p)}_\omega(V)$ by Definition~\ref{def:h-space}. Then by Definition~\ref{def:finite p-var} there exists $K_{H,\mathrm{var}}\ge0$ such that
\begin{equation}\label{eq:KH-var-new}
\|H^j_{s,t}\|
\le
K_{H,\mathrm{var}}\,
\frac{\omega(s,t)^{\frac{j}{p}}}{\beta\left(\frac{j}{p}\right)!},
\qquad j=1,\dots,k.
\end{equation}

\smallskip

\noindent (a) \emph{Finite $p$-variation of $X\oplus H$.}
By $(X\oplus H)^j=X^j+H^j$ for $j\ge1$,
\[
\|(X\oplus H)^j_{s,t}\|
\le
\|X^j_{s,t}\|+\|H^j_{s,t}\|
\le
(K_X+K_{H,\mathrm{var}})\,
\frac{\omega(s,t)^{\frac{j}{p}}}{\beta\left(\frac{j}{p}\right)!}.
\]
So $X\oplus H$ has finite $p$-variation controlled by $\omega$ with constant $K_X+K_{H,\mathrm{var}}$.

\smallskip

\noindent (b) \emph{Finite $p$-variation of $X\otimes H$.}
Fix $m\in\{1,\dots,k\}$ and $(s,t)\in\triangle_J$. Since
\[
(X\otimes H)^m_{s,t}=\sum_{i=0}^m X^i_{s,t}\otimes H^{m-i}_{s,t},
\]
we have, by admissibility (Definition~\ref{def: admissible}) and \eqref{eq:KX-bound-new}, \eqref{eq:KH-var-new},
\begin{align*}
\|(X\otimes H)^m_{s,t}\|
&\le \|X^m_{s,t}\|+\|H^m_{s,t}\|+\sum_{i=1}^{m-1}\|X^i_{s,t}\|\,\|H^{m-i}_{s,t}\|\\
&\hspace{-1cm}\le (K_X+K_{H,\mathrm{var}})\frac{\omega(s,t)^{\frac{m}{p}}}{\beta\left(\frac{m}{p}\right)!}
+\frac{K_XK_{H,\mathrm{var}}}{\beta^2}\,\omega(s,t)^{\frac{m}{p}}
\sum_{i=1}^{m-1}\frac{1}{\left(\frac{i}{p}\right)!\left(\frac{m-i}{p}\right)!}.
\end{align*}
By the Neo-classical Inequality applied with $s=t=\omega(s,t)$ and $n=m$,
\[
\sum_{i=0}^m\frac{1}{\left(\frac{i}{p}\right)!\left(\frac{m-i}{p}\right)!}
\le \frac{p\,2^{\frac{m}{p}}}{\left(\frac{m}{p}\right)!},
\]
hence the sum over $i=1,\dots,m-1$ is bounded by the same right-hand side. Therefore,

\begin{align}
\|(X\otimes H)^m_{s,t}\|
&\le \left(K_X+K_{H,\mathrm{var}}+\frac{K_XK_{H,\mathrm{var}}\,p}{\beta}\,2^{\frac{m}{p}}\right)
\frac{\omega(s,t)^{\frac{m}{p}}}{\beta\left(\frac{m}{p}\right)!}
\nonumber\\
&\le K_{\otimes,\mathrm{var}}\frac{\omega(s,t)^{\frac{m}{p}}}{\beta\left(\frac{m}{p}\right)!}.
\label{eq:XotimesH-pvar-new}
\end{align}

where one may take
\[
K_{\otimes,\mathrm{var}}
:=K_X+K_{H,\mathrm{var}}+\frac{K_XK_{H,\mathrm{var}}\,p}{\beta}\,2^{\frac{k}{p}}.
\]
Since $m\in\{1,\dots,k\}$ was arbitrary, \eqref{eq:XotimesH-pvar-new} holds for all $m=1,\dots,k$.
In particular, the “high-level $p$-variation” hypothesis \eqref{eq:hi-level-pvar-assumption} in Corollary~\ref{cor:X_Y_close_implies_SX_eq_SY} holds for $Y=X\otimes H$ with $K_{\mathrm{hi}}=K_{\otimes,\mathrm{var}}$.

\smallskip

\noindent (c) \emph{Apply the closeness corollary to identify sewings.}
From Step~1 and (a), $X\oplus H$ is $\theta$-almost multiplicative and has finite $p$-variation controlled by $\omega$, hence $X\oplus H\in\Omega^{\mathrm{am},(k,p)}_{\theta,\omega}(V)$ (Definition~\ref{def: Omega^(am, (k,p))_omega,theta}).
From Step~2 we have the closeness bound \eqref{eq:close-new}, i.e. \eqref{ineq: ineq in X and Y close => S(X)=S(Y) corollary} with $X=(X\oplus H)$ and $Y=(X\otimes H)$, and (b) provides \eqref{eq:hi-level-pvar-assumption}.

Thus Corollary~\ref{cor:X_Y_close_implies_SX_eq_SY} applies and yields that $X\otimes H$ is $\theta$-almost multiplicative of degree $k$ with finite $p$-variation controlled by $\omega$ (with some constant) and
\[
\mathscr S^k(X\otimes H)=\mathscr S^k(X\oplus H).
\]
Finally, define
\[
X\boxplus H := \mathscr S^k(X\oplus H)\in\Omega^{(k,p)}_\omega(V),
\]
and the displayed equality gives $\mathscr S^k(X\otimes H)=X\boxplus H$ as claimed.
\end{proof}

\subsection{Extension and upgrading tools}
We conclude this section with technical lemmas connecting extension and sewing, and clarifying stepwise invariance of the upgrading procedure.

\begin{lemma}[Extension as sewing of the trivial embedding]\label{lem:ext-is-sewing}

Let $p\ge1$ and let $n\in\mathbb N$ satisfy $n\ge \lfloor p\rfloor$. Let $X:\triangle_J\to T^{(n)}(V)$ be a multiplicative functional of finite $p$-variation. Choose a control $\omega$ such that $X$ has finite $p$-variation controlled by $\omega$. Define the (level-$(n+1)$) \emph{trivial embedding}
\[
\iota_n(X)_{s,t}:=\bigl(1,X^1_{s,t},\dots,X^n_{s,t},0\bigr)\in T^{(n+1)}(V),
\qquad (s,t)\in\triangle_J,
\]
and set
\[
\theta:=\frac{n+1}{p}>1.
\]
Then $\iota_n(X)\in \Omega^{\mathrm{am},(n+1,p)}_{\theta,\omega}(V)$, so its sewing is well-defined. The following identity involving $\operatorname{Ext}_{(n,p)}:\Omega^{(n,p)}(V)\to\Omega^{(n
+1,p)}(V)$,
\[
\operatorname{Ext}_{(n,p)}(X)
\;=\;
\mathscr S^{(n+1,p)}_{\theta,\omega}\bigl(\iota_n(X)\bigr).
\]

In particular, by Lemma \ref{lem:control-independence-sewing} one may write succinctly
\[
\operatorname{Ext}_{(n,p)}\;=\; \mathscr S^{n+1}\circ \iota_n .
\]

\end{lemma}

\begin{proof}

\textbf{\boldmath\underline{Step 1: $\iota_n(X)$ is $\theta$-almost multiplicative of degree $n+1$}}

For $s\le u\le t$, since $X$ is multiplicative up to level $n$, we have

\[
\bigl(\iota_n(X)_{s,u}\otimes \iota_n(X)_{u,t}-\iota_n(X)_{s,t}\bigr)^j=0,
\qquad j=1,\dots,n.
\]

At level $n+1$, using that the $(n+1)$-component of $\iota_n(X)$ is identically $0$,
\begin{align*}
\bigl(\iota_n(X)_{s,u}\otimes \iota_n(X)_{u,t}-\iota_n(X)_{s,t}\bigr)^{n+1}
&=
\sum_{k=1}^{n} X^k_{s,u}\otimes X^{n+1-k}_{u,t}.
\end{align*}

Using admissibility of tensor norms and the $p$-variation bounds of $X$,

\[
\Big\|\sum_{k=1}^{n} X^k_{s,u}\otimes X^{n+1-k}_{u,t}\Big\|
\le
\sum_{k=1}^n
\frac{\omega(s,u)^{k/p}}{\beta(p)\bigl(\frac{k}{p}\bigr)!}\,
\frac{\omega(u,t)^{(n+1-k)/p}}{\beta(p)\bigl(\frac{n+1-k}{p}\bigr)!}.
\]

Set $a:=\omega(s,u)$ and $b:=\omega(u,t)$. Then
\[
\sum_{k=1}^n
\frac{a^{k/p}b^{(n+1-k)/p}}{\bigl(\frac{k}{p}\bigr)!\bigl(\frac{n+1-k}{p}\bigr)!}
\le
\sum_{k=0}^{n+1}
\frac{a^{k/p}b^{(n+1-k)/p}}{\bigl(\frac{k}{p}\bigr)!\bigl(\frac{n+1-k}{p}\bigr)!}
\le
\frac{p\,(a+b)^{(n+1)/p}}{\bigl(\frac{n+1}{p}\bigr)!},
\]
where the last inequality is the neo-classical inequality applied with $n$ replaced by $n+1$.

Using also $a+b=\omega(s,u)+\omega(u,t)\le \omega(s,t)$ yields
\[
\Big\|\sum_{k=1}^{n} X^k_{s,u}\otimes X^{n+1-k}_{u,t}\Big\|
\le
\frac{p}{\beta(p)^2\bigl(\frac{n+1}{p}\bigr)!}\,\omega(s,t)^{(n+1)/p}
=
C_{p,n}\,\omega(s,t)^{\theta},
\]
with $C_{p,n}:=\frac{p}{\beta(p)^2\bigl(\frac{n+1}{p}\bigr)!}$.

Thus $\iota_n(X)$ is $\theta$-almost multiplicative of degree $n+1$.

\smallskip

\textbf{\boldmath\underline{Step 2: $\iota_n(X)$ has finite $p$-variation}}

The first $n$ levels of $\iota_n(X)$ coincide with those of $X$, and the $(n+1)$st level is $0$. In particular, the $p$-variation bound at level $i=n+1$ holds trivially. Hence $\iota_n(X)$ has finite $p$-variation controlled by $\omega$.

\smallskip

Set $K:=\max\{1,C_{p,n}\}$. Then $\iota_n(X)$ is $\theta$-almost multiplicative and has finite $p$-variation controlled by $\omega$ with constant $K$.

Therefore $\iota_n(X)\in \Omega^{\mathrm{am},(n+1,p)}_{\theta,\omega}(V)$ and $\mathscr S^{(n+1,p)}_{\theta,\omega}(\iota_n(X))$ is well-defined.

\medskip

\textbf{\boldmath\underline{Step 3: Identify the sewing with $\operatorname{Ext}_{(n,p)}(X)$ by uniqueness}}

Let
\[
\bar X := \mathscr S^{(n+1,p)}_{\theta,\omega}\bigl(\iota_n(X)\bigr)\in\Omega^{(n+1,p)}(V).
\]
By the rough sewing lemma (Theorem~\ref{thm:Rough-Sewing-Lemma}), $\bar X$ is the \emph{unique} multiplicative functional of degree $n+1$ with finite $p$-variation such that

\[
\sup_{\substack{(s,t)\in\triangle_J,\, s<t,\\				\text{s.t. }\omega(s,t)>0;\\ i=1,\dots,n+1}}
\frac{\|\bar X^i_{s,t}-\iota_n(X)^i_{s,t}\|}{\omega(s,t)^{\theta}}<\infty.
\]
On the other hand, $\operatorname{Ext}_{(n,p)}(X)\in \Omega^{(n+1,p)}(V)$ is multiplicative and extends $X$, so for $i=1,\dots,n$ we have $\operatorname{Ext}_{(n,p)}(X)^i=\iota_n(X)^i$, and thus
\[
\|\operatorname{Ext}_{(n,p)}(X)^i_{s,t}-\iota_n(X)^i_{s,t}\|=0,
\qquad i=1,\dots,n.
\]
At level $n+1$, the extension theorem gives
\[
\|\operatorname{Ext}_{(n,p)}(X)^{n+1}_{s,t}\|
\le
\frac{\omega(s,t)^{(n+1)/p}}{\beta(p)\bigl(\frac{n+1}{p}\bigr)!}
=
\frac{1}{\beta(p)\bigl(\frac{n+1}{p}\bigr)!}\,\omega(s,t)^{\theta},
\]
so $\operatorname{Ext}_{(n,p)}(X)$ satisfies the same sewing-closeness condition to $\iota_n(X)$, and in particular the supremum above is bounded by $\frac{1}{\beta(p)\bigl(\frac{n+1}{p}\bigr)!}$.

By uniqueness in the rough sewing lemma, it follows that
\[
\operatorname{Ext}_{(n,p)}(X)=\bar X
=
\mathscr S^{(n+1,p)}_{\theta,\omega}\bigl(\iota_n(X)\bigr),
\]
as claimed. 
\end{proof}

\begin{corollary}[One-step extension does not depend on the exponent]
\label{cor:Ext-q-independence}
Fix an integer $n\ge 1$ and let $Z:\triangle_J\to T^{(n)}(V)$ be a degree-$n$ multiplicative functional. Suppose  $Z\in\Omega^{(n,q)}(V)\cap\Omega^{(n,\tilde q)}(V)$ for $q,\tilde q\in[1,n+1)$, then
\[
\operatorname{Ext}_{(n,q)}(Z)=\operatorname{Ext}_{(n,\tilde q)}(Z)
\quad\text{as }T^{(n+1)}(V)\text{-valued multiplicative functionals.}
\]
\end{corollary}
\begin{proof}
Apply Lemma~\ref{lem:ext-is-sewing} with $p=q$ and $X=Z$ to get\[
\operatorname{Ext}_{(n,q)}(Z)=\mathscr S^{\,n+1}(1,Z^1,\dots,Z^n,0).
\]
Applying the same lemma with $p=\tilde q$ yields
\[
\operatorname{Ext}_{(n,\tilde q)}(Z)=\mathscr S^{\,n+1}(1,Z^1,\dots,Z^n,0).
\]
Hence $\operatorname{Ext}_{(n,q)}(Z)=\operatorname{Ext}_{(n,\tilde q)}(Z)$.
\end{proof}

The following proposition enlarges the domain of the extension map $\operatorname{Ext}$ using the fact that, at fixed level $k$, the one-step extension does not depend on the choice of exponent $q$ as soon as  $q<k+1$.

\begin{proposition}[Generalised extension map] \label{prop:generalised-extension-map}
Let $k\ge1$ be an integer. The map
\[
\operatorname{Ext}_k:\bigcup_{1\leq q< k+1}\Omega^{(k,q)}(V)\to\bigcup_{1\leq q< k+1}\Omega^{(k+1,q)}(V)
\]
given by, for each $q\in [1,k+1)$,
\[
\operatorname{Ext}_k|_{\Omega^{(k,q)}(V)}=\operatorname{Ext}_{(k,q)},
\]
is well-defined. For any real number $p\ge1$, we set $\operatorname{Ext}_p:=\operatorname{Ext}_{\lfloor p \rfloor}$.
\end{proposition}

\begin{proof}
If $X\in\Omega^{(k,q)}(V)\cap\Omega^{(k,\tilde q)}(V)$ with $q,\tilde q<k+1$, then $\operatorname{Ext}_{(k,q)}(X)=\operatorname{Ext}_{(k,\tilde q)}(X)$ by Corollary~\ref{cor:Ext-q-independence} with $n=k$.
\end{proof}

\begin{theorem}{\normalfont\textbf{(Stepwise invariance of sewing along the upgrading sequence)}}
\label{thm:stepwise-invariance-sewing}
Fix $p\ge1$, $\theta>1$, $k\in\mathbb N$, and a control $\omega$.
Let $X\in\Omega^{\mathrm{am},(k,p)}_{\theta,\omega}(V)$ and let $\bar X:=\mathscr S^{(k,p)}_{\theta,\omega}(X)$ be its sewing given by Theorem~\ref{thm:Rough-Sewing-Lemma}. Construct $X^{(0)},X^{(1)},\dots,X^{(k)}$ as in Theorem~\ref{thm:Rough-Sewing-Lemma}(ii), so that $X^{(0)}=X$ and $X^{(k)}=\bar X$.
Then for every $m=1,\dots,k$,
\[
\mathscr S^{(k,p)}_{\theta,\omega}\bigl(X^{(m)}\bigr)
=
\mathscr S^{(k,p)}_{\theta,\omega}\bigl(X^{(m-1)}\bigr),
\]
and hence
\[
\mathscr S^{(k,p)}_{\theta,\omega}(X)
=\mathscr S^{(k,p)}_{\theta,\omega}\bigl(X^{(1)}\bigr)
=\cdots
=\mathscr S^{(k,p)}_{\theta,\omega}\bigl(X^{(k)}\bigr)
=\bar{X} .
\]
Moreover, each $X(m)$ belongs to $\Omega^{\mathrm{am},(k,p)}_{\theta,\omega}(V)$.

\end{theorem}
\begin{proof}
We prove the claim by induction on $m=1,\dots,k$, and in particular we simultaneously justify that each $\mathscr S^{(k,p)}_{\theta,\omega}(X^{(m)})$ is well-defined.

\medskip
\noindent\textbf{\underline{Preparatory bounds}}

Let $\bar X:=\mathscr S^{(k,p)}_{\theta,\omega}(X)$. By \eqref{bound almost mult. functionals}, there exists $K_{\mathscr S}\ge0$ such that for all $(s,t)\in\triangle_J$ and all $j=1,\dots,k$,
\[
\|\bar X^j_{s,t}-X^j_{s,t}\|\le K_{\mathscr S}\,\omega(s,t)^\theta.
\]
Also, since $X\in\Omega^{\mathrm{am},(k,p)}_{\theta,\omega}(V)$ and $\bar X\in\Omega^{(k,p)}_\omega(V)$, there exist constants $K_X,K_{\bar X}\ge0$ such that for all $(s,t)\in\triangle_J$ and $j=1,\dots,k$,
\[
\|X^j_{s,t}\|\le K_X\,\frac{\omega(s,t)^{j/p}}{\beta(p)\,(\frac{j}{p})!},
\qquad
\|\bar X^j_{s,t}\|\le K_{\bar X}\,\frac{\omega(s,t)^{j/p}}{\beta(p)\,(\frac{j}{p})!}.
\]
Consequently, for every $m$ and every level $j$, since $(X^{(m)})^j$ is equal to either $X^j$ or $\bar X^j$, we have the uniform bound
\begin{equation}\label{eq:Xm-uniform-pvar-bound}
\|(X^{(m)})^j_{s,t}\|\le K_{\mathrm{hi}}\,
\frac{\omega(s,t)^{j/p}}{\beta(p)\,(\frac{j}{p})!},
\qquad
K_{\mathrm{hi}}:=\max\{K_X,K_{\bar X}\}.
\end{equation}
In particular, \eqref{eq:Xm-uniform-pvar-bound} provides the “high-level $p$-variation” hypothesis \eqref{eq:hi-level-pvar-assumption} in Corollary~\ref{cor:X_Y_close_implies_SX_eq_SY} whenever it is needed (i.e. for those $j$ with $j/p>\theta$).

\medskip
\noindent\textbf{\underline{Induction statement}}

For each $m=1,\dots,k$ we prove:
\smallskip
\noindent (A$_m$) $X^{(m)}\in\Omega^{\mathrm{am},(k,p)}_{\theta,\omega}(V)$ (so $\mathscr S^{(k,p)}_{\theta,\omega}(X^{(m)})$ is well-defined), and

\smallskip
\noindent (B$_m$) $\mathscr S^{(k,p)}_{\theta,\omega}(X^{(m)})=\mathscr S^{(k,p)}_{\theta,\omega}(X^{(m-1)})$.

\vspace{2cm}

\noindent\textbf{\underline{Base case $m=1$.}}

We have $X^{(0)}=X\in\Omega^{\mathrm{am},(k,p)}_{\theta,\omega}(V)$ by assumption. By construction of the upgrading sequence, $X^{(1)}$ and $X^{(0)}$ coincide in every level $j\neq 1$, and $(X^{(1)})^1=\bar X^1$ while $(X^{(0)})^1=X^1$. Hence for all $(s,t)\in\triangle_J$,
\[
\bigl\|(X^{(1)})^j_{s,t}-(X^{(0)})^j_{s,t}\bigr\|
=
\begin{cases}
0, & j\neq 1,\\[3pt]
\|\bar X^1_{s,t}-X^1_{s,t}\|\le K_{\mathscr S}\,\omega(s,t)^\theta, & j=1.
\end{cases}
\]
Together with the high-level $p$-variation bound \eqref{eq:Xm-uniform-pvar-bound} for $m=1$, Corollary~\ref{cor:X_Y_close_implies_SX_eq_SY} applies with $X^{(0)}$ in the role of $X$ and $X^{(1)}$ in the role of $Y$. It yields
\[
\mathscr S^{(k,p)}_{\theta,\omega}(X^{(1)})=\mathscr S^{(k,p)}_{\theta,\omega}(X^{(0)}),
\]
which is (B$_1$), and it also asserts that $X^{(1)}$ is itself $\theta$-almost multiplicative with finite $p$-variation
(controlled by $\omega$ up to a constant), i.e. (A$_1$).

\medskip
\noindent\textbf{\underline{Inductive step}}

Fix $m\in\{2,\dots,k\}$ and assume (A$_{m-1}$), i.e. that $X^{(m-1)}\in\Omega^{\mathrm{am},(k,p)}_{\theta,\omega}(V)$. By construction, $X^{(m)}$ and $X^{(m-1)}$ coincide in every level $j\neq m$, while $(X^{(m)})^m=\bar X^m$ and $(X^{(m-1)})^m=X^m$. Therefore, for all $(s,t)\in\triangle_J$,
\[
\bigl\|(X^{(m)})^j_{s,t}-(X^{(m-1)})^j_{s,t}\bigr\|
=
\begin{cases}
0, & j\neq m,\\[3pt]
\|\bar X^m_{s,t}-X^m_{s,t}\|\le K_{\mathscr S}\,\omega(s,t)^\theta, & j=m.
\end{cases}
\]
Moreover, \eqref{eq:Xm-uniform-pvar-bound} gives the high-level $p$-variation hypothesis for $Y=X^{(m)}$. Thus Corollary~\ref{cor:X_Y_close_implies_SX_eq_SY} applies with $X^{(m-1)}$ as $X$ and $X^{(m)}$ as $Y$,
and yields
\[
\mathscr S^{(k,p)}_{\theta,\omega}(X^{(m)})=\mathscr S^{(k,p)}_{\theta,\omega}(X^{(m-1)}),
\]
i.e. (B$_m$). The same corollary also implies that $X^{(m)}\in\Omega^{\mathrm{am},(k,p)}_{\theta,\omega}(V)$, i.e. (A$_m$), closing the induction.

\medskip
\noindent\textbf{\underline{Closing the induction}}

Iterating (B$_m$) over $m=1,\dots,k$ gives
\[
\mathscr S^{(k,p)}_{\theta,\omega}(X)
=\mathscr S^{(k,p)}_{\theta,\omega}(X^{(1)})
=\cdots
=\mathscr S^{(k,p)}_{\theta,\omega}(X^{(k)}).
\]
Since $X^{(k)}=\bar X$ by construction, the final value is $\bar X$, completing the proof.

\end{proof}
	\newpage
\section{\textit{H}-space}
In this section, we devote ourselves to studying the space of rough paths we introduced in Definition \ref{def:h-space}. We begin by introducing the following equivalent characterisation of $H$-space, which will prove instrumental to prove our later results. Its proof can be seen as a generalisation of the idea underlying Example \ref{ex:pure-area-rp}.

\subsection{Equivalent \textit{H}-space criteria}
We begin with an equivalent, level-wise characterization of $H$-space that will be used repeatedly to track regularity across tensor levels.

\begin{theorem}[Equivalent characterisation of $\mathscr H_{\!\omega}^{(k,p)}(V)$]
\label{thm:equiv-H-space}\label{lem:311}
Fix $p\ge1$, an integer $1\le k\le\lfloor p\rfloor$ and a control $\omega$ on $J$. Let  $H=(1,H^1,...,H^k)\in\Omega^{(k,p)}_\omega(V)$. Let $k':=\min\{k,\lfloor p \rfloor -1 \}$. Then, the following conditions are equivalent:
\begin{enumerate}[label=(\roman*)]
\item\label{H1} \textbf{Uniform control.}  
$H\in\mathscr H_{\!\omega}^{(k,p)}(V)$.
\item\label{H2} \textbf{Level--wise exponents.}
When $p>1$, for each $j=1,\dots,k'$ there exists $q_j\in\bigl[j,\;j+\tfrac{j}{p-1}\bigr)$ such that
\[
H(j)=(1,H^1,\dots,H^j)\in \Omega^{(j,q_j)}_\omega(V).
\]
If $k=\lfloor p\rfloor$, we  additionally set $q_k:=p$; noting that $H(k)=H\in \Omega^{(k,p)}_\omega(V)$ already by hypothesis. When $p=1$ (hence $k=1$), set $q_1:=1$ and $(1,H^1)\in \Omega^{(1,1)}_\omega(V)$.
\end{enumerate}

If $H$ satisfies Condition 1 with $H\in\mathscr H_{\phi,\omega}^{(k,p)}(V)$ for some $\phi\in(1-1/p,1]$, then $H$ satisfies Condition 2 with $q_j=j/\phi$ for each $j=1,...,k'$.
\end{theorem}
\begin{proof}
If $p=1 = k$, then the equivalence is trivial. We will thus treat the case that $p>1$. Recall for both directions of this proof that, by super-additivity and positivity of the control $\omega$, $\omega(s,t) \le \omega(S,T)$ for all $(s,t)\in\triangle_J$.

\textbf{Proof of \ref{H1}$\!\Rightarrow$\ref{H2}.} 

Because $H\in\mathscr H_{\!\omega}^{(k,p)}(V)$, there exist $K>0$ and $\phi\in\bigl(1-\frac{1}{p},\,1\bigr]$ such that
\[
  \lVert H^m_{s,t}\rVert \;\le\; K\,\omega(s,t)^{\phi},
  \qquad m=1,\dots,k,\;\; (s,t)\in\triangle_J .
\]

\medskip\noindent\textbf{\underline{Step 1: Choose the level-wise exponents}}

For every $j\in\{1,\dots ,k'\}$ set
$$
   q_j\;:=\;\frac{j}{\phi}.
$$
\emph{Lower bound.} Because $\phi\le1$,
$$
   q_j=\frac{j}{\phi}\;\ge\;j .
$$
\emph{Upper bound.} Since $\phi>1-\tfrac1p$,
$$
   q_j \;=\;\frac{j}{\phi}
          \;<\;\frac{j}{\,1-\frac1p\,}
          \;=\;j\Bigl(1+\frac{1}{p-1}\Bigr)
          \;=\;j+\frac{j}{p-1}.
$$
 Together with \(q_j\ge j\), this yields
$$
   q_j\;\in\;\Bigl[j,\;j+\frac{j}{p-1}\Bigr) .
$$
This establishes the required inclusion for all $j=1,\dots ,k'$.

\medskip\noindent\textbf{\underline{Step 2: Estimate the $j$-th level}}
\begin{equation}\label{eq:levelj}
   \lVert H^j_{s,t}\rVert
   \;\le\; K\,\omega(s,t)^{\frac{j}{q_j}}
   \;=\; K\,\omega(s,t)^{\phi},
\end{equation}
where we used the identity \(\frac{j}{q_j}=\phi\).

\medskip\noindent\textbf{\boldmath\underline{Step 3: Estimate the lower levels \(\ell=1,\dots,j-1\)}}

Because \(q_\ell=\frac{\ell}{\phi}<q_j\), the exponent \(\frac{\ell}{q_\ell}-\frac{\ell}{q_j}\) is non-negative, so
\[
  \omega(s,t)^{\frac{\ell}{q_\ell}-\frac{\ell}{q_j}}
      \;\le\; \omega(S,T)^{\frac{\ell}{q_\ell}-\frac{\ell}{q_j}},
  \qquad(s,t)\in\triangle_J .
\]
Hence
\begin{equation}\label{eq:lowerlevel}
   \lVert H^\ell_{s,t}\rVert
   \;\le\;
   K\,\omega(s,t)^{\frac{\ell}{q_\ell}}
     \,\omega(S,T)^{\frac{\ell}{q_\ell}-\frac{\ell}{q_j}}
   =:
   \tilde K(\ell,j)\,\omega(s,t)^{\frac{\ell}{q_j}},
   \qquad \ell=1,\dots,j-1 .
\end{equation}

Define the uniform constant
\[
   K^{*}:=
   \max\Bigl\{\,K,\;
     \max_{1\le j\le k'}\;
     \max_{1\le \ell<j}
     \tilde K(\ell,j)
   \Bigr\}
   \;\ge\; K .
\]
Consequently, \(K^{*}\) dominates $K$ as well as the individual constants \(\tilde K(\ell,j)\) for all \(j=1,\dots,k'\) and \(\ell=1,\dots,j-1\).

\medskip\noindent\textbf{\underline{Step 4: Collect the bounds}}

Equations \eqref{eq:levelj} and \eqref{eq:lowerlevel} give, for every \(j=1,\dots,k'\) and every \(\ell=1,\dots,j\), as well as every $(s,t)\in\triangle_J$,
\[
   \lVert H^\ell_{s,t}\rVert
   \;\le\; K^{*}\,\omega(s,t)^{\frac{\ell}{q_j}} \le K^{*}\beta(q_j)\max_{1\le \ell\le j}\{(\tfrac{\ell}{q_j})!\}\frac{\omega(s,t)^{\frac{\ell}{q_j}} }{\beta(q_j)\cdot(\tfrac{\ell}{q_j})!}.
\]
Set $K^{**}:=K^{*}\beta(q_j)\max_{j=1,...,k'}\max_{1\le \ell\le j}\{(\tfrac{\ell}{q_j})!\}$. We thus have, for every \(j=1,\dots,k'\),

\begin{equation}\label{eq:qj-range}
       \lVert H^\ell_{s,t}\rVert
  \le K^{**} \frac{\omega(s,t)^{\frac{\ell}{q_j}} }{\beta(q_j)\cdot(\tfrac{\ell}{q_j})!} \quad \forall(s,t)\in\triangle_J, \forall\ell=1,..., j.
\end{equation}

Therefore, the truncated functional \(H(j):=(1,H^1,\dots,H^j)\) has finite \(q_j\)- variation controlled by \(\omega\). Using \eqref{eq:qj-range}, we deduce $H(j)$ has finite $q_j$-variation. Using the multiplicativity of \(H\) (since $H$ is a $p$-rough path), we trivially deduce the multiplicativity of $H(j)$. Hence, we conclude \(H(j)\in\Omega^{q_j}_{\!\omega}(V)\), completing the proof of \(\textbf{H1}\Rightarrow\textbf{H2}\).

 
\textbf{{Proof of\ref{H2}$\!\Rightarrow$\ref{H1}}.}

Assume \ref{H2}. Set $C := \max\{ \max_{j \in \{1, \dots, k'\}} \{ \frac{1}{\beta(q_j)\left( \frac{j}{q_j} \right)!} \}, \frac{1}{\beta(q_j)\left( \frac{\lfloor p \rfloor}{p} \right)!} \}\}
$. Since $\omega$ is fixed, there exists a constant $K_1>0$ such that
\[
   \|H^j_{s,t}\|\le K_1\,\frac{\omega(s,t)^{\frac{j}{q_j}}}{\beta(q_j)\cdot(\frac{j}{q_j})!}\le K_1 \,C\, \omega(s,t)^{\frac{j}{q_j}} ,
   \qquad \forall (s,t)\in\triangle_J,\; j=1,\dots ,k',
\]
with \(q_j\in\bigl[j,\;j+\tfrac{j}{p-1}\bigr)\).
In the case where $k=\lfloor p  \rfloor$, at the $\lfloor p\rfloor$-th level, we also have a constant $K_2>0$ such that
\[
   \|H^{k}_{s,t}\|\le K_2\,\frac{\omega(s,t)^{\frac{\lfloor p \rfloor}{p}}}{\beta(q_j) (\frac{\lfloor p \rfloor}{p})!}\le K_2\, C\,\omega(s,t)^{\frac{\lfloor p \rfloor}{p}},
   \qquad (s,t)\in\triangle_J,
\]
since $H\in\Omega^{(k,p)}_\omega(V)$, we have the $p$-variation estimate at level $k$ controlled by $\omega$.

We then define 
$$K=C\max\{K_1,K_2\} .$$
This implies that, recalling that $q_{\lfloor p \rfloor}= p$ when $k=\lfloor p \rfloor$, 
\[
  \|H^j_{s,t}\|\le K\, \omega(s,t)^{\frac{j}{q_j}}   \qquad \forall (s,t)\in\triangle_J,\; j=1,\dots ,k.\tag{$\ast$}
\]

\smallskip\noindent\textbf{\underline{Step 1: Level-wise exponents}}

For each \(j=1,\dots ,k'\) define \(\displaystyle\phi_j:=\frac{j}{q_j}\). If \(k=\lfloor p\rfloor\) set in addition \(\displaystyle\phi_k:=\frac{k}{p}\).

For \(j\le k'\) we have \(q_j<j+\frac{j}{p-1}\), hence
\[
   \phi_j=\frac{j}{q_j}>\frac{j}{\,j+\frac{j}{p-1}\,}
         =\frac{p-1}{p}=1-\frac1p,
   \qquad\text{and}\qquad
   \phi_j\le1.
\]
If \(k=\lfloor p\rfloor\) then \(\displaystyle\phi_k=\frac{k}{p}\in\bigl(1-\tfrac1p,\,1\bigr]\) because 

$$1-\frac1p < \frac{\lfloor p\rfloor}{p } \iff p < \lfloor p \rfloor +1,$$
which is true. Thus, in every case
\[
   \phi_j\in\bigl(1-\tfrac1p,\,1\bigr],\qquad j=1,\dots ,k.
\]

\smallskip\noindent\textbf{\underline{Step 2: Choose a single exponent}}

Set
\[
   \phi:=\min_{1\le j\le k}\phi_j\in(1-\tfrac1p,\,1].
\]

\smallskip\noindent\textbf{\underline{Step 3: Obtain uniform bounds}}

Let \(M:=\omega(S,T)\). 

\smallskip\noindent\emph{Case A (\(M\le1\)).}
For every \(j\le k\) and every \((s,t)\in\triangle_J\),
\begin{equation}
    \label{ineq: step 3}
   \|H^j_{s,t}\|\le K\,\omega(s,t)^{\phi_j}\le K\,\omega(s,t)^{\phi},
\end{equation}

because \(\phi\le\phi_j\) and $\omega(s,t)\le M\le1$.  Take \(K^{*}:=K\).

\smallskip\noindent\emph{Case B (\(M>1\)).}
If \(\omega(s,t)\le1\) the same bound in \eqref{ineq: step 3} holds.  When \(1<\omega(s,t)\le M\),
\[
   \omega(s,t)^{\phi_j}
      =\omega(s,t)^{\phi}\,\omega(s,t)^{\phi_j-\phi}
      \le\omega(s,t)^{\phi}\,M^{\phi_j-\phi},
\]
so
\[
   \|H^j_{s,t}\|\le K\,M^{\phi_j-\phi}\,\omega(s,t)^{\phi}.
\]
Define
\[
   K^{*}:=K\max_{1\le j\le k}M^{\phi_j-\phi}.
\]
\newpage
\smallskip\noindent\textbf{\underline{Step 4: Combined bound}}

Since $\phi_j-\phi\ge0$ for each $j\in\{1,...,k\}$ and $M\ge1$, it follows that $\max_{1\le j\le k}M^{\phi_j-\phi}\ge1$ and thus $K^*\ge K$. Therefore, in both Case A and Case B,
\[
   \|H^j_{s,t}\|\le K^{*}\,\omega(s,t)^{\phi},
   \qquad j=1,\dots ,k,\; (s,t)\in\triangle_J.
\]

\end{proof}

Below, we provide definitions of $\text{Comm}^k(V)$ for $k\ge0$, $\mathfrak{g}(V)$ and $\mathfrak{g}^{(n)}(V)$ for $n\ge1$. Given an $H\in \mathscr{H}^p(V)$, such definitions will help us to make sense of the implications of assuming that $H$ is a weakly-geometric $p$-rough path.

\subsection{Commutators and geometry}
The following commutator and Lie-algebra notation is used to describe the weakly geometric case and the structure of level-wise ‘innovations’.

\begin{definition}[Comm]\label{def:comm}
	Let $k \geqslant 0$ be an integer. For two vector spaces $E_1$ and $E_2$, define the set $[E_1, E_2]: = \operatorname{span}\{ e \otimes f - f \otimes e  | e \in E_1 \   \text{and} \   f \in E_2 \}$. Then, we define
	$$
	\operatorname{Comm}^k(V):= 
	\begin{cases}
		\{0\}, 	& \operatorname{when} k=0 \\
		V, 		& \operatorname{when} k=1 \\
		{\left[V, \; \operatorname{Comm}^{k-1}(V)\right],} & \operatorname{ for } k \geqslant 2 .
	\end{cases}
	$$

\end{definition}


\begin{definition}[$\mathfrak{g}$]\label{def:g}
	Let $n \geqslant 1$ be an integer. Set
	\begin{align*}
		\mathfrak{g}(V) & := \bigoplus_{k=0}^{\infty} \operatorname{Comm}^k(V) \Bigl(\subseteq T(V)\Bigr)  \text{,}\\
		\mathfrak{g}^{(n)}(V) &:= \bigoplus_{k=0}^n \operatorname{Comm}^k(V)\Bigl( \subseteq T^{(n)}(V) \Bigr)  \text{.}
	\end{align*}
	
\end{definition}


The following lemma is an extended version of Lemma 3.4 in \cite{stflour}.
\begin{lemma}
  \label{lem:mult. func.}
  Let \(X:\triangle_J \to T(V)\) and \(Y:\triangle_J\to T(V)\) be multiplicative  functionals of degree \(n\) and assume \(X^i = Y^i\) for each \(i=1,\ldots,n-1\).  Then \(\psi := X^n - Y^n : \triangle_J \longrightarrow V^{\otimes n}\)  is an additive function:

  \EqLine{a}{\psi_{su} + \psi_{ut} = \psi_{st},\quad \forall s \le u \le t.}

  If, in addition, \(X\) and \(Y\) both have finite \(p\)-variation controlled by $\omega$ for some \(p\ge1\), then:

  \EqLine{b}{\|\psi_{st}\| \le 2\,\frac{\omega(s,t)^{\tfrac{n}{p}}}{\beta(p)(\tfrac{n}{p})!},\quad \forall\,(s,t)\in\triangle_J.}

  If, moreover, \(X\) and \(Y\) are weakly geometric $p$-rough paths for some \(p<n+1\), then:

  \EqLine{c}{\psi \in \mathrm{Comm}^n(V).}
\end{lemma}

 
The proof of Lemma \ref{lem:mult. func.}\,(a) is given in \cite{stflour}. Part (b) follows directly from the triangle inequality.
Part (c), though more delicate, is omitted because it is not needed in Section \ref{sec:main-results}.

The next lemma is a partial converse to Lemma \ref{lem:mult. func.}. Part (a) is likewise proved in \cite{stflour}, and Part (b) again follows immediately from the triangle inequality. Part (c) is omitted for the same reason: it is not required for the results that follow.


\begin{lemma}
  \label{lem:weakly geometric p-rough path (518)}
  Suppose that \(Z\) is a degree-\(n\) multiplicative functional and
  \(\eta:\triangle_J \longrightarrow V^{\otimes n}\) is a functional.
  If \(\eta\) is additive, then:

  \EqLine{a}{Z+(0,\ldots,0,\eta)\ \text{is a degree-}n\ \text{multiplicative functional.}}

  In addition, if \(Z\) has finite \(p\)-variation for some \(p\in[1,n+1)\) with control \(\omega\), and  \(\|\eta_{st}\|\le K\,\omega(s,t)^{\frac{n}{p}}\) for all  \((s,t)\in\triangle_J\) and some constant \(K\), then:

  \EqLine{b}{Z+(0,\ldots,0,\eta)\ \text{has finite }p\text{-variation.}}

  Finally, if \(Z\) is a weakly geometric \(p\)-rough path, and \(\eta\)
  is additive, has finite \(p\)-variation, and
  \(\eta_{st}\in\mathrm{Comm}^n(V)\), then:

  \EqLine{c}{Z+(0,\ldots,0,\eta)\ \text{is a weakly geometric }p\text{-rough path.}}
\end{lemma}


\begin{remark}\label{rem:2.9}

Let \(p\ge1\), \(k\in\mathbb{N}\cup\{0\}\), \(\phi\in\bigl(1-\tfrac1p,\,1\bigr]\), and let \(\omega\) be a control.  Given \(H\in\mathscr{H}_{\phi,\omega}^{(k,p)}(V)\), set
\[
  q_j := \frac{j}{\phi}\quad\text{for }j=1,\dots,k.
\]
By Theorem \ref{thm:equiv-H-space} we have \(H(j)\in\Omega^{q_j}(V)\), so \(\operatorname{Ext}_{j}(H(j))\) is well-defined; moreover,
\(\operatorname{Ext}_{j}(H(j))\in\Omega^{(j+1,q_j)}(V)\) by definition of \(\operatorname{Ext}_{j}\).

Part (a) of Lemma \ref{lem:mult. func.} now applies: the two \((j+1)\)-step multiplicative functionals
\(H(j+1) = (1,H^1,\ldots,H^{j+1})\) and \(\operatorname{Ext}_{j}(H(j))\) agree up to level \(j\), so
\[
  H^{j+1}
  \;-\;
  \bigl(\operatorname{Ext}_{j}(1,\ldots,H^j)\bigr)^{j+1}
\]
is an additive functional.

    If, furthermore, we require that  $H\in WG\mathscr{H}^{(k,p)}_{\phi,\omega}(V)$, then Part (c) of Lemma \ref{lem:mult. func.} gives that $H^{j+1} - \operatorname{Ext}_{j} \left(( 1, \ldots, H^j )\right)^{j+1}$ is in $\operatorname{Comm}^{j+1}(V)$.
\end{remark}

The following provides the definition of the development map ``$\text{dev}(\cdot)$''.

\subsection{Development map and \textit{I}-space}
We define the development map $\mathrm{dev}$, which isolates the additive remainder at each level not forced by extension of lower levels. This motivates the definition of $I$-space as the natural additive target of $\mathrm{dev}$.

\begin{definition}[Development map]
    \label{def:dev (886)}\label{def: development map}
	Let $p \geqslant 1$ be a real number,\newline  $k\in\{1,...,\lfloor p\rfloor\}$, and $H \in \mathscr{H}^{(k,p)}(V)$. Take $\phi\in(1-1/p,1]$ such that $H\in\mathscr{H}^{(k,p)}_\phi(V)$. Then for each $j=1, \ldots, k- \! 1$, we set $E^{j+1} := \operatorname{Ext}_{j} ( ( 1, \ldots, H^j ) )^{j+1}$. We then let $\operatorname{dev}: \mathscr{H}^{(k,p)}_{\phi}(V) \to  C_0(\triangle_J, \,T^{(k)}(V) )$ be  defined through, for $\text{and} \ (s,t)$ $ \in \triangle_J$,
\begin{align*}
		\operatorname{dev}(H)_{st} & := \left( \bigl( \operatorname{dev}(H) \bigr)^0, \bigl( \operatorname{dev}(H) \bigr)^1, \bigl( \operatorname{dev}(H) \bigr)^2, \ldots, \bigl( \operatorname{dev}(H) \bigr)^{k} \right)_{st} \\
		& := \ \left(
		0, \; H^1, \; H^2-E^2, \ldots, H^{k} - E^{k}
		\right)_{st} \in T^{(k)}(V).
	\end{align*}
\end{definition}
\begin{remark} The above definition is well-defined since for $q_j := \frac{j}{\phi}$, $H(j)\in\Omega^{q_j}(V)$ by Theorem \ref{thm:equiv-H-space}), so the expression $ \operatorname{Ext}_{j} ( ( 1, \ldots, H^j ) )^{j+1}$ makes sense. Moreover, we know that $\operatorname{Ext}$ is independent of the $q_j<j+1$ we choose, so the definition of $\operatorname{dev}$ is independent of the choice of $\phi$ in $(1-1/p,1]$ used to witness that $H\in \mathscr{H}^p(V).$
    
\end{remark}

The development map $\operatorname{dev}$ isolates the \emph{innovation} of $H$: the additive remainder not forced by extending its lower levels. As a formula, writing $E^{j+1} := \bigl(\operatorname{Ext}_{j}(1,\ldots,H^j)\bigr)^{j+1}$, we have $H^{j+1}=\operatorname{dev}(H)^{j+1}+E^{j+1}$. Motivated by this, we introduce the $\mathfrak{I}$-space below as the natural target of the development map; in particular, we will see that $\operatorname{dev}(H)\in\mathfrak{I}^p(V)$.


\begin{definition}[$\mathfrak{I}$-space]
    \label{def:I space 755}
    Let $p \in \mathbb{R}_{\geqslant 1}$ and $k \in \mathbb{N}\cup\{0\}$. If $k=0$, set $\mathfrak{I}^{(0, p)}(V):= \{(0)\}$. If $k\ge1$, define $\mathfrak{I}^{(k, p)}(V)$ s.t. $I = (0, I^1, \ldots, I^k) \in \mathfrak{I}^{(k, p)}(V)$ if
    \begin{enumerate}[label=(\roman*)]
        \item $I:\triangle_J \to T^{(\lfloor p\rfloor)}(V)$ is additive.
        \item There exists a control $\omega$ and a constant $\phi \in \bigl(1-\frac{1}{p},1\bigr]$ such that
        \begin{equation}
          \left\|I_{s t}^j\right\| \leqslant K\,\omega(s, t)^\phi
          \qquad \forall\, j=1,\ldots,k,\ \forall(s, t) \in \triangle_J,
          \label{eq:inequality-$I$-space}
        \end{equation}
        and, in addition, if $k=\lfloor p\rfloor$,\[
          \|I^k_{s t}\| \le K\,\omega(s,t)^{\,k/p}.
        \]
        Here $K$ depends on $\omega(S,T)$, $p$, and $\phi$.
        
    \end{enumerate}

    \noindent We also write $WG\mathfrak{I}^{(k, p)}(V)$ if we require further that
    \begin{enumerate}[label=(\roman*),resume]
            \item $\operatorname{Im}(I) \subseteq
            \left\{0\right\} \; \oplus \Bigl( \mathop{\oplus}\limits_{k=1}^\infty \operatorname{Comm}^k(V)\Bigr) \; \stackrel{\text{def}}{=} \; \mathfrak{g}^{( k )} (V) \text{.} $
    \end{enumerate}
We set $\mathfrak{I}^p(V) := \mathfrak{I}^{(\lfloor p \rfloor),p}(V)$ (and $WG\mathfrak{I}^p(V) := WG\mathfrak{I}^{( \lfloor p \rfloor,p)}(V)$). We set $\mathfrak{I}(V):=\cup_{p\ge1}\mathfrak{I}^p(V)$ and call this ``$\mathfrak{I}$-space''.
\end{definition}

\begin{remark} \label{rmk: definition of Ikp omega phi}
To specify a constant $\phi \in \bigl(1-\tfrac{1}{p}, 1\bigr]$ or a control $\omega$ such that Inequality \ref{eq:inequality-$I$-space} is fulfilled, we write $\mathfrak{I}^{(k, p)}_\phi(V)$, $\mathfrak{I}^{(k, p)}_\omega(V)$, or $\mathfrak{I}^{(k, p)}_{\phi,\omega}(V)$ to specify $\phi$, $\omega$, or both, respectively.

It is easy to thus see that:
    $$   \mathfrak{I}^{(k,p)}(V)
   =
   \bigcup_{\substack{\phi\in\bigl(1-\tfrac1p,\,1\bigr],\\[2pt]
                       \omega\ \text{control}}}
   \mathfrak{I}^{(k,p)}_{\phi,\omega}(V).$$
\end{remark}

The proof of the following lemma is elementary and is thus omitted.

\begin{lemma}
	\label{lem:additive functional 182}
	Let $X: \triangle_J \longrightarrow V$ be an additive function, where $J=[S,T]$. Then $X_{S, \cdot}:J \longrightarrow V$ is a path. Assume that $X_\cdot:J$ is a path. Then $X_{\cdot, \cdot}: \triangle_J \longrightarrow V$ given by $X_{s, t} = X_t - X_s$ is an additive function.
\end{lemma}

\begin{remark}
	Using Lemma \ref{lem:additive functional 182}, it is easy to see that functionals $I$ within the spaces $WG\mathfrak{I}^p(V)$ and $\mathfrak{I}^p(V)$ as defined in Definition \ref{def:I space 755} correspond to paths in $\mathfrak{g}^{(\lfloor p \rfloor)} (V)$ or $T^{(\lfloor p \rfloor)} (V)$, respectively, which satisfy a given regularity condition.
\end{remark}

\begin{proposition}\label{prop: dev H is in I}
	Let $p \in \mathbb{R}_{\geqslant 1}, k \in \mathbb{N}\cup\{0\}, \phi \in \bigl(1-\tfrac{1}{p}, 1\bigr]$, and $\omega$ be a control. Given $H \in \mathscr{H}_{\phi,\omega}^{(k,p)}(V),$ $\operatorname{dev}(H) \in \mathfrak{I}_{\phi,\omega}^{(k, p)}(V)$. Furthermore, if $H \in WG\mathscr{H}_{\phi,\omega}^{(k, p)}$, then $\operatorname{dev}(H) \in WG\mathfrak{I}_{\phi,\omega}^{(k,p)}(V)$.
\end{proposition}


\begin{proof}
	To show that, given $H \in \mathscr{H}_{\phi,\omega}^{(k,p)}(V)$, $\operatorname{dev}(H) \in \mathfrak{I}_{\phi,\omega}^{(k, p)}(V)$, we want to show two conditions:

(i) $\operatorname{dev} (H)$ is additive.

(ii) $\|\operatorname{dev} (H)^j_{st}\|\leqslant \dfrac{K\omega(s,t)^{\frac{j}{q_j}}}{\left(\frac{j}{q_j}\right)!}$, for some constant $K$, $\forall j\in \{1, \ldots ,k\}$, $\forall (s,t)\in \triangle_J$

Condition (i) is guaranteed by the first paragraph of Remark \ref{rem:2.9}. For condition (ii), first observe that, $\operatorname{dev}(H)^1=H^1$ satisfies the following bound for some constant $K\ge0$ by virtue of $H\in  \mathscr{H}_{\phi,\omega}^{(k,p)}(V)$, and using Theorem \ref{thm:equiv-H-space},
$$\|\operatorname{dev} (H)^1_{st}\|\leqslant \dfrac{K\omega(s,t)^{\frac{1}{q_1}}}{\left(\frac{1}{q_1}\right)!}$$
For the higher levels of $\operatorname{dev}$, observe that for $j\in \{1, \ldots, k-1\}$ and $(s,t)\in \triangle_J$,

\begin{align}
\|\operatorname{dev}(H)_{st}^{j+1}\|
&= \| H^{j+1}_{st} - \operatorname{Ext}_{j}(1,H^1,\ldots,H^j)^{\,j+1}_{st} \| \notag\\
&\le K\,\omega(s,t)^{\phi}
  + \frac{K\,\omega(s,t)^{\frac{j+1}{q_j} (= \,j\frac{\phi}{j}+\frac{\phi}{j})}}{\beta(q_j)\!\left(\frac{j+1}{q_j}\right)!}
  \label{ineq: dev(H)} \\
&= K\,\omega(s,t)^{\phi}\!\left(1+\frac{\omega(s,t)^{\frac{\phi}{j}}}{\beta(q_j)\!\left(\frac{j+1}{q_j}\right)!}\right) \notag\\
&\le K\,\omega(S,T)^{\phi}\!
   \max_{l\in\{0,\ldots,\lfloor p \rfloor\}}\left\{1+\frac{1+\omega(S,T)^{\frac{\phi}{l}}}{\beta(q_j)\!\left(\frac{l+1}{q_l}\right)!}\right\} \notag\\
&=: \tilde{K}\,\omega(s,t)^{\phi}.
\end{align} In particular, defining
\[
\widetilde K_k
:= K + \max_{1\le l\le k-1}
   \frac{\omega(S,T)^{\phi/l}}{\beta(q_j)\,\big(\frac{l+1}{q_l}\big)!},
\qquad q_l:=\tfrac{l}{\phi},
\]
we have
\[
\|\operatorname{dev}(H)^{j+1}_{st}\|\;\le\;\widetilde K_k\,\omega(s,t)^{\phi}
\quad\text{for all } j=1,\ldots,k-1.
\]
If $k=\lfloor p\rfloor$, we also have
\begin{align}
    \|\operatorname{dev}(H)^k_{st}\|\
&= \big\| H^k_{st} - \big(\mathrm{Ext}_{q_{k-1}}(1,H^1,\dots,H^{k-1})\big)^k_{st} \big\|\\
&\;\le\; K\,\omega(s,t)^{k/p} + K\,\omega(s,t)^{k/q_{k-1}}
\;\le\; \tilde{\tilde{K}}\,\omega(s,t)^{k/p},
\end{align}

since $q_{k-1}<  k-1 +(k-1)/p $ (Theorem \ref{thm:equiv-H-space}) implies $q_{k-1}< p$ which implies $k/q_{k-1}> k/p$, and where $\tilde{\tilde{K}}:= K(1+\omega(S,T)^{\,k/q_{k-1}-k/p})$). 

Finally, if $H\in WG\mathscr{H}_{\phi,\omega}^{(k,p)}(V)$, then the second paragraph of Remark \ref{rem:2.9} gives that $H^j-E^j\in \text{Comm}^j(V)$ for each $j\in \{1, \ldots , k\}$, as required to show that $\operatorname{dev}(H)\in WG\mathfrak{I}^{(k,p)}_{\phi,\omega}(V)$.
\end{proof}


\begin{proposition} \label{prop: I(k,p)(phi,omega)(V) is a vector space}
	Let $p \geqslant 1$, $\omega$ be a control, and $\phi \in \bigl(1-\tfrac{1}{p}, 1\bigr]$. Let $k\in\{1,..., \lfloor p\rfloor\}$. Then $\mathfrak{I}_{\phi,\omega}^{(k,p)}(V)$ is a vector space under pointwise addition. Moreover, $WG\mathfrak{I}^{(k, p)}_{\phi,\omega}(V)$ is a subspace of $\mathfrak{I}_{\phi,\omega}^{(k,p)}(V)$.
\end{proposition}


\begin{proof}
	Closure under Condition (i) in Definition \ref{def:I space 755} is immediate. Closure under Condition (ii) in Definition \ref{def:I space 755} follows by the triangle inequality. $WG\mathfrak{I}^{(k, p)}_{\phi,\omega}$ being a subspace follows from the subspace test, using the fact that $\mathfrak{g}^{(k)}(V)$ is a vector space in its own right.
\end{proof}

\begin{corollary}[Vector-space property of $\mathfrak{I}^{(k,p)}(V)$]
\label{cor:I-vector-space}
Let $p\ge1$ and $k\in\{1,\dots,\lfloor p\rfloor\}$. We then have:

(i) For any $\phi,\tilde \phi \in(1-1/p,1]$ and any two controls $\omega$ and $\tilde \omega$, there exist both a $\phi^* \in(1-1/p,1]$ and a control $\omega^*$ s.t. $\mathfrak I^{(k,p)}_{\phi,\omega}(V)\subset\mathfrak I^{(k,p)}_{\phi^*,\omega^*}(V)$ and $ \mathfrak I^{(k,p)}_{\tilde\phi,\tilde\omega}(V)\subset\mathfrak I^{(k,p)}_{\phi^*,\omega^*}(V)$.

(ii) The space $\mathfrak{I}^{(k,p)}(V)$ is a vector space under pointwise addition and scalar multiplication.  In particular, $\mathfrak{I}^{p}(V)=\mathfrak{I}^{(\lfloor p\rfloor,p)}(V)$ is a vector space.
\end{corollary}
\begin{proof}[Proof sketch]
Take $I,\tilde I\in \mathfrak I^{(k,p)}(V)$. By Remark \ref{rmk: definition of Ikp omega phi}, we can characterise $\mathfrak{I}^{(k,p)}(V)$ as follows:
\[
   \mathfrak{I}^{(k,p)}(V)
   =
   \bigcup_{\substack{\phi\in\bigl(1-\tfrac1p,\,1\bigr]\\[2pt]
                       \omega\ \text{control}}}
   \mathfrak{I}^{(k,p)}_{\phi,\omega}(V).
\]
Choose $(\phi,\omega,K)$ and $(\tilde\phi,\tilde\omega,\tilde K)$ s.t. $I\in \mathfrak I^{(k,p)}_{\phi,\omega}(V)$ and is degree-wise bounded by $K\omega$, while $\tilde I\in \mathfrak I^{(k,p)}_{\tilde\phi,\tilde\omega}(V)$ and is degree-wise bounded by $K\omega$. Set $\phi^*:=\min\{\phi,\tilde\phi\}$ and $\omega^*:=\omega+\tilde\omega$. Then $I,\tilde I\in \mathfrak I^{(k,p)}_{\phi^*,\omega^*}(V)$: condition (i) (additivity) is unchanged, and condition (ii) follows from the triangle inequality together with the fact that controls form a convex cone (Lemma~\ref{lem: controls are a convex cone}), up to defining a new constant $K^*$ (which will absorb $K$ and $\tilde K$). Since $\mathfrak I^{(k,p)}_{\phi^*,\omega^*}(V)$ is a vector space by Proposition \ref{prop: I(k,p)(phi,omega)(V) is a vector space} any linear combination $\lambda I+\mu\tilde I$ lies in $\mathfrak I^{(k,p)}_{\phi^*,\omega^*}(V)$, proving the vector‑space property of $\mathfrak I^{(k,p)}(V)$. If $k=\lfloor p\rfloor$, the added top‑degree bound with exponent $k/p$ is preserved by the same estimates. 
\end{proof}

\subsection{Lift map and bijection}
We now construct the lift map $1(\cdot)$ by iterated extension and show it is inverse to $\mathrm{dev}$, giving a bijection between $I$-space and $H$-space.

\begin{definition}	\label{def:X pow i (417)}
Let $p \in \mathbb{R}_{\geqslant 1}$ and set $k=\lfloor p \rfloor$. Take $X \in \Omega^{p}(V)$. Assume both that $i \in C_0\left(\triangle_J,  V^{\otimes(k +1)}\right)$ is additive, and that there exist $\phi \in \bigl(1-\tfrac{1}{p}, 1\bigr]$ and a control $\omega$ s.t.
	$$
	\left\|i_{s t}\right\|_{V^{\otimes(k+1)}} \leqslant \omega(s, t)^\phi .
	$$
Then, setting $q_{k+1} = \tfrac{k+1}{\phi} \in [k+1,k+1 + \tfrac{k+1}{p-1}),$ we define
	$$
	X^i:=\operatorname{Ext}_p(X)+(0, \ldots, 0, i) \in \Omega^{q_{k+1}}(V),
	$$
where $(0, \ldots, 0, i)$ is the embedding of $i$ into $T^{(k+1)}(V)$.
\end{definition}


\begin{remark}
	We know that $$\begin{aligned}
        \left\|\operatorname{Ext}_p(X)_{s t}^{k+1}\right\| & \leqslant \frac{\omega(s, t) ^ { \frac{k+1}{p}}}{\beta\left(\frac{k+1}{p}\right)!} \\
        &\le \omega(S,T)^{\varepsilon}  \frac{\omega(s, t) ^ { \frac{k+1}{q_{k+1}}}}{\beta(p)\left(\frac{k+1}{q_{k+1}}\right)!}
    \end{aligned} $$ where $\varepsilon:= \left(\frac{k+1}{p}-\frac{k+1}{q_{k+1}}\right)$ and where we use that $\varepsilon>0$ since $\frac{k+1}{p}> 1$ and $\frac{k+1}{q_{k+1}}\le1$. Coupled with the fact that $\left\|i_{s t}\right\| \leqslant \omega(s, t)^{\theta} = \omega(s, t)^{\tfrac{k+1}{q_{k+1}}}$, it becomes evident that $X^i$ does possess finite $q_{k+1}$-variation. 

According to Lemma {\ref{lem:weakly geometric p-rough path (518)}}, $X+(0, \ldots, 0, i)=X^i$ is furthermore a $(k\!+\!1)$-degree multiplicative functional. Consequently, $X^i \in \Omega^{q_{k+1}}(V)$ holds true, and Definition \ref{def:X pow i (417)} is indeed well-defined.

Moreover, if $i \in C_0\left(\triangle_J, \operatorname{Comm}^{k+1}(V)\right)$, and $X \in WG\Omega^q(V)$, Theorem {\ref{lem:weakly geometric p-rough path (518)}} also shows that $X^i \in WG\Omega^{q_{k+1}}(V)$.
\end{remark}

We will now provide a definition of the lift map, that is $\mathbb{1}^{(\cdot)}.$
\begin{definition}[Lift map]\label{def:1_pow_I}
Let $p \geqslant 1$ be a real number, $k\in\{1,...,\lfloor p \rfloor\}$ an integer, and fix $\phi\in(1-\frac{1}{p},1]$. Consider $I=\left(0, I^1, \ldots, I^k\right) \in \mathfrak{I}^{(k, p)}_\phi(V)$.  We set $q_j:=\tfrac{j}{\phi}\in[j,j+j/(p-1))$ for each $j$ in $\{1,\dots,k\}$ as per Theorem \ref{thm:equiv-H-space}. We define recursively the lift map $\mathbb{1}^{(\cdot)}:\mathfrak{I}^{(k, p)}_\phi(V)\to C_0(\triangle_J,T^{(k)}(V))$ as follows. At the first level, set $\mathbb{1}^{(0, I^1)}:=(1,I^1)\in\Omega^{q_1}(V)$. Given the recursive assumption $\mathbb{1}^{\left(0, I^1, \ldots, I^{k-1}\right)} \in \Omega^{q_{k-1}}(V)$ and Definition \ref{def:X pow i (417)}, we set:
\begin{align*}
\left(\mathbb{1}^I\right)^k &:= \left(\mathbb{1}^{\left(0, I^1, \ldots, I^{k-1}\right)}\right)^{I^k} \\
&\,=\operatorname{Ext}_{k-1}\left(\mathbb{1}^{\left(0, I^1, \ldots, I^{k-1}\right)}\right)+(0,\ldots,0,I^k),
\end{align*}
where $(0,\ldots,0,I^k)$ contains exactly $k-1$ zeros preceding $I^k$.
\end{definition}

The following proposition provides a direct link between the spaces $\mathfrak{I}_{\phi,\omega}^{(k,p)}(V)$ and $\mathscr{H}_{\phi,\omega}^{(k,p)}(V)$, paralleling Proposition \ref{prop: dev H is in I}. Its proof, being analogous, is omitted.

\begin{proposition}
  \label{prop:1 pow I_st j (620)}
  Let \(p \ge 1\) be a real number, \(k\ge 1\) an integer,   \(\phi \in \bigl(1-\tfrac{1}{p}, 1\bigr]\) a real number, and \(\omega\) a control.  Given \(I \in \mathfrak{I}_{\phi,\omega}^{(k,p)}(V)\), then \(\mathbb{1}^I \in \mathscr{H}_{\phi,\omega}^{(k,p)}(V)\), and if \(I\in WG\mathfrak{I}_{\phi,\omega}^{(k,p)}(V)\), then \(\mathbb{1}^I \in WG\mathscr{H}_{\phi,\omega}^{(k,p)}(V)\).

\end{proposition}

The following proposition and its corollary cement the relationship between $\mathfrak{I}_{\phi,\omega}^{(k,p)}(V)$ and $\mathscr{H}_{\phi,\omega}^{(k,p)}(V)$, proving they are bijective to each other.

\begin{proposition}[$\mathbb{1}$  and $\operatorname{dev}$ are mutual inverses]\label{prop:unit-dev}
  Let $p\ge 1$.  For every $H\in\mathscr H^{p}(V)$ one has $\mathbb 1^{\operatorname{dev}(H)}=H$,   and for every $I\in\mathfrak I^{p}(V)$ one has $\operatorname{dev}\!\bigl(\mathbb 1^{I}\bigr)=I$.
\end{proposition}

\begin{proof}
  Let $H\in\mathscr H^{p}(V)$. By definition of $\operatorname{dev}$,
  \begin{align}
    \operatorname{dev}(H)^{0} &:= 0, \\
    \operatorname{dev}(H)^{1} &:= H^{1}, \\
    \operatorname{dev}(H)^{k} &:= H^{k}-\operatorname{Ext}_{k-1}\!\bigl(1,\dots,H^{k}\bigr)^{k},
      \qquad k\ge 2. 
  \end{align}

  Hence, by the definition of the unit map $\mathbb 1^{(\,\cdot\,)}$,
  \[
    (\mathbb 1^{\operatorname{dev}(H)})^{0} \stackrel{\text{def.}}{=} 1,
    \qquad
    (\mathbb 1^{\operatorname{dev}(H)})^{1} := \operatorname{dev}(H)^{1}\stackrel{\text{def.}}{=} H^{1}.
  \]

  Assume inductively that $(\mathbb 1^{\operatorname{dev}(H)})^{j}=H^{j}$ for all $j\le k$.
  Then
  \begin{align*}
    (\mathbb 1^{\operatorname{dev}(H)})^{k+1}
      &=\operatorname{Ext}_{k}\!\bigl(1,(\mathbb 1^{\operatorname{dev}(H)})^{1},\dots,(\mathbb 1^{\operatorname{dev}(H)})^{k}\bigr)^{k+1}
        +\operatorname{dev}(H)^{k+1}\\
      &=\operatorname{Ext}_{k}\!\bigl(1,H^{1},\dots,H^{k}\bigr)^{k+1}
        +\Bigl(H^{k+1}-\operatorname{Ext}_{k}(1,\dots,H^{k})^{k+1}\Bigr)\\
      &=H^{k+1},
  \end{align*}
  completing the induction and proving $\mathbb 1^{\operatorname{dev}(H)}=H$.

  The converse identity $\operatorname{dev}(\mathbb 1^{I})=I$ for $I\in\mathfrak I^{p}(V)$ is analogous, starting from
  \[
    (\mathbb 1^{I})^{0}=1,\qquad
    (\mathbb 1^{I})^{1}=I^1,\qquad
    (\mathbb 1^{I})^{k+1}=\operatorname{Ext}_{k}\!\bigl(1,(\mathbb 1^{I})^{1},\dots,(\mathbb 1^{I})^{k}\bigr)^{k+1}+I^{k+1}.
  \]
\end{proof}

\begin{remark}[$\mathrm{WG}\,\mathfrak I^p(V)$ and $\mathfrak I^p(V)$ as paths]\label{rmk: mathfrak Ip(v) as paths}
An additive two–parameter functional $(s,t)\mapsto I_{s,t}$ on $\triangle_J$ determines a path on $J=[S,T]$ (via $I_{\cdot}:=I_{S,\cdot}$) and is determined by a path $t\mapsto I_t$ with $I_0=0$ via $I_{s,t}=I_t-I_s$. 

Take $\phi\in (1-1/p,\,1]$. For each $j=1,...,\lfloor p \rfloor$, consider  $q_j:=j/\phi\in [j, j+j/(p-1))$, set $\tilde q_j:=q_j/j\in[1,2)$ (Young regime); then, each degree $I^j:t\mapsto I^j_t\in V^{\otimes j}$ has finite $\tilde q_j$–variation, i.e.
$$\|I_{st}^j\|=\|I_t^j-I_s^j\|\le \omega(s,t)^{\phi} = \omega(s,t)^{\frac{1}{\tilde{q}_j}}.$$
Thus $\mathrm{WG}\,\mathfrak I^p(V)$ is the set of paths $I:J\to\mathfrak g^{(\lfloor p\rfloor)}(V)$ whose degrees $I^j$ have finite $\tilde q_j$–variation for some $\tilde{q}_j\in[1,2)$, whereas $\mathfrak I^p(V)$ is the set of paths $I:J\to T^{(\lfloor p\rfloor)}(V)$ with $I_0=0$ and the same degree–wise Young property.
\end{remark}

\begin{corollary}[Bijection of $\mathbb 1^{(\,\cdot\,)}$ and $\operatorname{dev}$]\label{cor: 1(.) is bijective}Let \(p\!\ge 1\) and \(k\!\in\!\{0,\dots,\lfloor p\rfloor\}\).
\begin{enumerate}[label=(\alph*)]
\item The unit map
\[
      \mathbb 1^{(\,\cdot\,)}:\;
      \mathfrak I^{(k,p)}(V)\;\longrightarrow\;
      \mathscr H^{(k,p)}(V)
\]
is a bijection whose inverse is \(\operatorname{dev}\).

\item For every control \(\omega\) and every exponent \(\phi\in\bigl(1-\tfrac1p,\,1\bigr]\) the restricted maps
\[
\begin{aligned}
   \mathbb 1^{(\,\cdot\,)} &: 
      \mathfrak I_{\phi,\omega}^{(k,p)}(V)\;\longrightarrow\;
      \mathscr H_{\phi,\omega}^{(k,p)}(V),\\[4pt]
   \mathbb 1^{(\,\cdot\,)} &: 
      \mathfrak I_{\phi}^{(k,p)}(V)\;\longrightarrow\;
      \mathscr H_{\phi}^{(k,p)}(V),\\[4pt]
   \mathbb 1^{(\,\cdot\,)} &: 
      \mathfrak I_{\omega}^{(k,p)}(V)\;\longrightarrow\;
      \mathscr H_{\omega}^{(k,p)}(V)
\end{aligned}
\]
are each bijective with inverse given by the corresponding restriction of \(\operatorname{dev}\).
\end{enumerate}
\end{corollary}

\begin{proof}
Let  \[
      F := \mathbb 1^{(\,\cdot\,)}, 
      \qquad 
      G := \operatorname{dev}.\]
\noindent\textbf{\underline{Firstly,}} let us recall the domain and codomain of $F$ and $G$,
\[
      F : \mathfrak I^{(k,p)}(V)\longrightarrow\mathscr H^{(k,p)}(V),
      \qquad
      G : \mathscr H^{(k,p)}(V)\longrightarrow\mathfrak I^{(k,p)}(V).
\]
By construction each image lies in the stated codomain. Proposition~\ref{prop:unit-dev} asserts \(F\!\circ G=\mathrm{id}_{\mathscr H^{(k,p)}(V)}\) and \(G\!\circ F=\mathrm{id}_{\mathfrak I^{(k,p)}(V)}\);  hence \(F\) is bijective and \(G\) is its inverse.

\medskip
\noindent\textbf{\underline{Secondly,}} fix a control \(\omega\) and an exponent \(\phi\in(1-\tfrac1p,1]\). Proposition~\ref{prop:1 pow I_st j (620)} gives
\[
     F\bigl(\mathfrak I_{\phi,\omega}^{(k,p)}(V)\bigr)
     \subseteq
     \mathscr H_{\phi,\omega}^{(k,p)}(V),
\]
while Proposition~\ref{prop: dev H is in I} gives
\[
     G\bigl(\mathscr H_{\phi,\omega}^{(k,p)}(V)\bigr)
     \subseteq
     \mathfrak I_{\phi,\omega}^{(k,p)}(V).
\]
Because the identities in Proposition~\ref{prop:unit-dev} hold on the larger spaces, their restrictions satisfy
\[
     F\!\circ G=\mathrm{id}_{\mathscr H_{\phi,\omega}^{(k,p)}(V)},
     \qquad
     G\!\circ F=\mathrm{id}_{\mathfrak I_{\phi,\omega}^{(k,p)}(V)},
\]
so \(F\) is a bijection between the \((\phi,\omega)\)-constrained subspaces, with inverse \(G\).

\medskip
\noindent\textbf{\underline{Thirdly,}} keep the exponent \(\phi\) fixed but allow the control to vary, i.e.
\[
      F : \mathfrak I_{\phi}^{(k,p)}(V)\longrightarrow\mathscr H_{\phi}^{(k,p)}(V),
      \qquad
      G : \mathscr H_{\phi}^{(k,p)}(V)\longrightarrow\mathfrak I_{\phi}^{(k,p)}(V).
\]
The same two references, interpreted with an arbitrary control, yield the required inclusions of images, and the restricted compositions remain the identity maps; thus \(F\) is bijective with inverse \(G\) on these spaces.

\medskip
\noindent\textbf{\underline{Fourthly,}} keep the control \(\omega\) fixed but allow the exponent to vary, i.e.
\[
      F : \mathfrak I_{\omega}^{(k,p)}(V)\longrightarrow\mathscr H_{\omega}^{(k,p)}(V),
      \qquad
      G : \mathscr H_{\omega}^{(k,p)}(V)\longrightarrow\mathfrak I_{\omega}^{(k,p)}(V).
\]
Again, Proposition~\ref{prop:1 pow I_st j (620)} and Proposition~\ref{prop: dev H is in I} apply without change, and the same two–sided inverse property shows bijectivity.

\medskip
Since in every case \(F\) is bijective and \(G\) is its inverse, the corollary is proved.
\end{proof}

These maps let us transport linear structure from $I$-space to $H$-space, a key ingredient in the main results.
	\newpage

\section{Lemmas}

\subsection{Top-level additive insertion}

Before proving the first lemma, we record a simple observation explaining why an additive functional \(\psi\) may be inserted at the highest level without affecting the lower levels of the sewn rough path.

Let \(\psi:\triangle_J\to V^{\otimes(\lfloor p\rfloor+1)}\) be additive, i.e.
\[
\psi_{s,t}=\sum_{i=0}^{m-1}\psi_{t_i,t_{i+1}}
\quad\text{for any partition }D=\{s=t_0<t_1<\cdots<t_{m}=t\}\subset[s,t].
\]
Given a functional \(Y:\triangle_J\to T((V))\) and a partition \(D\), write
\[
Y^D:=Y_{t_0,t_1}\otimes\cdots\otimes Y_{t_{m-1},t_{m}},
\]
and denote by \(Y^{D,j}\) the \(j\)-th tensor level of \(Y^D\). Consider two almost multiplicative functionals that agree up to level \(\lfloor p\rfloor\) and differ only at the top level:
\[
Z:=(1,X^1,\dots,X^{\lfloor p\rfloor},\psi),\qquad 
Z^{(0)}:=(1,X^1,\dots,X^{\lfloor p\rfloor},0).
\]
At level \(\lfloor p\rfloor+1\) the expansions along \(D\) satisfy
\[
Z^{D,\lfloor p\rfloor+1}-\big(Z^{(0)}\big)^{D,\lfloor p\rfloor+1}
=\sum_{i=0}^{m-1}\psi_{t_i,t_{i+1}}
=\psi_{s,t},
\]
because all terms built solely from the levels \(\le \lfloor p\rfloor\) cancel, and the only surviving contribution is the single insertion of \(\psi\) on one subinterval with trivial (level-$0$) factors elsewhere. Thus the two inputs have identical components up to level \(\lfloor p\rfloor\), and at level \(\lfloor p\rfloor+1\) they differ by exactly the additive increment \(\psi_{s,t}\).

Consequently, if we assume that $Z$ and $Z^{(0)}$ are almost rough paths, applying the level-$\lfloor p \rfloor +1$  sewing map, i.e. \(\mathscr S^{\lfloor p \rfloor +1}\),  to either input yields sewn rough paths that agree in all levels \(\le \lfloor p\rfloor\), while at the top level
\[
\mathscr S^{\lfloor p \rfloor +1}\big(1,X^1,\dots,X^{\lfloor p\rfloor},0\big)^{\lfloor p\rfloor+1}_{s,t}
\;+\;\psi_{s,t}
\;=\;
\mathscr S^{\lfloor p \rfloor +1}\big(1,X^1,\dots,X^{\lfloor p\rfloor},\psi\big)^{\lfloor p\rfloor+1}_{s,t}.
\]
No \(p\)-variation bound on \(\psi\) is needed for this telescoping identity; such bounds enter only later to ensure that the inputs of \(\mathscr S^{\lfloor p \rfloor +1}\) are almost rough paths. This is the precise sense in which an additive \(\psi\) can be “slipped in” at the highest level.

\begin{lemma}\label{lem: bring psi inside sewing map}
Let $p\ge 1$ and set $n:=\lfloor p\rfloor$. Let $X=(1,X^1,\dots,X^n)\in\Omega^p(V)$ be a $p$-rough path controlled by $\omega_X$. Let $\psi:\triangle_J\to V^{\otimes(n+1)}$ be additive and of finite $p'$-variation for some $p'\in[n+1,n+2)$, controlled by a (possibly different) control $\omega_\psi$ with some constant $K_\psi$ , i.e.
\[
\|\psi_{st}\|\le K_\psi\,\omega_\psi(s,t)^{(n+1)/p'}.
\]
Define
\[
Z^{(0)}:=(1,X^1,\dots,X^n,0),\qquad Z^{(\psi)}:=(1,X^1,\dots,X^n,\psi).
\]
Then
\[
Z^{(0)}\ \text{is }\tfrac{n+1}{p}\text{-almost }p\text{-rough},\qquad
Z^{(\psi)}\ \text{is }\tfrac{n+1}{p}\text{-almost }p'\text{-rough},
\]
and, writing  $\mathscr S^{n+1}:\Omega^{\mathrm{am},(n+1)}(V)\to\Omega^{(n+1)}(V)$ for the level $n+1$-sewing map (Theorem~\ref{thm:Rough-Sewing-Lemma}),
\[
\mathscr S^{n+1}\!\big(Z^{(0)}\big)\;+\;\psi\;=\;\mathscr S^{n+1}\!\big(Z^{(\psi)}\big).
\]
\end{lemma}

\begin{proof}

Set $\omega=\omega_X+\omega_\psi$, i.e. a combined control which controls the $p$-variation of both $X$ and $\psi$.

\medskip\noindent\textbf{\underline{Step 1}:}\textit{ $Z^{(0)}$ is $\dfrac{n+1}{p}$-almost $p$-multiplicative of degree $n+1$ and has finite $p$-variation (and has finite $p'$-variation as well).}

By Lemma~\ref{lem:ext-is-sewing}, $Z^{(0)}=\iota_n(X)$ is $\theta$-almost multiplicative of degree $n+1$ with $\theta=\frac{n+1}{p}>1$ and has finite $p$-variation controlled by $\omega$, hence lies in the domain of $\mathscr S^{n+1}$. Since $p'\ge p$, $Z^{(0)}$ also has finite $p'$-variation.

\medskip\noindent\textbf{\underline{Step 2:} $Z^{(\psi)}$ is $\dfrac{n+1}{p}$-almost $p'$-rough.}
For any $s\le u\le t$,
\begin{align}
&\big(Z^{(\psi)}_{su}\otimes Z^{(\psi)}_{ut}\big)^{n+1}-Z^{(\psi)}_{st}{}^{\,n+1}\notag\\
&\qquad=\sum_{j=1}^{n} X^j_{su}\otimes X^{\,n+1-j}_{ut}
\ +\ \psi_{su}\otimes 1\ +\ 1\otimes\psi_{ut}\ -\ \psi_{st}\notag\\
&\qquad=\sum_{j=1}^{n} X^j_{su}\otimes X^{\,n+1-j}_{ut},
\label{eq:Step2-X-only}
\end{align}
by additivity of $\psi$. A standard argument using the neo-classical inequality (together with the $p$-variation bounds of $X$ and super-additivity of $\omega$) gives
\begin{equation}\label{eq:Step1-bound}
\Big\|\sum_{j=1}^{n} X^j_{su}\otimes X^{\,n+1-j}_{ut}\Big\|
\ \le\ p\,\frac{\omega(s,t)^{(n+1)/p}}{\beta((n+1)/p)!}.
\end{equation}
Therefore $Z^{(\psi)}$ is $\theta$-almost multiplicative with $\theta=\frac{n+1}{p}>1$. Since $p'\ge p$ and using that $X$ has finite $p$-variation, $X$ must have finite $p'$-variation. Given that the first $n$ levels of $Z^{(\psi)}$ equal those of $X$, the first $n$ levels of $Z^{(\psi)}$ also have finite $p'$-variation. As we assumed that the $(n+1)$-st level of $Z^{(\psi)}$, i.e. $\psi$, has finite $p'$-variation, we know that altogether $Z^{(\psi)}$ has also have finite $p'$-variation. Hence $Z^{(\psi)}$ is a $\tfrac{n+1}{p}$-almost $p'$-rough path.

\medskip\noindent\textbf{\underline{Step 3: Sewing at level $n+1$ and the identity}}

Apply the sewing map $\mathscr S^{n+1}$ to both inputs. By Theorem~\ref{thm:Rough-Sewing-Lemma}\,(ii), Equation~\eqref{eqn: construction of associated rough path}, $\mathscr S^{n+1}(\cdot)^j$ depends only on input levels $\le j$; since $Z^{(0)}$ and $Z^{(\psi)}$ agree up to level $n$, we have
\[
\mathscr S^{n+1}\big(Z^{(0)}\big)^j=\mathscr S^{n+1}\big(Z^{(\psi)}\big)^j,\qquad j=1,\dots,n.
\]
For the top level, set $Z_D:=Z_{t_0,t_1}\otimes\cdots\otimes Z_{t_{m-1},t_m}$. A direct degree-$(n+1)$ expansion shows, for \emph{every} partition $D\subset [s,t]$,
\begin{equation}\label{eq:partition-identity}
\big(Z^{(\psi)}_D\big)^{n+1}-\big(Z^{(0)}_D\big)^{n+1}
=\sum_{i=0}^{m-1}\psi_{t_i,t_{i+1}}
=\psi_{s,t}.
\end{equation}
Therefore\begin{align*}
\mathscr S^{n+1}\big(Z^{(\psi)}\big)^{n+1}_{st}
&=\lim_{\substack{|D|\to 0\\ D\subset [s,t]}}\big(Z^{(\psi)}_D\big)^{n+1}
=\lim_{\substack{|D|\to 0\\ D\subset [s,t]}}\Big(\big(Z^{(0)}_D\big)^{n+1}+\psi_{st}\Big)\\
&=\lim_{\substack{|D|\to 0\\ D\subset [s,t]}}\big(Z^{(0)}_D\big)^{n+1}\;+\;\psi_{st}
=\mathscr S^{n+1}\big(Z^{(0)}\big)^{n+1}_{st}\;+\;\psi_{st}.
\end{align*}
Together with the equality in lower levels it gives \( \mathscr S^{n+1}(Z^{(0)})+\psi=\mathscr S^{n+1}(Z^{(\psi)}).\)
\end{proof}

\subsection{Lift truncation invariance}
The next lemma collects reconstruction and truncation invariance properties of the lift map and $\boxplus$, which will streamline level-by-level arguments later.

We will now present a key lemma which will be essential towards proving the main result of this work. Recall from Definition~\ref{def:truncations} that given $p\in\mathbb R_{\ge1}$, for \(I=(0,I^{1},\dots,I^{\lfloor p\rfloor})\in\mathfrak{I}^{p}(V)\) and an integer \(k\le\lfloor p\rfloor\),
\[
  I(k)=(0,I^{1},\dots,I^{k})
  \quad\text{and}\quad
  I[k]=(0,I^{1},\dots,I^{k},0,\dots,0),
\]
the latter having \(\lfloor p\rfloor-k\) trailing zeros.

\begin{lemma}[Lift reconstruction and truncation invariance]
\label{lem:unit-boxplus-triple}
 \label{lem:first one} 
\label{lem:unit-trunc-invariance}
\label{lem:unit-boxplus-trunc-invariance}

Let $p \ge 1$ and let
\[
I = (0,I^1,\dots,I^{\lfloor p\rfloor}), 
\qquad 
\widetilde I = (0,\widetilde I^1,\dots,\widetilde I^{\lfloor p\rfloor})
\in \mathfrak{I}^p(V).
\]
For $k \in \{0,\dots,\lfloor p\rfloor\}$ write
\[
  I(k) := (0,I^1,\dots,I^k),
  \qquad
  I[k] := (0,I^1,\dots,I^k,0,\dots,0),
\]
where $I[k]$ has $\lfloor p\rfloor-k$ trailing zeros (and similarly for $\widetilde I$). Then:

\smallskip
\noindent\textnormal{(i) Reconstruction from the lift.}
Fix $k \in \{1,\dots,\lfloor p\rfloor\}$ and assume
\begin{equation}\label{eq:triple-eq-lifts}
   \bigl(\mathbb{1}^I\bigr)^j
   =
   \bigl(\mathbb{1}^{\widetilde I}\bigr)^j
   \qquad
   \forall j\in\{1,\dots,k\}.
\end{equation}
Then
\[
   I^j = \widetilde I^j
   \qquad
   \forall j\in\{1,\dots,k\}.
\]

\smallskip
\noindent\textnormal{(ii) Unit–truncation invariance.}
Let $k \in \{0,\dots,\lfloor p\rfloor\}$. Then, for every $j\in\{0,\dots,k\}$,
\[
   \bigl(\mathbb{1}^I\bigr)^j
   =
   \bigl(\mathbb{1}^{I[k]}\bigr)^j
   =
   \bigl(\mathbb{1}^{I(k)}\bigr)^j .
\]

\smallskip
\noindent\textnormal{(iii) $\boxplus$–truncation invariance.}
Let $k \in \{1,\dots,\lfloor p\rfloor\}$. Then, for every $j\in\{1,\dots,k\}$,
\[
   \bigl(\mathbb{1}^I \boxplus \mathbb{1}^{\widetilde I}\bigr)^j
   \;=\;
   \bigl(\mathbb{1}^{I[k]} \boxplus \mathbb{1}^{\widetilde I[k]}\bigr)^j
   \;=\;
   \bigl(\mathbb{1}^{I(k)} \boxplus \mathbb{1}^{\widetilde I(k)}\bigr)^j.
\]
\end{lemma}

\begin{proof}
Fix a control $\omega$ and an exponent $\phi\in(1-\tfrac1p,1]$ such that $I,\widetilde I \in \mathfrak I^{p}_{\phi,\omega}(V)$ (possible by Corollary \ref{cor:I-vector-space}). We then write $q_j := j/\phi$ for $j=1,\dots,\lfloor p\rfloor$. We prove (i), (ii) and (iii) in turn.

\medskip
\noindent\textbf{(i) Reconstruction from the lift.}

Trivially,
\[
   I^1 = \bigl(\mathbb{1}^I\bigr)^1
       = \bigl(\mathbb{1}^{\widetilde I}\bigr)^1
       = \widetilde I^1 .
\]
Assume inductively that for some $j\in\{1,\dots,k-1\}$ we have
\begin{equation}\label{eq:triple-I-equal-up-to-j}
   (1,I^1,\dots,I^j)
   =
   (1,\widetilde I^1,\dots,\widetilde I^j).
\end{equation}
By the recursive definition of $\mathbb{1}^I$,
\[
   \bigl(\mathbb{1}^I\bigr)^{j+1}
   = \operatorname{Ext}_{j}((\mathbb{1}^{I})^1,\dots, (\mathbb{1}^{I})^j)^{j+1}+ I^{j+1}.
\]
Hence
\[
   I^{j+1}
   = \bigl(\mathbb{1}^I\bigr)^{j+1}
     -\operatorname{Ext}_{j}((\mathbb{1}^{I})^1,\dots, (\mathbb{1}^{I})^j)^{j+1}
\]
Using \eqref{eq:triple-eq-lifts} and \eqref{eq:triple-I-equal-up-to-j}, we can replace both the lift and the lower levels by their $\widetilde I$–counterparts:
\[
\begin{aligned}
   I^{j+1}
   &= \bigl(\mathbb{1}^{\widetilde I}\bigr)^{j+1}
      - \operatorname{Ext}_{j}((\mathbb{1}^{\widetilde{I}})^1,\dots, (\mathbb{1}^{\widetilde{I}})^j)^{j+1} \\
   &= \widetilde I^{j+1},
\end{aligned}
\]
again by the recursive definition of $\mathbb{1}^{\widetilde I}$. This closes the induction and proves (i).

\noindent\textbf{(ii) Unit–truncation invariance.}
We argue by induction on $j$.

\smallskip
\emph{Base cases}.
For $j=0$ and  $j=1$, we use the recursive definition of $\mathbb{1}^{(\cdot)}$ (Definition \ref{def:1_pow_I})  :
\[
   \bigl(\mathbb{1}^I\bigr)^0
   =
   \bigl(\mathbb{1}^{I[k]}\bigr)^0
   =
   \bigl(\mathbb{1}^{I(k)}\bigr)^0
   = 1; \qquad
   \bigl(\mathbb{1}^I\bigr)^1
   =
   \bigl(\mathbb{1}^{I[k]}\bigr)^1
   =
   \bigl(\mathbb{1}^{I(k)}\bigr)^1
   = I^1.
\]

\smallskip
\emph{Inductive step.}
Fix $j\in\{1,\dots,k-1\}$ and assume that for every $\ell\in\{0,\dots,j\}$ we have
\[
   \bigl(\mathbb{1}^I\bigr)^\ell
   =
   \bigl(\mathbb{1}^{I[k]}\bigr)^\ell
   =
   \bigl(\mathbb{1}^{I(k)}\bigr)^\ell.
\]
Using the recursive definition of $\mathbb{1}^I$ and the induction hypothesis,
\[
\begin{aligned}
   \bigl(\mathbb{1}^I\bigr)^{j+1}
   &=
   \operatorname{Ext}_{j}\!\bigl(
      1,
      (\mathbb{1}^I)^1,
      \dots,
      (\mathbb{1}^I)^j
   \bigr)^{j+1}
   + I^{j+1} \\
   &=
   \operatorname{Ext}_{j}\!\bigl(
      1,
      (\mathbb{1}^{I[k]})^1,
      \dots,
      (\mathbb{1}^{I[k]})^j
   \bigr)^{j+1}
   + I^{j+1} \\
   &= \bigl(\mathbb{1}^{I[k]}\bigr)^{j+1}.
\end{aligned}
\]
Repeating the same computation with $I(k)$ in place of $I[k]$ yields
\[
   \bigl(\mathbb{1}^I\bigr)^{j+1}
   = \bigl(\mathbb{1}^{I(k)}\bigr)^{j+1}.
\]
Thus the claim holds for $j+1$. By induction, (ii) holds for all $j\in\{0,\dots,k\}$.

\medskip
\noindent\textbf{(iii) $\boxplus$–truncation invariance.}
Fix $k\in\{1,\dots,\lfloor p\rfloor\}$. We prove the triple equality by induction on $j$.

\smallskip
\emph{Base case $j=1$.}
By the definition of $\boxplus$ at level~1,
\[
\begin{aligned}
   \bigl(\mathbb{1}^I \boxplus \mathbb{1}^{\widetilde I}\bigr)^1
   &= I^1 + \widetilde I^1 \\
   &= \bigl(\mathbb{1}^{I[k]} \boxplus \mathbb{1}^{\widetilde I[k]}\bigr)^1
    = \bigl(\mathbb{1}^{I(k)} \boxplus \mathbb{1}^{\widetilde I(k)}\bigr)^1.
\end{aligned}
\]

\smallskip
\emph{Inductive step.}
Take $\ell\in\{1,\dots,k-1\}$ and suppose that for all $j\in\{1,\dots,\ell\}$,
\begin{equation}\label{eq:triple-inductive-hyp}
   \bigl(\mathbb{1}^I \boxplus \mathbb{1}^{\widetilde I}\bigr)^j
   =
   \bigl(\mathbb{1}^{I[k]} \boxplus \mathbb{1}^{\widetilde I[k]}\bigr)^j
   =
   \bigl(\mathbb{1}^{I(k)} \boxplus \mathbb{1}^{\widetilde I(k)}\bigr)^j.
\end{equation}
We show that the same holds for $j=\ell+1$.

By definition of $\boxplus$ via sewing, for any $(s,t)\in\triangle_J$,
\[
\bigl(\mathbb{1}^I \boxplus \mathbb{1}^{\widetilde I}\bigr)^{\ell+1}_{st}
=
\lim_{\substack{|D|\to 0 \\ D\subset[s,t]}}
\bigl(1,
  (\mathbb{1}^I \boxplus \mathbb{1}^{\widetilde I})^1,
  \dots,
  (\mathbb{1}^I \boxplus \mathbb{1}^{\widetilde I})^\ell,
  (\mathbb{1}^I \otimes \mathbb{1}^{\widetilde I})^{\ell+1}
\bigr)^{D,\ell+1}.
\]

Now fix
\[
   (\widetilde I^*, I^*) \in \{(\widetilde I[k],I[k]), (\widetilde I(k),I(k))\}.
\]
The following argument will hold for either choice of $(\widetilde I^*,I^*)$, and thus yields the desired triple equality.

For any $(s,t)\in\triangle_J$,
\begin{align}
   \bigl(\mathbb{1}^I \otimes \mathbb{1}^{\widetilde I}\bigr)^{\ell+1}_{st}
   &=
   \sum_{m=0}^{\ell+1}
     \bigl(\mathbb{1}^I\bigr)^m_{st}
     \otimes
     \bigl(\mathbb{1}^{\widetilde I}\bigr)^{\ell+1-m}_{st}
     \nonumber\\
   &=
   \sum_{m=0}^{\ell+1}
     \bigl(\mathbb{1}^{I^*}\bigr)^m_{st}
     \otimes
     \bigl(\mathbb{1}^{\widetilde I^*}\bigr)^{\ell+1-m}_{st}
     \label{eq:triple-223}\\
   &=
   \bigl(\mathbb{1}^{I^*} \otimes \mathbb{1}^{\widetilde I^*}\bigr)^{\ell+1}_{st},
   \label{eq:triple-224}
\end{align}
where \eqref{eq:triple-223} follows from part~(ii): since $m\le\ell+1\le k$ and
$\ell+1-m\le \ell+1\le k$, we have
\[
   \bigl(\mathbb{1}^I\bigr)^m
   = \bigl(\mathbb{1}^{I[k]}\bigr)^m
   = \bigl(\mathbb{1}^{I(k)}\bigr)^m,
   \quad
   \bigl(\mathbb{1}^{\widetilde I}\bigr)^{\ell+1-m}
   = \bigl(\mathbb{1}^{\widetilde I[k]}\bigr)^{\ell+1-m}
   = \bigl(\mathbb{1}^{\widetilde I(k)}\bigr)^{\ell+1-m},
\]
and \eqref{eq:triple-224} is simply the definition of the $(\ell+1)$–st degree of the tensor product.

Using the inductive hypothesis \eqref{eq:triple-inductive-hyp} and
\eqref{eq:triple-224}, we obtain
\begin{align*}
   \bigl(\mathbb{1}^I \boxplus \mathbb{1}^{\widetilde I}\bigr)^{\ell+1}_{st}
   &=
   \lim_{\substack{|D|\to 0 \\ D\subset[s,t]}}
   \bigl(
      1,
      (\mathbb{1}^I \boxplus \mathbb{1}^{\widetilde I})^1,
      \dots,
      (\mathbb{1}^I \boxplus \mathbb{1}^{\widetilde I})^\ell,
      (\mathbb{1}^I \otimes \mathbb{1}^{\widetilde I})^{\ell+1}
   \bigr)^{D,\ell+1}
   \\
   &=
   \lim_{\substack{|D|\to 0 \\ D\subset[s,t]}}
   \bigl(
      1,
      (\mathbb{1}^{I^*} \boxplus \mathbb{1}^{\widetilde I^*})^1,
      \dots,
      (\mathbb{1}^{I^*} \boxplus \mathbb{1}^{\widetilde I^*})^\ell,
      (\mathbb{1}^{I^*} \otimes \mathbb{1}^{\widetilde I^*})^{\ell+1}
   \bigr)^{D,\ell+1}
   \\
   &=
   \bigl(\mathbb{1}^{I^*} \boxplus \mathbb{1}^{\widetilde I^*}\bigr)^{\ell+1}_{st}.
\end{align*}
Since this holds for either choice of $(I^*,\widetilde I^*)$,
\[
   \bigl(\mathbb{1}^I \boxplus \mathbb{1}^{\widetilde I}\bigr)^{\ell+1}
   =
   \bigl(\mathbb{1}^{I[k]} \boxplus \mathbb{1}^{\widetilde I[k]}\bigr)^{\ell+1}
   =
   \bigl(\mathbb{1}^{I(k)} \boxplus \mathbb{1}^{\widetilde I(k)}\bigr)^{\ell+1}.
\]
By induction on $\ell$, (iii) holds for all $j\in\{1,\dots,k\}$.
\end{proof}


The following lemma will be useful in Section \ref{sec:main-results}. 
\subsection{Recursive $\boxplus$ formula}
We now derive an explicit recursive expression for $1_I \boxplus 1_{\tilde I}$ at a fixed level, enabling the induction in the main theorem.

\begin{lemma}\label{lem:1I_plus_1I_tilde_eq_Ext_1I_km1_plus_1I_tilde_km1_plus_Ik_plus_I_tilde_k}
Let $p \ge 1$ and let $\ell \in \mathbb{N}$ with $\ell \le \lfloor p \rfloor$. Let
\[
    I = \bigl(0, I^1,\dots, I^\ell\bigr), \qquad 
    \widetilde{I} = \bigl(0,\widetilde{I}^1,\dots,\widetilde{I}^\ell\bigr)
    \in \mathfrak{I}^{(\ell,p)}(V).
\]
Take $\phi\in(1-1/p,1]$ to be such that $I,\widetilde I\in \mathfrak I^{(\ell, p)}_\phi(V)$ (possible by Corollary \ref{cor:I-vector-space}). Then we have
\[
\mathbb{1}^{I} \boxplus \mathbb{1}^{\widetilde{I}}
=
\begin{cases}
(1,\; I^1+\widetilde{I}^{\,1}), & \ell=1,\\[6pt]
\begin{aligned}[t]
\operatorname{Ext}_{\ell-1}\!\Bigl(
\mathbb{1}^{(0,I^1,\dots,I^{\ell-1})}
\boxplus
\mathbb{1}^{(0,\widetilde{I}^1,\dots,\widetilde{I}^{\,\ell-1})}
\Bigr)
&\\
+\,(0_{\mathbb{R}},0_V,\dots,0_{V^{\otimes(\ell-1)}},\,I^\ell+\widetilde{I}^{\,\ell})
\end{aligned}
& \ell\ge2,
\end{cases}
\]
where $q_{\ell-1}:=\frac{\ell-1}{\phi}$.
\end{lemma}

\begin{proof}[Proof of Lemma~\ref{lem:1I_plus_1I_tilde_eq_Ext_1I_km1_plus_1I_tilde_km1_plus_Ik_plus_I_tilde_k}]
Let us fix $\omega$ such that $I,\widetilde I\in \mathfrak I^{(\ell, p)}_{\phi,\omega}(V)$, which is possible by Corollary \ref{cor:I-vector-space}).

\textbf{\boldmath\underline{Step 0: Case $\ell=1$; short sewing argument}}

We consider $\ell=1$. By the definition of $\boxplus$ via sewing,
\[
\mathbb{1}^{(0,I^1)}\boxplus\mathbb{1}^{(0,\widetilde I^{\,1})}
=\mathscr S^1\!\big(\mathbb{1}^{(0,I^1)}\oplus\mathbb{1}^{(0,\widetilde I^{\,1})}\big)
=\mathscr S^1\!\big(1,\;I^1+\widetilde I^{\,1}\big),
\]
which follows by the definition of $\mathbb{1}$ (Definition \ref{def:1_pow_I}). Since $I^1$ and $\widetilde I^{\,1}$ are additive and have finite $q_1$-variation for some $q_1\in[1,2)$, the pair $Z:=(1,I^1+\widetilde I^{\,1})$ is already a $q_1$-rough path. Hence its sewing satisfies $\mathscr S(Z)=Z$, and
\[
\mathbb{1}^{(0,I^1)}\boxplus\mathbb{1}^{(0,\widetilde I^{\,1})}
=\mathscr S^1(Z)=(1,I^1+\widetilde I^{\,1}),
\]
which is the $\ell=1$ case.

\noindent\textbf{Standing assumption for the remainder.}
From now on we assume $\ell\ge2$ (and thus necessarily $p\ge2$). Throughout the remainder of the proof, $\mathscr S$ denotes the level-$\ell$ sewing map $\mathscr S^\ell$.

\textbf{\underline{Step 1: Insert the top level inside the sewing input}}

Set
\[
A:=\mathbb{1}^{\,I[\ell-1]}\boxplus \mathbb{1}^{\,\widetilde I[\ell-1]},\qquad 
q_{\ell-1}:=\frac{\ell-1}{\phi},\qquad 
\Psi:=I^\ell+\widetilde I^{\,\ell}.
\]
By Lemma~\ref{lem:ext-is-sewing} we have $\mathrm{Ext}_{q_{\ell-1}}(A(\ell-1))=\mathscr S\!\big(\iota(A(\ell-1)\big)$ with $\iota(A(\ell-1))=(1,A^1,\dots,A^{\ell-1},0)$. Since $A(\ell-1)$ is a $q_{\ell-1}$-rough path, Lemma~\ref{lem: bring psi inside sewing map} (applied to $\iota(A(\ell-1))$ and the additive top‑level increment $\Psi$ of finite $q_\ell$-variation) yields that
$$Z_\diamond:=(1,A^1,\dots,A^{\ell-1},\Psi).$$
is a $\theta$–almost $q_{\ell}$–rough path for $\theta:=\ell/q_{\ell-1}$ controlled by $\omega$, and also
\begin{equation}\label{eq:insert-top-level}
\mathscr S\big(\iota(A(\ell-1))\big)+(0,\dots,0,\Psi)=\mathscr S\bigl(1,A^1,\dots,A^{\ell-1},\Psi\bigr).
\end{equation}
Consequently,
\begin{equation}\label{eq:RHS-as-SZdiamond}
\operatorname{Ext}_{\ell-1}(A(\ell-1))+(0,\dots,0,\Psi)=\mathscr S\!\big(Z_\diamond\big).
\end{equation}

\textbf{\underline{Step 2: Write the LHS as a sewing and record the level‑wise consequence}}

By\newline Theorem~\ref{thm:X_otimes_H_is_ARP}, applied with $X=\mathbb{1}^{I}\in\Omega^{(\ell,p)}_\omega(V)$ and
$H=\mathbb{1}^{\widetilde I}\in \mathscr{H}^{(\ell,p)}_{\phi,\omega}(V)$, we obtain that $\mathbb{1}^{I} \oplus\mathbb{1}^{\widetilde I}$ is a $(\phi+\tfrac{1}{p})$–almost multiplicative functional of degree $\ell$ with finite $p$-variation controlled by $\omega$, and
\[
\mathbb{1}^{I}\boxplus\mathbb{1}^{\widetilde I}
=\mathscr S\!\big(\mathbb{1}^{I}\oplus \mathbb{1}^{\widetilde I}\big).
\]

By Theorem \ref{thm:stepwise-invariance-sewing} we know that 
$$\mathbb{1}^{I}\boxplus\mathbb{1}^{\widetilde I}
=\mathscr S\!\Big((1, (\mathbb{1}^{I}\boxplus\mathbb{1}^{\widetilde I})^1,..., (\mathbb{1}^{I}\boxplus\mathbb{1}^{\widetilde I})^{\ell -1},(\mathbb{1}^{I}+ \mathbb{1}^{\widetilde I})^\ell)\Big)$$
and that 
$$(1, (\mathbb{1}^{I}\boxplus\mathbb{1}^{\widetilde I})^1,..., (\mathbb{1}^{I}\boxplus\mathbb{1}^{\widetilde I})^{\ell -1},(\mathbb{1}^{I}+ \mathbb{1}^{\widetilde I})^\ell)$$

is also a $(\phi+\tfrac{1}{p})$–almost multiplicative functional of degree $\ell$ with finite $p$-variation controlled by $\omega$.

Now, Lemma \ref{lem:unit-boxplus-triple} (iii) tells us that
\[
A^j=\bigl(\mathbb{1}^{\,I[\ell-1]}\boxplus \mathbb{1}^{\,\widetilde I[\ell-1]}\bigr)^{j}=(\mathbb{1}^{I}\boxplus\mathbb{1}^{\widetilde I})^j\ (1\le j\le \ell-1);
\]
so, if we define
$$Z_+:=(1,A^1,\dots,A^{\ell-1},(\mathbb{1}^{I}+\mathbb{1}^{\widetilde I})^\ell),
$$
we have that 
$$\mathbb{1}^{I}\boxplus\mathbb{1}^{\widetilde I}
=\mathscr S(Z_+)$$
and that $Z_+$ is a $(\phi+\tfrac{1}{p})$–almost multiplicative functional of degree $\ell$ with finite $p$-variation controlled by $\omega$.

Finally, we see that $Z_+^j=Z_\diamond^j( = A^{j})$ for $j=1,\dots,\ell-1$. By the level‑wise dependence of $\mathscr S$ (Remark~\ref{rmk: if X^j =Y^j for all j, then S(X)^j = S(Y)^j for all j}), it follows that $\mathscr S(Z_+)^j=\mathscr S(Z_\diamond)^j$ for $j\le \ell-1$. Thus, only the degree-$\ell$ differences between $Z_+^j$ and $Z_\diamond^j$ will need to be controlled.

\textbf{\underline{Step 3: Compute the degree‑$\ell$ difference}}

At degree $\ell$,
\begin{align*}
(Z_+^\ell-Z_\diamond^\ell)_{st}
&=(\mathbb{1}^{I}+\mathbb{1}^{\widetilde I})_{st}^{\ell}-\Psi_{st}\\
&=\operatorname{Ext}_{\ell -1}\!\bigl(
      1,
      (\mathbb{1}^I)^1,
      \dots,
      (\mathbb{1}^I)^{\ell-1}
   \bigr)_{st}^{\ell}
  +\operatorname{Ext}_{\ell -1}\!\bigl(
      1,
      (\mathbb{1}^{\widetilde{I}})^1,
      \dots,
      (\mathbb{1}^{\widetilde{I}})^{\ell-1}
   \bigr)_{st}^{\ell}
\end{align*}

(with $\Psi_{st}=I_{st}^\ell+\widetilde{I}_{st}^\ell$ cancelling out the terms $I_{st}^\ell$ and $\widetilde{I}_{st}^\ell$). 

\textbf{\boldmath\underline{Step 4: Bound the degree-\(\ell\) difference}}

We begin Step 4 with an observation: we know that $ \mathbb{1}^I(\ell-1) $ and $\mathbb{1}^{\widetilde{I}}(\ell-1) $ are in $\Omega^{q_{\ell-1}}_{\omega}(V)$ because of the following reason. On the one hand,  Proposition \ref{prop:1 pow I_st j (620)} tells us that $I\in \mathfrak I^{(\ell, p)}_{\phi,\omega}(V) $ implies that $ \mathbb{1}^I \in \mathscr{H}^{(\ell, p)}_{\phi,\omega}(V) $. On the other hand, Theorem \ref{thm:equiv-H-space} tells us that  $ \mathbb{1}^I \in \mathscr{H}^{(\ell , p)}_{\phi,\omega}(V) $ implies that $ \mathbb{1}^I(\ell-1) \in \Omega^{q_{\ell-1}}_{\omega}(V)$  where $q_{\ell -1}:=(\ell-1)/\phi$.

Thus, defining $\theta:=\ell/q_{\ell-1}$, we have that
\begin{align*}
&\,\|\operatorname{Ext}_{\ell -1}\!\bigl(
      1,
      (\mathbb{1}^I)^1,
      \dots,
      (\mathbb{1}^I)^{\ell-1}
   \bigr)^{\ell}
  +\operatorname{Ext}_{\ell -1}\!\bigl(
      1,
      (\mathbb{1}^{\widetilde{I}})^1,
      \dots,
      (\mathbb{1}^{\widetilde{I}})^{\ell-1}
   \bigr)^{\ell}\|\\
\le&\, K_1 \frac{\omega(s,t)^{\frac{\ell}{q_{\ell-1}}}}{\beta(q_{\ell-1})(\frac{\ell}{q_{\ell -1}})} +K_2 \frac{\omega(s,t)^{\frac{\ell}{q_{\ell-1}}}}{\beta(q_{\ell-1})(\frac{\ell}{q_{\ell -1}})}
\le K \omega(s,t)^{\theta},
\end{align*}
where $K_1$ and $K_2$ are constants such that $ \mathbb{1}^I(\ell-1) = (1, (\mathbb{1}^I)^1,
      \dots,
      (\mathbb{1}^I)^{\ell-1})\in \Omega^{q_{\ell-1}}_{\omega}(V)$ and $ \mathbb{1}^{\widetilde{I}}(\ell-1) = (1, (\mathbb{1}^{\widetilde{I}})^1,
      \dots,
      (\mathbb{1}^{\widetilde{I}})^{\ell-1})\in \Omega^{q_{\ell-1}}_{\omega}(V)$, and where $$K:=  \frac{K_1 + K_2}{\beta(q_{\ell-1})(\frac{\ell}{q_{\ell -1}})}.$$

\textbf{\underline{Step 5: Verify the hypotheses and sew with a $j$‑independent bound}}

Using the uniform level‑wise estimate
\[
\|(Z_{+})_{st}^{\,j}-(Z_{\diamond})_{st}^{j}\|\ \le\ K\,\omega(s,t)^\theta\qquad(1\le j\le \ell)
\]

holds with the $j$–independent constant $K$. We apply Corollary~\ref{cor:X_Y_close_implies_SX_eq_SY} to the pair $(Z_+,Z_\diamond)$ with level $k=\ell$ (and sewing map $\mathscr S = \mathscr S^\ell$), using $q_\ell$ as the witnessing exponent, which allows us to conclude that 
$$\mathscr S(Z_+)=\mathscr S(Z_\diamond).$$
Using \eqref{eq:RHS-as-SZdiamond} and the representation $\mathbb{1}^{I}\boxplus\mathbb{1}^{\widetilde I}=\mathscr S(Z_+)$, we obtain
\[
\mathbb{1}^{I}\boxplus\mathbb{1}^{\widetilde I}
=\mathscr S(Z_+)
=\mathscr S(Z_\diamond)
=\operatorname{Ext}_{\ell-1}(A)+(0,\dots,0,\Psi).
\]
Substituting the definitions of  $A=\mathbb{1}^{\,I[\ell-1]}\boxplus \mathbb{1}^{\,\widetilde I[\ell-1]}$ and $\Psi=I^\ell+\widetilde I^{\,\ell}$ yields
\[
\mathbb{1}^{I}\boxplus\mathbb{1}^{\widetilde I}
=\operatorname{Ext}_{\ell-1}\!\Big(\mathbb{1}^{\,I[\ell-1]}\boxplus \mathbb{1}^{\,\widetilde I[\ell-1]}\Big)
+(0,\dots,0,\,I^\ell+\widetilde I^{\,\ell}),
\]
as claimed.
\end{proof}

    \newpage

\section{Main Results}\label{sec:main-results}

\subsection{Roadmap}

Below, we prove one of the central results in our work. The idea of the proof is as follows: we prove the identity level by level: assuming it holds up to degree $k$, Lemma~\ref{lem:1I_plus_1I_tilde_eq_Ext_1I_km1_plus_1I_tilde_km1_plus_Ik_plus_I_tilde_k} gives an explicit formula for the $(k+1)$-st component of $1_I \,\boxplus\, 1_{\tilde I}$ in terms of the first $k$ components, Lemma~ \ref{lem:unit-boxplus-triple} shows that truncation of $I,\tilde I$ does not affect these lower levels, and comparing with the recursive definition of the lift in Definition~\ref{def:1_pow_I} shows that this $(k+1)$-st component coincides with that of $1_{I+\tilde I}$, so induction on $k$ yields \eqref{eqn:1I_plus_1I_tilde_eq_1IplusI}.

\subsection{Lift addition identity}
We first show that, on lifted elements, $\boxplus$ coincides with pointwise addition in $I$-space via the lift map.

\begin{theorem}
\label{thm:1I_plus_1I_tilde_eq_1IplusI}
Let $p \in [1, \infty)$ and $\phi \in (1-\frac{1}{p}, 1]$. Let $I, \widetilde{I} \in \mathfrak{I}^{p}_{ \phi}(V)$. Then

\begin{equation}\label{eqn:1I_plus_1I_tilde_eq_1IplusI}
    \mathbb{1}^I  \boxplus \mathbb{1}^{\widetilde{I}} = \mathbb{1}^{I+\widetilde{I}} .
\end{equation}
 
\end{theorem}

\begin{proof}

\underline{\textbf{(A) Ensuring Equation \eqref{eqn:1I_plus_1I_tilde_eq_1IplusI} is well-posed}}

We will start by establishing that both sides of the identity \eqref{eqn:1I_plus_1I_tilde_eq_1IplusI} belong to the same space, which is easy to see. We know that since $I$ and $\widetilde{I}$ are in $\mathfrak{I}^p_\phi(V)$, we have that both $\mathbb{1}^I$ and $\mathbb{1}^{\widetilde{I}}$ are in $\mathscr{H}^p_{\phi}(V) \subset \Omega^p(V)$, hence we have that the pair $(\mathbb{1}^I, \mathbb{1}^{\widetilde{I}})$ is in $\Omega^p(V) \times \mathscr{H}^p_{\phi}(V) $. Given the map $\boxplus:\Omega^p(V) \times \mathscr{H}^p_{\phi}(V)  \to \Omega^p(V)$, we then have that $\mathbb{1}^I \boxplus \mathbb{1}^{\widetilde{I}}$ is well-defined and in $\Omega^p(V)$. On the other hand, $I+\widetilde{I} \in \mathfrak{I}^p_\phi(V)$, and thus $\mathbb{1}^{I+\widetilde{I}} \in \mathscr{H}^p_{\phi}(V) \subset \Omega^p(V)$. 

\underline{\textbf{(B) Induction: Base case}}

We will now proceed by induction. First, let us consider the base case:
\begin{equation}
    \left(\mathbb{1}^I \boxplus  \mathbb{1}^{\widetilde{I}}\right)^1 =\left(\mathbb{1}^{I+\widetilde{I}}\right)^1 \label{eqn:main-thm-base-case}.\end{equation}
By definition of $\mathbb{1}^{(\cdot)}$, the right-hand side in Equation \eqref{eqn:main-thm-base-case} equals $I^1+\tilde I^1$. The left-hand side, by definition of the operation $\boxplus$, is equal to the map

$$  (s,t) \mapsto \lim_{\substack{|D|\to0\\D\subset[s,t]}}
   \left(1, (\mathbb{1}^I\otimes\mathbb{1}^{\widetilde{I}})^1\right)^{D,1}.$$
But, we know that, for any pair $(s,t)$ in $\triangle_J$ and any given partition  $D=\{t_0=s,...,t_n=t\}\subset[s,t]$,

$$\begin{aligned}
      \left(1, (\mathbb{1}^I\otimes\mathbb{1}^{\widetilde{I}})^1\right)^{D,1} &=   \left(\left(1, (\mathbb{1}^I\otimes\mathbb{1}^{\widetilde{I}})^1\right)_{t_0,t_1}\otimes\ldots\otimes   \left(1, (\mathbb{1}^I\otimes\mathbb{1}^{\widetilde{I}})^1\right)_{t_{n-1},t_n}\right)^{1}\\
      &=   \sum_{k\in\{1,\ldots,n\}}    \left(1, (\mathbb{1}^I\otimes\mathbb{1}^{\widetilde{I}})^1\right)_{t_{k-1},t_k} ^{1}\\
      &=   \sum_{k\in\{1,\ldots,n\}}    (\mathbb{1}^I\otimes\mathbb{1}^{\widetilde{I}})^1_{t_{k-1},t_k}  \\
      &=   \sum_{k\in\{1,\ldots,n\}}     (\mathbb{1}^I)^1_{t_{k-1},t_k} + (\mathbb{1}^{\widetilde{I}})^1_{t_{k-1},t_k} \\
      &=   \sum_{k\in\{1,\ldots,n\}}   I^1_{t_{k-1},t_k} +  \widetilde{I}^1_{t_{k-1},t_k} \\
      &=   I^1_{st}+\tilde I^1_{st},
\end{aligned}$$
    where the last line follows from the additivity of $I^1$ and $\widetilde{I}^1$. This concludes the base case.
    
\underline{\textbf{(C) Induction: Inductive hypothesis}}

For the inductive step, take $k \in\{1, \ldots, \lfloor p \rfloor$ $
- 1\}$. Now, assume the identity holds $\forall j \leqslant k$, i.e.
\begin{equation} \label{eqn:main-thm-inductive-hyp}
\left(\mathbb{1}^I \boxplus  \mathbb{1}^{\widetilde{I}}\right)^j = \left(\mathbb{1}^{I+\widetilde{I}}\right)^j \quad \forall j=1, \ldots, k \text{.}
\end{equation}
Equation \eqref{eqn:main-thm-inductive-hyp} is the inductive hypothesis. 

We want to show this equation also holds for $j=k+1$. In other words, we would like to prove that for all $(s, t) $ in $ \triangle_J$
\begin{align}
	\operatorname{LHS} &:= \left(\mathbb{1}^I \boxplus \mathbb{1}^{\widetilde{I}}\right)_{s t}^{k+1}=\left(\mathbb{1}^{I+\widetilde{I}}\right)_{s t}^{k+1}=:\text{RHS}. \notag 
\end{align}

\vspace{2cm}
\underline{\textbf{(D.1) Induction: Inductive step, RHS}}

We will now manipulate
\[
\operatorname{RHS} \;=\; (\mathbb{1}^{I+\widetilde I})^{k+1}_{st}.
\]
To this end, we obtain the following chain of equalities: the first follows from the definition of $\mathbb{1}^{(\cdot)}$, the second from the inductive hypothesis \eqref{eqn:main-thm-inductive-hyp}, and the third from Lemma \ref{lem:unit-boxplus-trunc-invariance}, which identifies $(\mathbb{1}^{I}\boxplus\mathbb{1}^{\widetilde I})^{j}$ with $(\mathbb{1}^{I[k]}\boxplus\mathbb{1}^{\widetilde I[k]})^{j}$ for every $j\le k$:
\begin{align*}
\operatorname{RHS}
&=\Bigl[\operatorname{Ext}_{q_k}\bigl(1,(\mathbb{1}^{I+\widetilde I})^{1},\dots,(\mathbb{1}^{I+\widetilde I})^{k}\bigr)\Bigr]^{k+1}_{st}+I^{k+1}_{st}+\widetilde I^{k+1}_{st}\\
&=\Bigl[\operatorname{Ext}_{q_k}\bigl(1,(\mathbb{1}^{I}\boxplus\mathbb{1}^{\widetilde I})^{1},\dots,(\mathbb{1}^{I}\boxplus\mathbb{1}^{\widetilde I})^{k}\bigr)\Bigr]^{k+1}_{st}+I^{k+1}_{st}+\widetilde I^{k+1}_{st}\\
&=\Bigl[\operatorname{Ext}_{q_k}\bigl(1,(\mathbb{1}^{I[k]}\boxplus\mathbb{1}^{\widetilde I[k]})^{1},\dots,(\mathbb{1}^{I[k]}\boxplus\mathbb{1}^{\widetilde I[k]})^{k}\bigr)\Bigr]^{k+1}_{st}+I^{k+1}_{st}+\widetilde I^{k+1}_{st}.
\end{align*}

\underline{\textbf{(D.2) Induction: Inductive step, LHS}}

We now develop $\text{LHS}$. In the below chain of equalities, the first equality is given by the definition of $\text{LHS}$. 

Crucially, the second equality follows by Lemma \ref{lem:1I_plus_1I_tilde_eq_Ext_1I_km1_plus_1I_tilde_km1_plus_Ik_plus_I_tilde_k}.

The third equality uses the fact that  $  \bigl(\mathbb{1}^{I(k)} \boxplus \mathbb{1}^{\widetilde I(k)}\bigr)_{st}\in T^{(\lfloor p \rfloor )}(V)$ can be decomposed degree-wise: $$  \left(\mathbb{1}^{I(k)}\boxplus\mathbb{1}^{\widetilde I(k)}\right)_{st}  
    =\bigl(1,(\mathbb{1}^{I(k)}\boxplus\mathbb{1}^{\widetilde I(k)})^{1},\dots,(\mathbb{1}^{I(k)}\boxplus\mathbb{1}^{\widetilde I(k)})^{k}\bigr)_{st}.$$

The fourth and last equality follows by Lemma~\ref{lem:unit-boxplus-triple}\,(iii), which says that, for every $j \in \{1,\dots,k\}$ we have
\[
   \bigl(\mathbb{1}^{I(k)} \boxplus \mathbb{1}^{\widetilde I(k)}\bigr)^j
   = \bigl(\mathbb{1}^{I[k]} \boxplus \mathbb{1}^{\widetilde I[k]}\bigr)^j; 
\]

in other words, we replace $I(k)$ by $I[k]$, and $\widetilde I(k)$  by $\widetilde I[k]$. Hence the collection of lower levels entering $\operatorname{Ext}_{q_k}$ can be written as
\[
   \bigl(1,\,
   (\mathbb{1}^{I[k]} \boxplus \mathbb{1}^{\widetilde I[k]})^1,\dots,
   (\mathbb{1}^{I[k]} \boxplus \mathbb{1}^{\widetilde I[k]})^k \bigr).
\]
Putting it all together yields
\begin{align*}
    \operatorname{LHS} &= \left(  \mathbb{1}^I \boxplus \mathbb{1}^{\widetilde{I}}\right)^{k+1}_{st} \\
    &= \operatorname{Ext}_{q_k}\left(\mathbb{1}^{I(k)}\boxplus\mathbb{1}^{\widetilde I(k)}\right)^{k+1}_{st} + I^{k+1}_{st} + \widetilde I^{k+1}_{st}\\
    &=\Bigl[\operatorname{Ext}_{q_k}\bigl(1,(\mathbb{1}^{I(k)}\boxplus\mathbb{1}^{\widetilde I(k)})^{1},\dots,(\mathbb{1}^{I(k)}\boxplus\mathbb{1}^{\widetilde I(k)})^{k}\bigr)\Bigr]^{k+1}_{st}+I^{k+1}_{st}+\widetilde I^{k+1}_{st}\\
    &=\Bigl[\operatorname{Ext}_{q_k}\bigl(1,(\mathbb{1}^{I[k]}\boxplus\mathbb{1}^{\widetilde I[k]})^{1},\dots,(\mathbb{1}^{I[k]}\boxplus\mathbb{1}^{\widetilde I[k]})^{k}\bigr)\Bigr]^{k+1}_{st}+I^{k+1}_{st}+\widetilde I^{k+1}_{st}.
\end{align*}
Since $\operatorname{LHS}=\operatorname{RHS}$, the result follows by induction.

\end{proof}

\subsection{Associativity of perturbations}
We next establish associativity of successive $H$-space perturbations for a general base rough path $X$. This shows $\boxplus$ behaves like an action compatible with the internal addition on $H$-space.

The following is our second key result.
\begin{theorem}[Associativity of $\boxplus$] \label{thm:assoc-trivial-kernel}\label{thm:X_plus_H_plus_Htilde_eq_X_plus_H_Htilde}
Let \(p \in \mathbb{R}_{\ge 1}\). Let \(X \in \Omega^p(V)\) and let \(H, \widetilde H \in \mathscr H^{p}(V)\). Then
\begin{equation}\label{eqn:assoc_boxplus}
    (X \boxplus H) \boxplus \widetilde{H}
    = X \boxplus (H \boxplus \widetilde{H}).
\end{equation}
\end{theorem}

\begin{proof}
\underline{\textbf{Normalising $\phi$ and $\omega$}:}

Fix \(X \in \Omega^p(V)\) and \(H, \widetilde H \in \mathscr H^{p}(V)\). By definition of \(\mathscr H^{p}(V)\), choose exponents \(\phi_H,\phi_{\widetilde H} \in (1-\tfrac{1}{p},1]\) and controls \(\omega_X,\omega_H,\omega_{\widetilde H}\) such that \(X\) has finite \(p\)-variation controlled by \(\omega_X\), \(H \in \mathscr H^{p}_{\phi_H,\omega_H}(V)\), and \(\widetilde H \in \mathscr H^{p}_{\phi_{\widetilde H},\omega_{\widetilde H}}(V)\). Set
\[
\phi := \min\{\phi_H,\phi_{\widetilde H}\}
\qquad\text{and}\qquad
\omega := \omega_X + \omega_H + \omega_{\widetilde H}.
\]
Then \(\omega\) is a control (controls form a convex cone), and \(\omega_X,\omega_H,\omega_{\widetilde H}\le \omega\) pointwise on \(\triangle_J\). To see that \(H,\widetilde H \in \mathscr H^p_{\phi,\omega}(V)\), note that for any \((s,t)\in\triangle_J\) and \(1\le j\le \lfloor p\rfloor\),
\[
\begin{aligned}
\|H^j_{s,t}\|
&\le C_H\,\omega_H(s,t)^{\phi_H}
 \le C_H\,\omega(s,t)^{\phi_H} \\
&= C_H\,\omega(s,t)^{\phi}\,\omega(s,t)^{\phi_H-\phi}
 \le C_H\,\omega(s,t)^{\phi}\,\omega(S,T)^{\phi_H-\phi}.
\end{aligned}
\]

and similarly \(\|\widetilde H^j_{s,t}\|\le C_{\widetilde H}\,\omega(s,t)^{\phi}\,\omega(S,T)^{\phi_{\widetilde H}-\phi}\).

Thus, with
\[
K_H:=\max\Big\{C_H\,\omega(S,T)^{\phi_H-\phi},\; C_{\widetilde H}\,\omega(S,T)^{\phi_{\widetilde H}-\phi}\Big\},
\]
we have \(\|H^j_{s,t}\|,\ \|\widetilde H^j_{s,t}\|\le K_H\,\omega(s,t)^{\phi}\) for all \(1\le j\le \lfloor p\rfloor\), i.e. \(H,\widetilde H \in \mathscr H^p_{\phi,\omega}(V)\).

Moreover, since \(\omega_X\le \omega\), the same reasoning shows that \(X\) is controlled by \(\omega\). Hence, without loss of generality, we may work under a common pair \((\phi,\omega)\) and a single constant \(K_H\) controlling both \(H\) and \(\widetilde H\) in their definitional bounds.

\underline{\textbf{Proving the identity}}

Our proof strategy will consist in showing, for all $(s,t)$ in $\triangle_J$ and all $j$ in $\{1,..., \lfloor p \rfloor\}$,
\begin{equation}
    \label{ineq:key_assoc_X_plus_H_plus_Htilde_eq_X_plus_H_Htilde}
\left\|\bigl [\mathscr{S}(X \otimes H) \otimes \widetilde{H} - X \otimes \mathscr{S}(H \otimes \widetilde{H})\bigr]^j_{st} \right\| \leqslant  K \omega(s,t)^{\phi+\frac{1}{p}}
\end{equation}
for some $K\ge0$. By Corollary \ref{cor:X_Y_close_implies_SX_eq_SY}, Inequality \eqref{ineq:key_assoc_X_plus_H_plus_Htilde_eq_X_plus_H_Htilde} implies 
\begin{equation} \label{NEW: eqn:S_SXotimesH_otimesHtilde_eq_S_Xotimes_SHotimesHtilde}
 \mathscr{S} \big( \mathscr{S}(X \otimes H) \otimes \widetilde{H} \big) = \mathscr{S} \big( X \otimes \mathscr{S}(H \otimes \widetilde{H}) \big).   
\end{equation}
By the fact that $\mathscr{S}(X\otimes H)=\mathscr{S}(X\oplus H)$ given by Theorem \ref{thm:X_otimes_H_is_ARP}, Equation \ref{NEW: eqn:S_SXotimesH_otimesHtilde_eq_S_Xotimes_SHotimesHtilde} implies Equation \ref{eqn:assoc_boxplus} immediately.

Therefore, we will aim to prove Inequality \eqref{ineq:key_assoc_X_plus_H_plus_Htilde_eq_X_plus_H_Htilde}, which will imply the result of Theorem \ref{thm:X_plus_H_plus_Htilde_eq_X_plus_H_Htilde}. 

\begin{small}
\begin{align*}
&\left\| \left[ \mathscr{S}(X \otimes H) \otimes \widetilde{H} - X \otimes \mathscr{S}(H \otimes \widetilde{H}) \right]^j_{st} \right\| \\
&= \left\| \left[ \mathscr{S}(X \otimes H) \otimes \widetilde{H} - X \otimes H \otimes \widetilde{H} + X \otimes H \otimes \widetilde{H} - X \otimes \mathscr{S}(H \otimes \widetilde{H}) \right]^j_{st} \right\| \\
&= \left\| \left[ \left( \mathscr{S}(X \otimes H) - X \otimes H \right) \otimes \widetilde{H} + X \otimes \left( H \otimes \widetilde{H} - \mathscr{S}(H \otimes \widetilde{H}) \right) \right]^j_{st} \right\| \\
&= \left\| \sum_{k=0}^j \left( \mathscr{S}(X \otimes H) - X \otimes H \right)^k_{st} \otimes \widetilde{H}^{j-k}_{st} + \sum_{k=0}^j X^k_{st} \otimes \left( H \otimes \widetilde{H} - \mathscr{S}(H \otimes \widetilde{H}) \right)^{j-k}_{st} \right\| \\
& \leq \sum_{k=0}^j \bigl\| \bigl( \mathscr{S}(X \otimes H) - X \otimes H \bigr)^k_{st} \bigr\| \, \bigl\| \widetilde{H}^{j-k}_{st} \bigr\| + \sum_{k=0}^j \bigl\| X^k_{st} \bigr\| \, \bigl\| \bigl( H \otimes \widetilde{H} - \mathscr{S}(H \otimes \widetilde{H}) \bigr)^{j-k}_{st} \bigr\|.\tag{$\ast$}
\end{align*}
 \end{small}

By Theorem \ref{thm:X_otimes_H_is_ARP}, we know that $X\otimes H$ is a $(\phi+\frac{1}{p})$-almost rough path controlled by $\omega$. Since $X$ itself is a $p$-rough path controlled by $\omega$, we in particular have the level-wise estimate $\|X^{k}_{st}\| \le K_X \omega(s,t)^{k/p} \big/ (\beta(p)\bigl(k/p\bigr)!)$ for all $k \le \lfloor p \rfloor$ for some constant $K_X$. Moreover, since $H\in\Omega^p_\omega(V)$ by the definition of $\mathscr{H}^p_{\phi,\omega}(V)$, i.e. Definition \ref{def:h-space}, we know that, again by Theorem \ref{thm:X_otimes_H_is_ARP}, $H\otimes\widetilde{H}$ is a $(\phi+\frac{1}{p})$-almost rough path.

Equipped with the above understanding, we now apply Theorem \ref{thm:Rough-Sewing-Lemma} twice. Fix $k\in\{0,\dots ,j\}$ and $(s,t)\in\triangle_J$. First, we use the fact that $X\otimes H$ is a $(\phi+\frac{1}{p})$-almost rough path. Thanks to this, Theorem \ref{thm:Rough-Sewing-Lemma} tells us that there exists a constant $K_1$ such that the quantity $\|(\mathscr{S}(X \otimes H) - X \otimes H)^k_{st}\|$ is bounded by $K_1\,\omega(s,t)^{\phi+1/p}$. Secondly, we use the fact that $H\otimes\widetilde{H}$ is a $(\phi+\frac{1}{p})$-almost rough path. Similarly, Theorem \ref{thm:Rough-Sewing-Lemma} tells us that there exists another constant $K_2$ such that the quantity $\|(\mathscr{S}(H \otimes \widetilde H)-H \otimes \widetilde H)^{j-k}_{st}\|$ is bounded by $K_2\,\omega(s,t)^{\phi+1/p}$. We emphasise that both constants $K_1$ and $K_2$ are independent of the level $j$.

Hence, recalling that \(\widetilde H^0_{st}=1\), we split off the term \(k=j\) in the first sum in \((\ast)\) (since then \(j-k=0\) and the \(\omega^\phi\)-bound does not apply). We obtain

\[
\begin{aligned}
(\ast)
&\leqslant \sum_{k=0}^{j-1}\bigl\| \bigl( \mathscr{S}(X \otimes H) - X \otimes H \bigr)^k_{st} \bigr\| \, \bigl\| \widetilde{H}^{j-k}_{st} \bigr\|
      + \bigl\| \bigl( \mathscr{S}(X \otimes H) \\
&\qquad\qquad- X \otimes H \bigr)^j_{st} \bigr\|\,\bigl\|\widetilde H^0_{st}\bigr\| 
      + \sum_{k=0}^{j} \bigl\| X^k_{st} \bigr\| \, \bigl\| \bigl( H \otimes \widetilde{H} - \mathscr{S}(H \otimes \widetilde{H}) \bigr)^{j-k}_{st} \bigr\| \\
&\leqslant \sum_{k=0}^{j-1} K_1\,\omega(s,t)^{\phi+\frac{1}{p}} \, K_H\,\omega(s,t)^{\phi}
      + K_1\,\omega(s,t)^{\phi+\frac{1}{p}}
      \\
&\qquad\qquad+ \sum_{k=0}^{j} K_X\frac{\omega(s,t)^{\frac{k}{p}}}{\beta(p)\,(\frac{k}{p})!}\,K_2\,\omega(s,t)^{\phi+\frac{1}{p}} \\
&\le \omega(s,t)^{\phi +\frac1p} \left(
      K_1
      + \sum_{k=0}^{j-1} K_1K_H\,\omega(S,T)^\phi
      + \sum_{k=0}^{j}K_2K_X\,\frac{\omega(S,T)^{\frac{k}{p}}}{\beta(p)\,(\frac{k}{p})!}
      \right)\\
&\le \omega(s,t)^{\phi +\frac1p} \left(
      K_1
      + \lfloor p\rfloor\,K_1 K_H\,\omega(S,T)^\phi
      + K_2K_X\,\sum_{k=0}^{\lfloor p \rfloor}\frac{\omega(S,T)^{\frac{k}{p}}}{\beta(p)\,(\frac{k}{p})!}
      \right)\\
&=: K \,\omega(s,t)^{\phi +\frac{1}{p}},
\end{aligned}
\]

where the constant
\[
K
:= K_1
+ \lfloor p\rfloor\,K_1 K_H\,\omega(S,T)^\phi
+ K_2K_X\,\sum_{k=0}^{\lfloor p \rfloor}\frac{\omega(S,T)^{\frac{k}{p}}}{\beta(p)\,(\frac{k}{p})!}
\]
does not depend on \(j\) or \((s,t)\), and where we used the fact that \(\omega(s,t)^\alpha\le \omega(S,T)^\alpha\) for any \(\alpha\ge0\) by the super-additivity and non-negativity of the control \(\omega\).

\end{proof}

    \newpage

\section{Perturbing by almost rough paths}

\subsection{Kernel and identity}
We first identify the neutral element for $\boxplus$ and show that perturbation by $H$-space has trivial kernel: $X \boxplus H = X$ forces $H$ to be the identity.

\begin{lemma} \label{lem:identity_element_characterisation}Let $p\geq 1$. Let $X\in \Omega^p(V), H\in \mathscr{H}^p(V)$ and assume that
\begin{equation}
X \boxplus H = X.
\label{eqn:boxplus_identity}
\end{equation}
Then,
\[H=\mathbb{1}(\lfloor p \rfloor).\]
\end{lemma}

\begin{proof}
 
Write
\[
X = (1,X^1,\dots,X^{\lfloor p\rfloor}),\quad
H = (1,H^1,\dots,H^{\lfloor p\rfloor}).
\]
We now proceed by induction on the degree $k$ of the left- and right-hand sides. Consider the base case $k=1$ first, which is straightforward to verify:\begin{equation*}\begin{split}
     \phantom{\iff}\quad&\left(X\boxplus H\right)^{1}=\left(X\right)^{1}\\
     \iff \quad &    X^{1}+H^{1}=X^{1}\\
     \iff \quad&H^{1}=0=\mathbb{1}(\lfloor p \rfloor)^1.
   \end{split}\end{equation*}

We will now state the inductive hypothesis. Take $k\in\{1,\ldots \left\lfloor{p}\right\rfloor -1\}$. Now assume for induction that for all $ j\in \{1,2\ldots ,k\}$\[H^j=\mathbb{1}(\lfloor p \rfloor)^j=0.\]
   
We want to prove the inductive step, i.e. we want to show that $H^{k+1}=\mathbb{1}(\lfloor p \rfloor)^{k+1}=0.$

We begin by proving that $H^{k+1}:\triangle_J\to V^{\otimes (k+1)}$ is an additive functional. To this end, first observe that, since $H\in\mathscr{H}^p(V)$, there exists $\phi\in(1-1/p,1]$ such that $H\in\mathscr{H}^p_\phi(V)$. Then, recall that, since $H\in\mathscr{H}^p_\phi(V)$, Theorem \ref{thm:equiv-H-space} tells us that $H(k+1)=(1, H^1,\ldots\ldots, H^k, H^{k+1})$ is a $q_{k+1}(=(k+1)/\phi)$–rough path. Now, since we are assuming that $H^j=0$ for $j\in\{1,\ldots,k\}$, we can equally state that $H(k+1) = (1, 0_V,\ldots, 0_{V^{\otimes k}}, H^{k+1})$ is a $q_{k+1}$-rough path. Since $(1,0_V,\ldots, 0_{V^{\otimes k}},H^{k+1})\in \Omega^{q_{k+1}}(V)$ coincides with $\mathbb{1}(k+1)=(1,0_V,\ldots, 0_{V^{\otimes k}},0_{V^{\otimes (k+1)}})$ $\in$ $\Omega_1(V)\subset\Omega^{q_{k+1}}(V)$ up to level $k$ (and both are $q_{k+1}$-rough paths),  Part (a) of Lemma \ref{lem:mult. func.} tells us that their difference at level $k+1$, i.e. $H^{k+1}- 0 = H^{k+1}$, must be an additive functional in $V^{\otimes (k+1)}$.

Having established this fact which will later prove useful, we consider the following expression
   \[
   (X\boxplus H)^{k+1}_{st}=\lim_{\substack{|D|\rightarrow 0\\ D\subset[s,t]}}\left(1,\left(X\boxplus H\right)^1,\dots,\left(X\boxplus H\right)^k,\left(X\otimes H\right)^{k+1}\right)^{D,k+1}. \tag{$\ast$}\]
   
We know, by the theorem's main hypothesis, Equation \ref{eqn:boxplus_identity}, that, for all $j=1,...,k+1$\[\left(X\boxplus  H \right)^j=X^j.\]

Observe also that we have that, for any $(u,v)$ in $\triangle_J$,
   \begin{equation*}\begin{split}\left(X\otimes H\right)_{uv}^{k+1}&=\sum_{j=0}^{k+1}X_{uv}^{k+1-j}\otimes H_{uv}^j\\&=X_{uv}^{k+1}+H_{uv}^{k+1}+\sum_{j=1}^{k}X_{uv}^{k+1-j}\otimes H_{uv}^j.\end{split}\end{equation*}
But, since $H^j=\mathbb{1}(\lfloor p \rfloor)^j=0$ for all $j\in\{1,\ldots,k\},$
   \[\left(X\otimes H\right)_{uv}^{k+1}=X^{k+1}_{uv}+H_{uv}^{k+1}.\]
   
Given a partition $D$ of $[s,t]$, we will introduce the function $n(\cdot)$ which recovers the cardinality of $D$ minus one, namely $n(D) = \operatorname{Card}(D) -1$. 

We can now perform some manipulations on $(\ast)$. 
   \begin{align}
     &\left(X\boxplus H\right)_{st}^{k+1} \notag \\ 
& =\lim_{\substack{|D|\rightarrow 0\notag \\ D\subset[s,t]}}
     \left(1,\left(X\boxplus H\right)^1,\ldots,\left(X\boxplus H\right)^k,\left(X\otimes H\right)^{k+1}\right)^{D,k+1}\notag\\
     &
     =\lim_{\substack{|D|\rightarrow 0\notag\\ D\subset[s,t]}}\left(1,X^1,\ldots, X^k,X^{k+1}+H^{k+1}\right)^{D,k+1}\notag\\
     &
     =\lim_{\substack{|D|\rightarrow 0\notag\\ D\subset[s,t]}}\bigg(\sum_{\substack{(i_1,\ldots,i_{n(D)})\in \{0,1,\ldots,k+1\}^{n(D)}\notag\\ s.t.\sum_{l=1}^{n(D)}i_l=k+1}}\left(1,X^1,\ldots,X^k,X^{k+1}+H^{k+1}\right)_{t_0,t_1}^{i_1}\notag\\ &
     \otimes \ldots \otimes \left(1,X^1,\ldots,X^k,X^{k+1}+H^{k+1}\right)_{t_{n(D)-1},t_{n(D)}}^{i_{n(D)}}\bigg)\notag\\
     &=\lim_{\substack{|D|\rightarrow 0\\ D\subset[s,t]}}\bigg(
\sum_{i=0}^{n(D)-1}X_{t_i,t_{i+1}}^{k+1}
+H_{t_i,t_{i+1}}^{k+1}\notag \\[4pt]
&\qquad\qquad \qquad
+\sum_{\substack{(i_1,\ldots,i_{n(D)})\in \{0,1,\ldots,k+1\}^{n(D)}\\
s.t.\sum_{l=1}^{n(D)}i_l=k+1}}
\left(1,X^1,\ldots,X^k,0\right)_{t_0,t_{1}}^{i_1} \notag\\ &
      \otimes\ldots\otimes \left(1,X^1,\ldots,X^k,0\right)_{t_{n(D)-1},t_{n(D)}}^{i_{n(D)}}\bigg)\notag\\
      &
       =\lim_{\substack{|D|\rightarrow 0\notag\\ D\subset[s,t]}}\bigg(\sum_{i=0}^{n(D)-1}H_{t_i,t_{i+1}}^{k+1}+\sum_{\substack{(i_1,\ldots,i_n)\in \{0,1,\ldots,k+1\}^{n(D)}\notag\\ s.t.\sum_{l=1}^{n(D)}i_l=k+1}} \left(1,X^1,\ldots,X^k,X^{k+1}\right)_{t_0,t_{1}}^{i_1} \notag\\ &
       \otimes\ldots\otimes \left(1,X^1,\ldots,X^k,X^{k+1}\right)_{t_{n(D)-1},t_{n(D)}}^{i_{n(D)}}\bigg)\notag\\
      &
      =\lim_{\substack{|D|\rightarrow 0\notag\\ D\subset[s,t]}}\bigg(\sum_{i=0}^{n(D)-1}H_{t_i,t_{i+1}}^{k+1}+\left((X)^D\right)_{st}^{k+1}\bigg)\quad \downarrow\quad \text{definition of $X^D$}\notag\\
      &
      =\lim_{\substack{|D|\rightarrow 0\notag\\ D\subset[s,t]}}\bigg(\sum_{i=0}^{n(D)-1}H_{t_i,t_{i+1}}^{k+1}+\left(X\right)_{st}^{k+1}\bigg)\quad \downarrow\quad \text{multiplicativity}\notag\\
      &
      =\lim_{\substack{|D|\rightarrow 0\notag\\ D\subset[s,t]}}\bigg(\sum_{i=0}^{n(D)-1}H_{t_i,t_{i+1}}^{k+1}\bigg)+X_{st}^{k+1}.\notag\\
      &
      \text{Since $H^{k+1}$ is an additive functional,}\notag\\&
      =\lim_{\substack{|D|\rightarrow 0\notag\\ D\subset[s,t]}}\bigg(H_{st}^{k+1}\bigg)+ X_{st}^{k+1}\notag\\&
      =H_{st}^{k+1}+X_{st}^{k+1}.   \hspace{5em} \label{eq13}  
   \end{align} 
But, since by this theorem's hypothesis we have that
   \begin{equation}
     \label{eqn:X_boxplus_H}
     \left(X\boxplus H\right)^{k+1}_{st}=X^{k+1}_{st},
   \end{equation}

we then have that, by \eqref{eq13} and (\ref{eqn:X_boxplus_H})
    \begin{equation*}\begin{split}
     &X^{k+1}=X^{k+1}+H^{k+1}\\
     &\Leftrightarrow H^{k+1}=0.
     \end{split}\end{equation*}
\end{proof}


We now prove a simple result which, when combined with Lemma \ref{lem:identity_element_characterisation}, has strong consequences, yielding Corollary \ref{cor: X + H = X iff X = 1(p)}.

\begin{lemma}\label{lem:X_boxplus_1_equals_X}
    Let $p\geq 1$ and $X\in\Omega^p(V).$ Then $X\boxplus\mathbb{1}(\lfloor p \rfloor)=X$
\end{lemma}
\begin{proof}
   $ \left(X\boxplus \mathbb{1}(\lfloor p \rfloor)\right)^1=X^1$. Take some $k\in\{1,\ldots, \left\lfloor{p}\right\rfloor-1\}$. Now, assume  \newline $ \left(X\boxplus \mathbb{1}(\lfloor p \rfloor)\right)^j=X^j$ for all $ j=1,\ldots,k$. Then  
   \begin{equation*}\begin{split}       
   \left(X\boxplus \mathbb{1}(\lfloor p \rfloor)\right)_{st}^{k+1}
&=\lim_{\substack{|D|\rightarrow 0\\ D\subset[s,t]}}
\Bigl(
1,\left(X\boxplus \mathbb{1}(\lfloor p \rfloor)\right)^1,\ldots,
\left(X\boxplus \mathbb{1}(\lfloor p \rfloor)\right)^k, \\[4pt]
&\qquad
\left(X\otimes \mathbb{1}(\lfloor p \rfloor)\right)^{k+1}
\Bigr)^{D,k+1}\\&
   = \lim_{\substack{|D|\rightarrow 0\\ D\subset[s,t]}}\left(1,X^1,X^k,X^{k+1}\right)^{D,k+1}\quad \text{(inductive hypothesis)}\\&
   = \lim_{\substack{|D|\rightarrow 0\\ D\subset[s,t]}}\left(X_{st}^{k+1}\right)=X_{st}^{k+1}\quad \text{(multiplicativity)}
   \end{split}\end{equation*}
   
   Therefore, the Lemma is proved by induction.
\end{proof}

\begin{corollary} \label{cor: X + H = X iff X = 1(p)}
Let $X\in\Omega^p(V)$ and $H\in\mathscr{H}^{p}(V)$. Then 

$$X\boxplus H = X \iff H=\mathbb{1}(\lfloor p \rfloor)$$
\end{corollary}
\begin{proof}
Lemma \ref{lem:identity_element_characterisation} establishes the forward direction, while
Lemma \ref{lem:X_boxplus_1_equals_X} yields the reverse direction, completing the proof.
\end{proof}

\subsection{Extended lift notation}
We introduce shorthand notation for perturbing a base rough path by a lifted additive multi-index, and record basic consistency properties of this notation.

Thanks to Lemma \ref{lem:X_boxplus_1_equals_X} we can now broaden Definition \ref{def:X pow i (417)} to cover full multi-indices.
Given a multi-index

$$
I=(0,I^{1},\dots,I^{\lfloor p\rfloor}),
$$

whose components are additive functionals with values in the appropriate tensor powers of $V$, the lemma ensures that the construction of $X^{I}$ is compatible with our earlier notation for $\mathbb{1}^{I}$.

\begin{definition}\label{def:X_boxplus_lift}
Let $p\in\mathbb{R}_{\ge 1}$ and $k\in\{0,\dots,\lfloor p\rfloor\}$.
For $X\in\Omega^p(V)$ and $I\in\mathfrak{I}^{p}(V)$ we set

$$
\qquad X^{I}\coloneqq X\boxplus\mathbb{1}(\lfloor p\rfloor)^{I}.
$$

\end{definition}

\begin{remark}
Henceforth every occurrence of $\mathbb{1}^{I}$ is to be understood as the shorthand $\mathbb{1}(\lfloor p\rfloor)^{I}$.
Because $\mathbb{1}(\lfloor p\rfloor)$ is the neutral element for $\boxplus$,

$$
\mathbb{1}(\lfloor p\rfloor)^{I}
  =\mathbb{1}(\lfloor p\rfloor)\boxplus\mathbb{1}(\lfloor p\rfloor)^{I},
$$

independently of Definition \ref{def:X_boxplus_lift}.

Consequently, that definition genuinely extends Definition \ref{def:1_pow_I}: the two coincide when $X=\mathbb{1}(\lfloor p\rfloor)$.
\end{remark}

In what follows, we continue to establish properties of the key objects we have defined in this work. In particular, we now prove a property of the lift map.

\begin{lemma}\label{lem:unit_lift_of_zero_map}
    Let $p\geq 1.$ Let $\mathbf{0}(\lfloor p\rfloor)\in \mathfrak{I}^p(V)$ be the zero map, then 
    \[\mathbb{1}^{\mathbf{0}(\lfloor p\rfloor)}=\mathbb{1}(\lfloor p\rfloor)\in\Omega^{(\lfloor p \rfloor,1)}(V)\]
\end{lemma}

\begin{proof}
By Definition \ref{def:1_pow_I} of $\mathbb{1}^{(\cdot)}$ for $I\in \mathfrak{I}^P(V)$,\[\left(\mathbb{1}^{\mathbf{0}(\lfloor p\rfloor)}\right)^0=1\quad \text{and} \quad \left(\mathbb{1}^{\mathbf{0}(\lfloor p\rfloor)}\right)^1=\left(\mathbf{0}(\lfloor p\rfloor)\right)^1=0 \]
    Take $k\geq 1$ and assume $ \left(\mathbb{1}^{\mathbf{0}(\lfloor p\rfloor)}=\mathbb{1}(\lfloor p\rfloor)\right)^j$ for all $j=1,\ldots, k$. Then
    \begin{equation*}
         \begin{split}
             \left(\mathbb{1}^{\mathbf{0}(\lfloor p\rfloor)}\right)^{k+1}&=\operatorname{Ext}_k\left((1,(\mathbb{1}^{\mathbf{0}(\lfloor p\rfloor)})^1, \ldots,(\mathbb{1}^{\mathbf{0}(\lfloor p\rfloor)})^k)\right)^{k+1}+\mathbf{0}(\lfloor p\rfloor)^{k+1}\\&
             =\operatorname{Ext}_k\left((1,0_V,\ldots,0_{V^{\otimes \lfloor p \rfloor}})\right)^{k+1}+0\\&
             =0. \end{split}
     \end{equation*}
     The result follows by induction.
\end{proof}

\vspace{1em}
The definition below lays the groundwork for the ensuing theorem, which is a key result in our work.

\subsection{Vector space structure}
Using the bijection with $I$-space, we define scalar multiplication on $H$-space and deduce a full vector-space structure with addition given by $\boxplus$.

\begin{definition} \label{def:scalar_multiplication_on_H}
    Let $ a\in \mathbb{R}$, $p\in \mathbb{R}_{\geq 1}$, and $H\in \mathscr{H}^{p}(V). $ Then define $\odot :\mathbb{R}\times \mathscr{H}^p(V)\rightarrow \mathscr{H}^p(V)$ through:
    \[a\odot H \left(=a\odot \mathbb{1}^{\operatorname{dev}(H)}\right):=\mathbb{1}^{a \operatorname{dev}(H)}
    \]
    \end{definition}

\begin{theorem}\label{thm: Hp(V) is a vector space} Let $p \geq 1$. The space $\left( \mathscr{H}^p(V), \boxplus, \odot, \mathbb{R} \right)$ is a vector space.\end{theorem}

\begin{proof} 
Given that the bijection $\mathbb{1} : \mathfrak{I}^p(V) \to \mathscr{H}^p(V)$ is defined with inverse $\mathrm{dev} : \mathscr{H}^p(V) \to \mathfrak{I}^p(V)$, and that $\mathfrak{I}^p(V)$ is a vector space over $\mathbb{R}$, we aim to show that $\mathscr{H}^p(V)$ is also a vector space under the operations $\boxplus$ and $\odot$.

By a well-known theorem, if there exists a bijection between a vector space and another set, the latter can be endowed with a vector space structure by defining operations induced by the bijection. We denote these induced operations by $\boxplus_{\mathrm{ind}}$ and $\odot_{\mathrm{ind}}$.

For all $H_1, H_2 \in \mathscr{H}^p(V)$ and $\lambda \in \mathbb{R}$, the induced operations are defined as:
\begin{itemize}
    \item \textbf{Addition $\boxplus_{\mathrm{ind}}$:}
    \[
    H_1 \boxplus_{\mathrm{ind}} H_2 = \mathbb{1} ^{  \mathrm{dev}(H_1) + \mathrm{dev}(H_2)  }.
    \]
    
    \item \textbf{Scalar Multiplication $\odot_{\mathrm{ind}}$:}
    \[
    \lambda \odot_{\mathrm{ind}} H = \mathbb{1} ^{ \lambda \cdot \mathrm{dev}(H)}.
    \]
\end{itemize}

We know by Theorem \ref{thm:1I_plus_1I_tilde_eq_1IplusI} that $\boxplus$ coincides with $\boxplus_{\mathrm{ind}}$. Additionally, by Definition \ref{def:scalar_multiplication_on_H}, $\odot$ coincides with $\odot_{\mathrm{ind}}$.

Since $\mathfrak{I}^p(V)$ satisfies the vector space axioms and $\boxplus$ and $\odot$ are defined to coincide with the induced operations $\boxplus_{\mathrm{ind}}$ and $\odot_{\mathrm{ind}}$, it follows that $\mathscr{H}^p(V)$ satisfies these axioms under $\boxplus$ and $\odot$.

Therefore, $\left( \mathscr{H}^p(V), \boxplus, \odot, \mathbb{R} \right)$ is a vector space.

\end{proof}

\subsection{Almost \textit{H}-space}
We now broaden admissible perturbations by allowing $H$ to be almost multiplicative while retaining the same level-wise regularity bound.

Now, we task ourselves with the question of whether extending $H$-space to allow elements $H$ to be only almost rough paths, rather than rough paths, may expand the space of displacements $X\boxplus H$ of an arbitrary rough path $X$, the answer to which will ultimately be that it does not. Before answering this question, some preliminaries are in order:

\begin{definition}
Let $p\ge1$, $\phi \in(1-1/p,1]$, and $\theta>1$ be real numbers, and let $\omega$ be a control. Then define
$$\begin{aligned}\mathscr{H}_{\phi,\theta,\omega}^{\text{am},p}(V)=\{H\in\Omega^{\text{am},p}_{\omega,\theta}(V)\,| &\, \exists K>0 \text{ s.t. }\forall j \in \{1,...,\lfloor p \rfloor\}\text{ and } \forall(s,t)\in\triangle_J, \\
&\|H^j_{st}\|\le K\,\omega(s,t)^{\phi}\}.
\end{aligned}$$

\medskip
Moreover, we define all possible partial and full unions:
\[
\begin{aligned}
\mathscr{H}_{\omega,\theta}^{\mathrm{am},p}(V)
&:=\bigcup_{\phi\in(1-1/p,\,1]}\mathscr{H}_{\phi,\theta,\omega}^{\mathrm{am},p}(V), 
&\quad
\mathscr{H}_{\phi,\theta}^{\mathrm{am},p}(V)
&:=\;\;\bigcup_{\omega\,\text{control}}\:\:\:\mathscr{H}_{\phi,\theta,\omega}^{\mathrm{am},p}(V),\\[6pt]
\mathscr{H}_{\phi,\omega}^{\mathrm{am},p}(V)
&:=\quad\,\,\,\bigcup_{\theta>1}\quad\:\mathscr{H}_{\phi,\theta,\omega}^{\mathrm{am},p}(V), 
&\quad
\mathscr{H}_{\theta}^{\mathrm{am},p}(V)
&:=\bigcup_{\substack{\phi\in(1-1/p,\,1]\\\omega\,\text{control}}}
   \mathscr{H}_{\phi,\theta,\omega}^{\mathrm{am},p}(V),\\[6pt]
\mathscr{H}_{\omega}^{\mathrm{am},p}(V)
&:=\,\bigcup_{\substack{\phi\in(1-1/p,\,1]\\\theta>1}}
   \mathscr{H}_{\phi,\theta,\omega}^{\mathrm{am},p}(V), 
&\quad
\mathscr{H}_{\phi}^{\mathrm{am},p}(V)
&:=\;\;\,\bigcup_{\substack{\omega\,\text{control}\\\theta>1}}
   \:\:\,\mathscr{H}_{\phi,\theta,\omega}^{\mathrm{am},p}(V),\\[6pt]
\mathscr{H}^{\mathrm{am},p}(V)
&:=\,\bigcup_{\substack{\phi\in(1-1/p,\,1]\\\omega\,\text{control}\\\theta>1}}
   \mathscr{H}_{\phi,\theta,\omega}^{\mathrm{am},p}(V).
\end{aligned}
\]
We refer by ``almost $H$-space'' to the largest union, namely $\mathscr{H}^{\mathrm{am},p}(V)$.
\end{definition}

The theorem below, like Theorem \ref{thm:X_otimes_H_is_ARP}, will be stated without proof.  It is merely a straightforward generalisation of that result, and its proof is virtually identical –  the almost-multiplicativity of $H$ does not materially affect the mechanics of the proof.

\begin{theorem}
\label{thm:X_otimes_H_is_ARP_when_H_is_an_ARP}Let $p\ge1$, $\phi\in(1-\tfrac1p,1]$, and let $\omega$ be a control. For $X\in\Omega^{(k,p)}(V)$ and $H\in\mathscr{H}_{\phi,\omega}^{\mathrm{am},(k,p)}(V)$ the map \((X\oplus H):\triangle_J \to T^{(k)}(V) \) is a $(\phi+1/p)$–almost $p$–rough path. We then set
\[
  X\boxplus H:=\mathscr{S}^k(X\oplus H).
\]
\end{theorem}

The below proposition shows that the map $\mathscr{S}^{\lfloor p \rfloor}:\mathscr{H}_{\phi,\omega}^{\text{am},p}(V)\subset\Omega^{\text{am},p}_\omega(V)\to\Omega^p_\omega(V)$ takes values, in fact, in $\mathscr{H}_{\phi,\omega}^{p}(V)$. In other words, the sewing of an \textit{almost} multiplicative $H$ in \textit{almost} $H$-space is a \textit{fully} multiplicative $H$ in $H$-space.

\begin{proposition}\label{prop:S_preserves_H_space}

Let $p \geq 1$. Let $\omega$ be a control. Consider $\widetilde{H} \in \mathscr{H}_{\phi,\omega}^{\text{am},p}(V)$. Then:
\[
\mathscr{S}^{\lfloor p \rfloor}(\widetilde{H}) \in \mathscr{H}_{\phi,\omega}^p(V).
\]
\end{proposition}

\begin{proof}
Since $\widetilde{H} \in \mathscr{H}_{\phi,\omega}^{\text{am},p}(V)$, there exist $\phi \in (1 - \frac{1}{p}, 1]$, a control $\omega$ and a constant $K_1\ge0$ such that, for all $ j = 1, \dots, \lfloor p \rfloor$, and for all $ (s,t) \in \triangle_J$:
\[
\|\widetilde H_{st}^j\| \leq K_1\omega(s,t)^\phi.
\]

Setting $\mathscr{S}^{\lfloor p \rfloor}(\widetilde{H}) =: H$, we know that there exists a $ K_2 > 0$ and $\theta > 1$ such that for all $ j = 1, \dots, \lfloor p \rfloor$ and for all $ (s,t) \in \triangle_J$:
\[
\|H_{st}^j - \widetilde{H}_{st}^j\| \leq K_2\, \omega(s,t)^\theta = K_2\,\omega(s,t)^{\theta-\phi}\,\omega(s,t)^{\phi}.
\]

Let $\varepsilon := \theta - \phi $. Because $\theta>1$ and $\phi\le 1$, we have $\theta>\phi$, and thus $\varepsilon>0$. Then, by the super-additivity and positivity of $\omega$, 
\[
\|H_{st}^j - \widetilde{H}_{st}^j\| \leq K_2\, \omega(S,T)^\varepsilon \, \omega(s,t)^\phi =: K_3 \, \omega(s,t)^\phi.
\]

Hence, we know that for all $ j = 1, \dots, \lfloor p \rfloor$ and for all $ (s,t) \in \triangle_J$:
\[
\|H_{st}^j\| \leq \|\widetilde{H}_{st}^j\| + \|\widetilde{H}_{st}^j - {H}_{st}^j\| \leq K_1 \omega(s,t)^\phi + K_3 \, \omega(s,t)^\phi =: K \, \omega(s,t)^\phi.
\]

Thus, $\mathscr{S}^{\lfloor p \rfloor}(\widetilde{H}) = H \in \mathscr{H}_{\phi,\omega}^p(V)$. 
\end{proof}

Proposition~\ref{prop:S_preserves_H_space} is important to ensure that the statement of Theorem~\ref{thm:X_boxplus_almost_H_independent} below is well-defined, in particular the second equality in Equation~\eqref{eq132}. This theorem below shows that if two almost multiplicative functionals $\widetilde{H}$ and $\widehat{H}$ in almost $H$-space have the same associated rough path $H$, then the displacements of $X$ they induce, $X \boxplus \widetilde{H} $ and $ X \boxplus \widehat{H}$ must be equal.

\subsection{Displacements depend on sewing}
We show that, for almost $H$-space perturbations, the displacement $X \boxplus \tilde H$ depends only on the sewn rough path $S(\tilde H)$, and we prove the corresponding converse.

\begin{theorem}\label{thm:X_boxplus_almost_H_independent}Let $p \in [1, \infty)$, $\phi \in (1 - \frac{1}{p}, 1]$, and $\omega$ be a control. Let $\theta > 1$. Let $\widetilde{H}, \widehat{H} \in \mathscr{H}_{\phi, \theta, \omega}^{\text{am},p}(V)$. Assume that:
\[ \mathscr{S}^{\lfloor p \rfloor}(\widetilde H) = \mathscr{S}^{\lfloor p \rfloor}(\widehat{H}) =: H. \]
Then, for any $X \in \Omega_{\omega}^p(V)$:
\begin{equation}\label{eq132}
X \boxplus \widetilde{H} = X \boxplus \widehat{H} = X \boxplus H. 
\end{equation}
\end{theorem}

\begin{proof}
Set $\mathscr{S}:=\mathscr{S}^{\lfloor p \rfloor}$. To begin with, we observe that, by Proposition \ref{prop:S_preserves_H_space} we know that $H=\mathscr{S}(\widetilde H)$ is in $\mathscr{H}^{p}_{\phi,\omega}(V)$ since $\widetilde{H}$ is in $\mathscr{H}^{\text{am},p}_{\phi,\omega}(V)$. By Theorem \ref{thm:X_otimes_H_is_ARP}.,  we thus know that $X \otimes H $ is a $(\phi+1/p)$–almost $p$-rough path controlled by $\omega$.

By Theorem \ref{thm:Rough-Sewing-Lemma}, since $\mathscr{S}(\widetilde{H}) = H$, and $\widetilde{H}$ is a $\theta$-almost $p$-rough path controlled by $\omega$, we have that there exists a $ K_1 \geq 0$ such that , for all $ j = 1, \dots, \lfloor p \rfloor$ and for all $ (s,t) \in \triangle_J$:
\begin{equation}\label{eq31}
\|\widetilde{H}_{st}^j - H_{st}^j\| \leq K_1 \, \omega(s,t)^\theta. 
\end{equation}

Likewise, since $\mathscr{S}(\widehat{H}) = H$, and $\widehat{H}$ is also a $\theta$-almost $p$-rough path controlled by $\omega$, we have that there exists a $ K_2 \geq 0$ such that , for all $ j = 1, \dots, \lfloor p \rfloor$ and for all $ (s,t) \in \triangle_J$:
\begin{equation}\label{eq32}
\|\widehat{H}_{st}^j - H_{st}^j\| \leq K_2 \, \omega(s,t)^\theta. 
\end{equation}

Set $\theta^{\ast} := \min\{\theta,\phi+1/p\}>1$. From $\omega(s,t)^\theta \le \omega(S,T)^{\theta-\theta^{\ast}}\omega(s,t)^{\theta^{\ast}}$ and Theorem \ref{thm:X_otimes_H_is_ARP_when_H_is_an_ARP},  $X\otimes\widetilde H$ and $X\otimes\widetilde{\widetilde H}$, being $(\phi+1/p)$–almost, are also $\theta^{\ast}$–almost (up to constants).

Now, using \eqref{eq31} and \eqref{eq32}, we can compute the following estimate:
\begin{align*}
&\phantom{=}\:\,\,\bigg\|\Big( (X \otimes \widetilde{H})_{st} - (X \otimes \widehat{H})_{st} \Big)^j\bigg\|\\
&=\bigg\|\Big( X_{st} \otimes \widetilde{H}_{st} - X_{st} \otimes \widehat{H}_{st} \Big)^j\bigg\|\\
&= \bigg\|\Big( X_{st} \otimes \widetilde{H}_{st} - X_{st} \otimes H_{st} + X_{st} \otimes H_{st} - X_{st} \otimes \widehat{H}_{st} \Big)^j\bigg\| \\[4pt]
&\le \bigg\|\Big( X_{st} \otimes (\widetilde{H} - H)_{st} \Big)^j\bigg\| + \bigg\|\Big( X_{st} \otimes (H - \widehat{H})_{st} \Big)^j\bigg\| \\[4pt]
&= \sum_{k=0}^j \Bigl( \|X_{st}^k\|\,\|(\widetilde{H}_{st}-H_{st})^{j-k}\| + \|X_{st}^k\|\,\|(H_{st}-\widehat{H}_{st})^{\,j-k}\| \Bigr) \\[4pt]
&\le \sum_{k=0}^j \frac{\omega(s,t)^{k/p}}{\beta(p)(\frac{k}p)!}\,K_1\,\omega(s,t)^\theta + \frac{\omega(s,t)^{k/p}}{\beta(p)(\frac{k}p)!}\,K_2\,\omega(s,t)^\theta \\[4pt]
&\le \Bigl( \,(K_1 + K_2)\sum_{k=0}^j \frac{\omega(s,t)^{k/p}}{\beta(p)(\frac{k}p)!} \Bigr)\, \omega(S,T)^{\theta-\theta^{\ast}}\omega(s,t)^{\theta^*} \\[4pt]
&=: K_3\,\omega(s,t)^{\theta^\ast}.
\end{align*}

Hence, by Corollary \ref{cor:X_Y_close_implies_SX_eq_SY} applied with exponent $\theta^{\ast}$, we obtain
\[
\mathscr{S}(X \otimes \widetilde{H}) = \mathscr{S}(X \otimes \widehat{H}).
\]

By the same reasoning with exponent $\theta^{\ast}$ (using \eqref{eq31}), we also have
\[
\mathscr{S}(X \otimes H) = \mathscr{S}(X \otimes \widetilde{H}).
\]

Finally, by definition of the operation $\boxplus$ in Theorem \ref{thm:X_otimes_H_is_ARP_when_H_is_an_ARP}, we know that:
\[
X \boxplus H = \mathscr{S}(X \otimes H), \quad X \boxplus \widetilde{H} = \mathscr{S}(X \otimes \widetilde{H}), \quad \text{and} \quad X \boxplus \widehat{H} = \mathscr{S}(X \otimes \widehat{H}).
\]
Therefore, we have that:
\[
X \boxplus H = X \boxplus \widetilde{H} = X \boxplus \widehat{H}.
\]
\end{proof}

The following theorem provides a converse result to the one of Theorem \ref{thm:X_boxplus_almost_H_independent}, that is to say, two functionals $H$ and $\widetilde{H}$ in almost $H$-space inducing the same displacements of an $X$, i.e. $X \boxplus \widetilde{H} = X \boxplus \widetilde{G}$, means the functionals are equal up to sewing.

\begin{theorem}\label{thm: X+ tilde H = X+ tilde G => S(tilde H)=S(tilde G)}
  Let \( p \geq 1 \). Take \( X \in \Omega^p(V) \) and \( \widetilde{H}, \widetilde{G} \in \mathscr{H}^{\text{am},p}(V) \). Assume that:
\begin{equation}\label{eq917}
X \boxplus \widetilde{H} = X \boxplus \widetilde{G}, 
\end{equation}
then:
\[
\mathscr{S}^{\lfloor p \rfloor}(\widetilde{H}) = \mathscr{S}^{\lfloor p \rfloor}(\widetilde{G}).
\]

\end{theorem}

\begin{proof}
Set $\mathscr{S}:=\mathscr{S}^{\lfloor p \rfloor} $. Set \( G := \mathscr{S}(\widetilde{G}) \) and \( H := \mathscr{S}(\widetilde{H}) \).

Because the sewing map $\mathscr{S}$ fixes every genuine rough path, we have

$$
H=\mathscr{S}(H)=\mathscr{S}(\widetilde H)
\qquad\text{and}\qquad 
G=\mathscr{S}(\widetilde G).
$$

Applying Theorem \ref{thm:X_boxplus_almost_H_independent} with the pairs $(\widetilde H,H)$ and $(\widetilde G,G)$ therefore allows us to replace the almost-rough perturbations by their sewn rough paths inside the $\boxplus$ operation:
$$
X\boxplus\widetilde H \;=\; X\boxplus H,
\qquad
X\boxplus\widetilde G \;=\; X\boxplus G.
$$Thus, by Equation~\eqref{eq917}:
\begin{equation}\label{eq918}
X \boxplus H = X \boxplus G. 
\end{equation}

Since $\mathscr{H}^p(V)$ is a vector space by Theorem \ref{thm: Hp(V) is a vector space}, $H$ has an inverse. Denote it by $-(H)=-1 \odot H$. By Equation \eqref{eq918}, we know that:
\begin{equation}\label{eq137}
(X \boxplus H) \boxplus (-(H)) = (X \boxplus G) \boxplus -(H). 
\end{equation}

But since:
\[
(X \boxplus H) \boxplus (-(H)) = X \boxplus (H \boxplus (-(H))) = X \boxplus \mathbb{1} = X,
\]
Equation \eqref{eq137} becomes:
\begin{equation}\label{eq138}
X = X \boxplus (G \boxplus (-(H))). 
\end{equation}

By Lemma \ref{lem:identity_element_characterisation}, we know that \eqref{eq138} implies:
\begin{equation}\label{eq139}
G \boxplus (-(H)) = \mathbb{1}(\lfloor p\rfloor). 
\end{equation}

Let \( I, J \in \mathfrak{I}^p(V) \) be such that \( H = \mathbb{1}^I \) and \( G = \mathbb{1}^J \). Then
\[
G \boxplus (-(H))
= \mathbb{1}^J \boxplus \mathbb{1}^{-I}
= \mathbb{1}^{J-I}.
\]
By assumption this equals the identity rough path, hence
\[
\mathbb{1}^{J-I} = \mathbb{1}(\lfloor p\rfloor).
\]
Applying \(\operatorname{dev}\) to both sides and using that \(\operatorname{dev}\circ \mathbb{1}^{(\cdot)}=\mathrm{id}\) on \(\mathfrak{I}^p(V)\) \newline(and \(\operatorname{dev}(\mathbb{1}(\lfloor p\rfloor))=\mathbf{0}(p)\)), we obtain
\[
J-I = \mathbf{0}(p),
\]
and therefore \(J=I\).

Hence, we have that,
\[
\mathscr{S}(\widetilde{G}) = G = \mathbb{1}^J = \mathbb{1}^I = H = \mathscr{S}(\widetilde{H}),
\]

which is what we wanted to show.
\end{proof}

    \newpage
	\printbibliography
\end{document}